\documentclass[a4paper,reqno]{amsart}
\usepackage{amssymb,upref}
\usepackage[mathcal]{euscript}
  \usepackage[all,poly,graph]{xy}
   \usepackage{amsmath,fancybox,tabularx,color}
 \usepackage{wasysym}
\usepackage{tikz} 
\usepackage{graphicx,psfrag}
\usepackage{amssymb,amsthm,mathtools}
\usepackage{auto-pst-pdf}

\usepackage[T1]{fontenc}
\usepackage{amsfonts}
\usepackage{texdraw}

\newcommand{\intprodl}{%
    \mathbin{\scalebox{1.5}{$\lrcorner$}}%
}

\newcommand{\id}{\mathrm{id}}

\newcommand{\diag}{\mathrm{diag}}

\newcommand{\tr}{\mathrm{tr}}

\newcommand{\cB}{\mathcal{B}}
\newcommand{\cC}{\mathcal{C}}
\newcommand{\cD}{\mathcal{D}}

\newcommand{\cL}{\mathcal{L}}
\newcommand{\cM}{\mathcal{M}}

\newcommand{\cO}{\mathcal{O}}
\newcommand{\cQ}{\mathcal{Q}}

\newcommand{\cT}{\mathcal{T}}
\newcommand{\cU}{\mathcal{U}}
\newcommand{\cV}{\mathcal{V}}

\newcommand{\frg}{{\mathfrak g}}
\newcommand{\frh}{{\mathfrak h}}

\newcommand{\ZZ}{\mathbb{Z}}

\newcommand{\QQ}{\mathbb{Q}}
\newcommand{\RR}{\mathbb{R}}
\newcommand{\CC}{\mathbb{C}}
\newcommand{\HH}{\mathbb{H}}
\newcommand{\OO}{\mathbb{O}}
\newcommand{\FF}{\mathbb{F}}

\DeclareMathOperator{\Hom}{\mathrm{Hom}}
\DeclareMathOperator{\End}{\mathrm{End}}

\DeclareMathOperator{\Aut}{\mathrm{Aut}}
\DeclareMathOperator{\Der}{\mathrm{Der}}

\DeclareMathOperator{\Mat}{\mathrm{Mat}}

\newcommand{\ad}{\mathrm{ad}}
\newcommand{\Ad}{\mathrm{Ad}}

\newcommand{\Gl}{\mathfrak{gl}}

\newcommand{\frsl}{{\mathfrak{sl}}}
\newcommand{\frsp}{{\mathfrak{sp}}}
\newcommand{\frso}{{\mathfrak{so}}}

\newcommand{\frgl}{{\mathfrak{gl}}}

\newcommand{\frsu}{{\mathfrak{su}}}

\newcommand{\fru}{{\mathfrak{u}}}

\newcommand{\GL}{\mathrm{GL}}
\newcommand{\SL}{\mathrm{SL}}

\newcommand{\SO}{\mathrm{SO}}

\newcommand{\Cl}{\mathfrak{Cl}}
\newcommand{\spin}{\mathrm{Spin}}

\newtheorem{theorem}{Theorem}
\newtheorem{proposition}[theorem]{Proposition}
\newtheorem{lemma}[theorem]{Lemma}
\newtheorem{corollary}[theorem]{Corollary}

\theoremstyle{definition}
\newtheorem{df}[theorem]{Definition}

\newtheorem{remark}[theorem]{Remark}

\numberwithin{theorem}{section}
\numberwithin{equation}{section}

\begin{document}

\title[Notes on $G_2$]{Notes on $G_2$: \\
the Lie algebra and the Lie group}

\author[C.~Draper]{Cristina Draper Fontanals${}^*$}
\address{Departamento de Matem\'{a}tica Aplicada,   Universidad de M\'{a}laga,
  29071 M\'{a}laga, Spain}
\email{cdf@uma.es}
\thanks{${}^*$ Supported by the Spanish Ministerio de Econom\'{\i}a y Competitividad---Fondo Europeo de
Desarrollo Regional (FEDER) MTM2016-76327-C3-1-P, and by the Junta de Andaluc\'{\i}a grants FQM-336 and FQM-7156, with FEDER funds}



\date{}

\begin{abstract}
These notes have been prepared for the Workshop on "(Non)-existence of
complex structures on $\mathbb{S}^6$",   celebrated in Marburg in March, 2017.
The material is not intended to be original. 
It contains a survey about the smallest of the exceptional Lie groups: $G_2$, its definition and different characterizations as well as its relationship to the spheres $\mathbb{S}^6$ and   $\mathbb{S}^7$. With the exception of the summary of the Killing-Cartan classification, this survey is self-contained, and all the proofs are provided.
Although these proofs are well-known, they are scattered, some of them are difficult to find, and   others require stronger background, while we will usually stick to linear algebra arguments. The approach is algebraical, working at the Lie algebra level most often.  We analyze  the complex Lie algebra  (and group) of type $G_2$ as well as the two real Lie algebras of type $G_2$, the split and the compact one. The octonion algebra will play its role, but it is not the starting point. Also, both the 3-forms approach and  the  spinorial approach are viewed and connected.
Special emphasis is put on relating all the viewpoints by providing precise models.
\end{abstract}

\maketitle

\tableofcontents 

\section{A summary of the Killing-Cartan classification via roots}\label{se:intro}

Killing essentially classified from 1888 to 1890 \cite{killing}  the complex finite-dimen\-sio\-nal simple Lie algebras, as a first step to classify  Lie groups. An important idea for this classification is that, if $\frg$ is a simple Lie algebra over $\CC$, then $\frg$ is isomorphic to the Lie algebra of linear transformations $\{\ad x:x\in\frg\}$, where $\ad x\colon\frg\to\frg$ denotes the adjoint map given by $\ad x(y)=[x,y]$. To study this family it is natural to try to diagonalize the operators $\ad x$. So, a \emph{Cartan subalgebra} is defined as a maximal subalgebra $\frh$ such that $\ad h$ is diagonalizable for all $h\in \frh$. This algebra is necessarily abelian, and, of course, it produces the eigenspace decomposition (as simultaneous diagonalization) $
\frg=\oplus_{\alpha\in\frh^*}\frg_\alpha,
$
where $\frg_\alpha:=\{x\in\frg:\ad h(x)=\alpha(h)x\  \forall h\in\frh\}$. The set $\Phi:=\{0\ne \alpha\in\frh^*: \frg_\alpha\ne0\}$ is called a \emph{root system} and its elements are linear functions called \emph{roots}. As $\frg_0$ coincides with $\frh$, the previous decomposition is then written as
$$
\frg=\frh\oplus\left(\sum_{\alpha\in\Phi}\frg_\alpha\right),
$$
called \emph{root decomposition of $\frg$ relative to $\frh$}.
Some important properties of this decomposition are that $[\frg_\alpha,\frg_\beta]\subset\frg_{\alpha+\beta}$ (clear), and that $\dim\frg_\alpha=1$ for every root $\alpha$ (not so clear!). This gives that the bracket of $\frg$ is quite controlled by $\Phi$. Before providing more details, recall that the simplicity of $\frg$ implies the nondegeneracy of the \emph{Killing form} defined by
$$
\kappa\colon\frg\times\frg\to\CC,\qquad \kappa(x,y)=\tr(\ad x\,\ad y),
$$
which in particular allows to identify $\frh$ with $\frh^*$ and then to define a bilinear symmetric form also in $\frh^*$:
$$(\,,\,)\colon \frh^*\times\frh^*\to\CC, \qquad(\alpha,\beta):=\kappa(t_\alpha,t_\beta),
$$
being $t_\alpha\in\frh$ determined by $\alpha=\kappa(t_\alpha,-)$. It turns out that $(\alpha,\beta)\in\QQ$ for all $\alpha,\beta\in\Phi$, and $(\alpha,\alpha)>0$, so that
$E:=\sum_{\alpha\in\Phi}\RR\alpha$ can be seen as an Euclidean vector space containing $\Phi$ (of real dimension equal to $\dim_\CC\frh$, also called the \emph{rank} of $\frg$) for the inner product extending $(\,,\,)$. Thus, we can work with angles and lengths in $\Phi\subset E$. In fact, $\Phi$ satisfies:
\begin{enumerate}
\item[(R1)] $\Phi$ is a finite subset that spans $E$, and $0\notin\Phi$;
\item[(R2)] If $\alpha\in\Phi$, then $\RR\alpha\cap\Phi=\{\pm\alpha\}$;
\item[(R3)] For any   $\alpha,\beta\in\Phi$, $\langle \beta,\alpha\rangle:=2\frac{(\beta,\alpha)}{(\alpha,\alpha)}\in\ZZ$;
\item[(R4)] If $\alpha\in\Phi$, the reflection $\sigma_\alpha$ on the hyperplane $(\RR\alpha)^{\perp}$ leaves $\Phi$ invariant  ($\sigma_\alpha(\beta)=\beta-\langle \beta,\alpha\rangle \alpha\in\Phi$ for all $\beta\in\Phi$).
\end{enumerate}
To be precise, any subset of an Euclidean space satisfying the properties (R1)-(R4) is what is named a root system. One of the key points in the Killing's paper is that it reduces the difficult algebraical problem of the classification of   complex simple Lie algebras to an easy geometrical problem: to classify root systems. The definition of root system is so restrictive that there exist only a few possibilities for them.
The group generated by the reflections $\{\sigma_\alpha:\alpha\in\Phi\}$ is called the \emph{Weyl group} of the root system, and it is always finite (due to being embedded in the symmetric group of $\Phi$). 

Take $\nu\in E$ such that $(\nu,\alpha)\ne0$ for any $\alpha\in\Phi$. This choice allows to define as \emph{positive} (resp. negative) a root $\alpha$ such that $(\nu,\alpha)>0$ (resp. $(\nu,\alpha)<0$) and so $\Phi=\Phi^+\cup\Phi^-$, with $\Phi^-=-\Phi^+$ (R$2$). A positive root is called \emph{simple} if it is not the sum of two positive roots. The set of simple roots, $\Delta=\{\alpha_1,\dots,\alpha_n\}$, is a basis of the Euclidean space $E$ and it satisfies that any $\alpha\in\Phi^+$ is a linear combination of the elements in $\Delta$ where all the coefficients are nonnegative integers. 

The square matrix $C=\left(\langle \alpha_i,\alpha_j\rangle\right)_{1\le i\le j\le n}$ is called \emph{Cartan matrix}.
The \emph{Dynkin diagram}\footnote{Historical remark: We do not mean that Killing exposed his results in this way. First, Cartan in his thesis \cite{Cartantesis} corrected the proofs and completed them, Weyl simplified the theory in \cite{Weyl} and Dynkin introduced the simple roots and the Dynkin diagrams in \cite{Dynkin}.} of $\Phi$ is the graph which consists of a node for each simple root, and where two different nodes associated to  $\alpha$ and $\beta$ in $\Delta$ are connected by $N_{\alpha,\beta}:= \langle \beta,\alpha\rangle\,\langle \alpha,\beta\rangle$ edges (it turns out that $N_{\alpha,\beta}\in\{0,1,2,3\}$, since it equals $4\cos^2\theta<4$ for $\theta$ the angle formed by $\alpha$ and $\beta$, but it is a nonnegative integer!). Moreover, if   $N_{\alpha,\beta}=2$ or $3$, then $\alpha$ and $\beta$ have different length\footnote{The quotient between the square of the lengths is $\frac{(\beta,\beta)}{(\alpha,\alpha)}=\frac{\langle \beta,\alpha\rangle}{\langle \alpha,\beta\rangle}$, again equal to $1$, $2$ or $3$, and it is not difficult to prove that either all the roots have the same length, or there are just two possible lengths. In such a case, the roots are called -correspondingly- long and short ones.}, and then an arrow is added pointing from the long to the short root.

A root system is said \emph{irreducible} if it is not a disjoint union of two orthogonal root systems, what happens if and only if its Dynkin diagram is connected, which is just the case if we start with a simple Lie algebra (this construction can be realized with a semisimple Lie algebra, that would correspond to nonnecessarily connected Dynkin diagrams). 
Some important examples of simple Lie algebras are the following: 
\begin{enumerate}
\item[\boxed{$A$}] The special Lie algebra $\frsl_{n+1}(\CC)$ of     zero trace matrices of size $ n+1$,  whose Dynkin diagram is called of type $A_n$ if $n\ge1$,
$$
\xymatrix{
\bullet\ar@{-}[]+0;[r]+0^<{\alpha_1}  & \bullet\ar@{-}[]+0;[r]+0^<{\alpha_2} & \bullet\ar@{-}[]+0;[r]^<{\alpha_3} & \;\cdots\;\ar@{-}[r]+0 & \bullet\ar@{-}[]+0;[r]+0^<{\alpha_{n-1}} & \bullet\ar@{}[]+0;[r]^<{\alpha_n} & 
}
$$
\item[\boxed{$B-D$}] The orthogonal Lie algebra $\frso_{m}(\CC)$ of skew-symmetric matrices, whose Dynkin diagram is of type $B_n$ if $m=2n+1$ and $n\ge2$,
\vspace{20pt}
$$\xymatrix{
\bullet\ar@{-}[]+0;[r]+0^<{\alpha_1}  & \bullet\ar@{-}[]+0;[r]+0^<{\alpha_2} & \bullet\ar@{-}[]+0;[r]^<{\alpha_3} & \;\cdots\;\ar@{-}[r]+0 & \bullet\ar@2{->}[]+0;[r]+0^<{\alpha_{n-1}} & \bullet\ar@{}[]+0;[r]^<{\alpha_n} & 
}
$$
and of type $D_n$ if $m=2n$ and $n\ge4$,
$$
\xymatrix{
& & & & & \bullet\ar@{}[r]^<{\alpha_{n-1}} & \\
\bullet\ar@{-}[]+0;[r]+0^<{\alpha_1}  & \bullet\ar@{-}[]+0;[r]+0^<{\alpha_2} & \bullet\ar@{-}[]+0;[r]^<{\alpha_3} & \;\cdots\;\ar@{-}[r]+0 & \bullet\ar@{-}[]+0;[ur]+0^<{\alpha_{n-2}}\ar@{-}[]+0;[dr]+0 & \\
& & & & & \bullet\ar@{}[r]^<{\alpha_n} & 
}$$
\item[\boxed{$C$}] The symplectic Lie algebra 
$$\frsp_{2n}(\CC)=\left\{x\in\frgl_{2n}(\CC):x^t
\left(\begin{array}{cc}0&I_n\\-I_n&0\end{array}\right)+\left(\begin{array}{cc}0&I_n\\-I_n&0\end{array}\right)x=0\right\},$$ 
whose Dynkin diagram is of type $C_n$ if   $n\ge3$,
$$ \xymatrix{
\bullet\ar@{-}[]+0;[r]+0^<{\alpha_1}  & \bullet\ar@{-}[]+0;[r]+0^<{\alpha_2} & \bullet\ar@{-}[]+0;[r]^<{\alpha_3} & \;\cdots\;\ar@{-}[r]+0 & \bullet\ar@2{<-}[]+0;[r]+0^<{\alpha_{n-1}} & \bullet\ar@{}[]+0;[r]^<{\alpha_n} &  
}$$

\end{enumerate}

The Dynkin diagrams of the irreducible root systems are precisely the above ones and
\vspace{10pt}

 \begin{tabular}{ll}

$(E_6)$ & $\xymatrix{
\bullet\ar@{-}[]+0;[r]+0^<{\alpha_1}  & \bullet\ar@{-}[]+0;[r]+0^<{\alpha_2} & \bullet\ar@{-}[]+0;[r]+0^<{\alpha_3}\ar@{-}[]+0;[d]+0^>{\alpha_6}  & \bullet\ar@{-}[]+0;[r]+0^<{\alpha_4} & \bullet\ar@{}[]+0;[r]^<{\alpha_5} & \\
& & \bullet & &
}$
\cr
$(E_7)$ & $\xymatrix{
\bullet\ar@{-}[]+0;[r]+0^<{\alpha_1}  & \bullet\ar@{-}[]+0;[r]+0^<{\alpha_2} & \bullet\ar@{-}[]+0;[r]+0^<{\alpha_3} & \bullet\ar@{-}[]+0;[r]+0^<{\alpha_4}\ar@{-}[]+0;[d]+0^>{\alpha_7}  & \bullet\ar@{-}[]+0;[r]+0^<{\alpha_5} & \bullet\ar@{}[]+0;[r]^<{\alpha_6} & \\
& & & \bullet & &
}$
\cr
$(E_8)$ & $\xymatrix{
\bullet\ar@{-}[]+0;[r]+0^<{\alpha_1}  & \bullet\ar@{-}[]+0;[r]+0^<{\alpha_2} & \bullet\ar@{-}[]+0;[r]+0^<{\alpha_3} & \bullet\ar@{-}[]+0;[r]+0^<{\alpha_4} & \bullet\ar@{-}[]+0;[r]+0^<{\alpha_5}\ar@{-}[]+0;[d]+0^>{\alpha_8}  & \bullet\ar@{-}[]+0;[r]+0^<{\alpha_6} & \bullet\ar@{}[]+0;[r]^<{\alpha_7} & \\
& & & & \bullet & &
}$
\cr
$(F_4)$ & $\xymatrix{
\bullet\ar@{-}[]+0;[r]+0^<{\alpha_1}  & \bullet\ar@2{<-}[]+0;[r]+0^<{\alpha_2} & \bullet\ar@{-}[]+0;[r]+0^<{\alpha_3}  & \bullet\ar@{}[]+0;[r]^<{\alpha_4} & 
}$  \vspace{10pt}
\cr 
$(G_2)$ & $\xymatrix{
\bullet\ar@3{<-}[]+0;[r]+0^<{\alpha_1}  & \bullet\ar@2{}[]+0;[r]^<{\alpha_2} & 
}$
\end{tabular}\vspace{10pt}

Conversely,
each Dynkin diagram determines a unique root system up to isomorphism.
The key is to recover first the set of positive roots by working on induction on the height: $\textrm{ht}(\sum_in_i\alpha_i):=\sum_in_i$, and second, to recover the inner product (fixing arbitrarily the length of a short node). We only have  to take into account that:
\begin{itemize}
\item[\textbf{Fact 1}.] If $\alpha\in\Phi^+\setminus\Delta$, there is (at least one index) $i=1,\dots,n$ such that $\alpha-\alpha_i$ is a root;
\item[\textbf{Fact 2}.]
If  $\alpha,\beta\in\Phi$,
\begin{equation}\label{eq_strings}
\{\beta+m\alpha:m\in\RR\}\cap\Phi=\{\beta+m\alpha:m\in\ZZ,-r\le m\le q\}
\end{equation}
for $r$ and $q$  positive integers such that $\langle \beta,\alpha\rangle=r-q$  (if $\alpha,\beta\in\Delta$, necessarily such $r=0$, since $\alpha-\beta$ is not a root).
\end{itemize}

\noindent For instance, from the Dynkin diagram of type $G_2$, a set of simple roots is given by $\Delta=\{\alpha_1,\alpha_2\}$ and  
$\langle \alpha_1,\alpha_2\rangle\langle \alpha_2,\alpha_1\rangle=3 $ (there is a triple edge joining the two nodes). Hence, the Cartan matrix is 
\begin{equation}\label{eq_CartanmatrixG2}
C=\left(\begin{array}{cc}2&-1\\-3&2\end{array}\right)
\end{equation}

\noindent since  
$\langle \alpha_i,\alpha_j\rangle\in\ZZ_{<0}$ if $i\ne j$ and
$\alpha_1$ is a short root.
Thus,   the proportion between the lengths is  $\sqrt{3}$, and, as the cosine of the angle between $\alpha_1$ and $\alpha_2$ will be $-\sqrt{3}/2$, then the angle is 120$^\circ$. 
As $\langle \alpha_2,\alpha_1\rangle=-3 $, then $\alpha_1+\alpha_2$, $2\alpha_1+\alpha_2$ and $3\alpha_1+\alpha_2$ are the only roots of the form $i\alpha_1+\alpha_2$ with $i\ne0$. While, as $\langle \alpha_1,\alpha_2\rangle=-1 $, then $\alpha_1+2\alpha_2$ is not a root, so that $2\alpha_1+\alpha_2$ is the only root of height 3 and $3\alpha_1+\alpha_2$ the only one of height 4 ($2\alpha_1+2\alpha_2$ is not a root because it is the double of a root, see R$2$). Now, $\langle3\alpha_1+\alpha_2,\alpha_2\rangle=-3+2=-1$, so that $(3\alpha_1+\alpha_2)+\alpha_2$ is the only root of height 5, and of course there is not any more, since both $(3\alpha_1+2\alpha_2)+\alpha_i$ ($i=1,2$) are multiple of roots. That is, we have just obtained the root system in Figure~\ref{gdos}:
$$
\Phi_{G_2}=\pm\{\alpha_1,\alpha_2,\alpha_1+\alpha_2,2\alpha_1+\alpha_2,3\alpha_1+\alpha_2,3\alpha_1+2\alpha_2\}.
$$

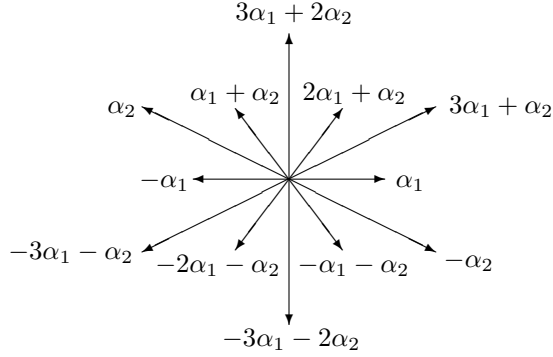
\begin{figure}
\begin{center}
\begin{picture}(0,80)
 \put(0,0){\vector(1,0){36}}
 \put(0,0){\vector(3,4){20}}
 \put(0,0){\vector(-3,4){20}}
\put(0,0){\vector(-1,0){36}}
  \put(0,0){\vector(3,-4){20}}
 \put(0,0){\vector(-3,-4){20}}
 \put(-37,30){$\alpha_1+\alpha_2$} \put(0,-35){$\,\, {-\alpha_1-\alpha_2}$}
 \put(40,-3){$\alpha_1$}\put(-56,-3){$-\alpha_1$}
 \put(5,30){$2\alpha_1+\alpha_2$}\put(-50,-35){$-2\alpha_1-\alpha_2$}
 \put(0,0){\vector(0,1){55}}
 \put(-20,60){$3\alpha_1+2\alpha_2$}
 \put(0,0){\vector(0,-1){55}}
 \put(-25,-63){$-3\alpha_1-2\alpha_2$}
 \put(0,0){\vector(2,1){55}}
 \put(60,25){$3\alpha_1+\alpha_2$}
 \put(0,0){\vector(-2,-1){55}}
 \put(-105,-30){$-3\alpha_1-\alpha_2$}
 \put(0,0){\vector(2,-1){55}}
 \put(58,-33){$-\alpha_2$}
 \put(0,0){\vector(-2,1){55}}
 \put(-68,25){$\alpha_2$}
\end{picture}\vskip 2.4cm
\caption{$\Phi_{G_2}=
\pm\{\alpha_1, \alpha_2, \alpha_1+\alpha_2, 2\alpha_1+\alpha_2,
3\alpha_1+\alpha_2,  3\alpha_1+2\alpha_2\}$.
\label{gdos}
}
\end{center}
\end{figure}

Now, the root system $\Phi$ determines the product of $\frg$. Before recalling the argument, note that the possibility of constructing $\frg$ (a basis joint with  the product of any pair of basic elements) does not prove the existence of a simple Lie algebra with such root system, but the uniqueness. In order to prove the existence,  one has   either to check that the constructed algebra satisfies the Jacobi identity (a straightforward but long and tedious task) or to find some concrete simple Lie algebra whose related root system is isomorphic to $\Phi $, that is, to find a model of $\frg$. Soon we will follow the second  way for proving the existence of a complex Lie algebra with root system of type $G_2$, but we previously want to emphasize that all the information of the Lie algebra is encoded in the Dynkin diagram, and with a bit of patience one could recover all the structure constants of the Lie algebra from the picture.

The idea of the proof is the following. For each $\alpha\in\Phi^+$,
we choose $h_\alpha=\frac{2t_\alpha}{(\alpha,\alpha)}\in\frh$ and some $x_\alpha\in\frg_\alpha$,
$x_{-\alpha}\in\frg_{-\alpha}$ such that $[x_\alpha,x_{-\alpha}]=h_\alpha$. In other words, the map
$$
x_\alpha\mapsto  \left(\begin{array}{cc}0&1\\0&0\end{array}\right),\quad
x_{-\alpha}\mapsto  \left(\begin{array}{cc}0&0\\1&0\end{array}\right),\quad
h_\alpha\mapsto  \left(\begin{array}{cc}1&0\\0&-1\end{array}\right),\quad
$$
provides an isomorphism between $S_\alpha:=\frg_{\alpha}\oplus\frg_{-\alpha}\oplus[\frg_{\alpha},\frg_{-\alpha}]$ and $\frsl_2(\CC)$.  Then 
\begin{equation}\label{eq_baseChev}
\{h_{\alpha_1},\dots,h_{\alpha_n}\}\cup\{x_\alpha:\alpha\in\Phi\}
\end{equation}
is a basis of $\frg$. For any $\alpha,\beta\in\Phi$ such that $\beta\ne\pm\alpha$, we have
$[x_{\alpha},x_\beta]=c_{\alpha,\beta}x_{\alpha+\beta}$
for some $c_{\alpha,\beta}\in\CC$. It is possible to scale $x_\alpha$ in such a way that, for all $\alpha,\beta\in\Phi$,
then $c_{-\alpha,-\beta}=-c_{\alpha,\beta}$.  Then $c_{\alpha,\beta}$ equals $0$ if $\alpha+\beta\notin\Phi$ and otherwise  equals\footnote{Of course, it is possible to be more precise about the determination of the sign here. But for our purpose, it is sufficient to say that, for \lq\lq some\rq\rq\  pairs of roots, the sign can be chosen arbitrarily and that determines the signs for all remaining pairs of roots. Also, most of the textbooks about the topic contain existence theorems guaranteeing  the existence of a complex semisimple Lie algebra whose root system has a prefixed Cartan matrix. The uniqueness is consequence of the Serre relations.}  $\pm k$ 
for $k$ the least positive integer for which $\alpha-k\beta\notin\Phi$.
In particular, all the structure constants for the chosen basis are in $\ZZ$. The basis is called a \emph{Chevalley basis}.

\smallskip

The proof is based on   representation theory of $\frsl_2(\CC)$: any irreducible representation of dimension $r+1$ works  in the same way as the action of
$$
h=X\frac{\partial}{\partial X}-Y\frac{\partial}{\partial Y},\quad
x=X\frac{\partial}{\partial Y},\quad y=Y\frac{\partial}{\partial X}
$$
on the homogeneous $r$th degree  polynomials in two variables $X$ and $Y$ (each monomial is connected with any other). This is applied to $S_\alpha\cong\frsl_2(\CC)$ acting on the vector space $\oplus_{i\in\ZZ} \frg_{\beta+i\alpha}$ to get a \lq\lq chain without holes\rq\rq, which allows to prove Fact~2.
\smallskip

When Killing discovered the five exceptional complex Lie algebras, mathematicians were not expecting to find them. It was a surprise, and not only for Killing. He was trying to prove that $\frsl_n(\CC)$, $\frso_n(\CC)$ and $\frsp_n(\CC)$ are the only simple finite-dimensional complex Lie algebras (at the beginning he was not even aware of $\frsp_n(\CC)$).  

This Killing paper  has been called \emph{the Greatest Mathematical Paper of all Time} in \cite{Coleman}, although Killing was not very proud of his work. In part the paper contains some mistakes (corrected in \cite{Cartantesis}), but probably the reason for his dissatisfaction  is that he wanted to find all the \textbf{real} Lie algebras, since he was interested in Geometry. The same idea may well be upheld by many potential readers, but we cannot forget the close relationship between complex and real Lie algebras. The classification of the real ones goes through the classification of the complex case: let us take into account that one of the main described tools has been the diagonalization of the semisimple operators, which is not always possible in the real case. Such (real) classification  was realized by Cartan in \cite{Cartanclasificareales}, after studying and completing Killing's work. \smallskip

\subsection{On simple real Lie algebras}\label{se_reales}

We will only mention a pair of facts about simple real Lie algebras. If $L$ is a simple real Lie algebra,
\begin{itemize}
\item Either $L$ is just a complex simple Lie algebra, but considered as a real Lie algebra;
\item Or $L^\CC=L\otimes_\RR\CC=L\oplus\mathbf{i}L$ is a complex simple Lie algebra. In this case $L$ is said a \emph{real form} of $L^\CC$. If $L^\CC$ has a root system $X$, then $L$ is called a real algebra of type $X$.
There are 17 real simple Lie algebras of exceptional types, two of them of type $G_2$, which will be our main goal of study today. 
\end{itemize}

We mentioned above that for any $X\in\{A_l,B_l,C_l,D_l,E_6,E_7,E_8,F_4,G_2\}$, there is a unique simple complex Lie algebra $\frg $ with related root system $X$. Now we will explain why at least there exist two real forms of $\frg$, that is, two real Lie algebras of type $X$, for each $X$.
If we take a Chevalley basis of $\frg$ as in Eq.~\eqref{eq_baseChev}, then
$$
L_s=\sum_{j=1}^n\RR h_{\alpha_j}\oplus\big(\sum_{\alpha\in\Phi}\RR x_\alpha\big)
$$
is a real Lie algebra such that $L_s^\CC=\frg$. This is clear by recalling that the structure constants of the Chevalley basis belong to $\ZZ$.
This real form is called \emph{split} due to the fact of possessing a root decomposition similar to that one in the complex case ($\ad h$ is diagonalizable for any $h\in \frh_0:=\sum_{j=1}^n\RR h_{\alpha_j}$ -with real eigenvalues-). It turns out that the signature of the Killing form $\kappa_s$ of $L_s$ coincides with $n$, the rank of $\frg$, because 
$\kappa_s\vert_{\frh_0}$ is positive definite and the decomposition 
$\frh_0\oplus\,\sum_{\alpha\in\Phi^+}\big(\RR x_\alpha\oplus\RR x_{-\alpha}\big)$ is an orthogonal decomposition where the hyperbolic planes $\RR x_\alpha\oplus\RR x_{-\alpha}$ do not contribute to the signature.

On the other hand,
$$
L_c=\sum_{j=1}^n\RR \mathbf{i}h_{\alpha_j}\oplus\sum_{\alpha\in\Phi^+}\big(\RR (x_\alpha-x_{-\alpha})\oplus\RR\mathbf{i}(x_{\alpha}+x_{-\alpha})\big)
$$
also satisfies $[L_c,L_c]\subset L_c$ and hence it is again a real form of $\frg$. It is not difficult to check that its Killing form (obtained by restriction of the Killing form of $\frg$) is negative definite. In particular $L_c$ cannot be isomorphic to $L_s$. This real form is called \emph{compact}, due to the following fact: 
being the Killing form negative definite implies that the related Lie group is compact by Myers' Theorem \cite[10.24]{ONeill} (any connected Lie group with negative definite Killing form is compact, according to \cite[11.11]{ONeill}). (Furthermore, any other real form with negative definite Killing form is isomorphic to $L_c$.)  
\smallskip

If $L$ is another  real form of $\frg$, then $L$ admits a $\ZZ_2$-grading $L=L_{\bar0}\oplus L_{\bar1}$, called \emph{Cartan decomposition} of $L$, such that $L_{\bar0}\oplus \mathbf{i}L_{\bar1}$ is a compact real form of $\frg$. From this, it is possible to conclude that there is a one-to-one correspondence between isomorphism classes of real forms of $\frg$ and conjugacy classes of involutive automorphisms of $\frg$. (Note that any $\ZZ_2$-grading on $\frg$ is produced -as the eigenspace decomposition- by an order two automorphism of $\frg$ and vice versa.)

 In the $G_2$-case, there is only one conjugacy class of order two automorphisms, since there is only one symmetric space quotient of the group $G_2$, namely, $G_2/\SO(4)$. This implies that there are only two real forms of the complex Lie algebra of type $G_2$, an interesting fact which will study in detail in Section~\ref{se_solodos}.


\section{A first linear model of (split) $G_2$}\label{se:modelog2}

\subsection{The model} 
Let $\cU=\FF^3$ be the column vectors  for $\FF\in\{\RR,\CC\}$, and let $\times$ denote the usual cross product on $\cU$. For any $x\in \cU$, let $l_x$ denote the coordinate matrix, in the canonical basis $\{e_i\}_{i=1,2,3}$ of $\FF^3$, of the map $y\mapsto x\times y$. Hence, for
\begin{equation}\label{loslx}
x=\left(\begin{array}{l}x_1\\x_2\\x_3\end{array}\right)\quad\Rightarrow \quad
l_x=\left(\begin{array}{ccc}0&-x_3&x_2\\x_3&0&-x_1\\-x_2&x_1&0\end{array}\right).
\end{equation}
\begin{proposition}\label{pr_defdeL}
$$
L=\left\{\left(\begin{array}{ccc}0&-2y^t&-2x^t\\x &a&l_y\\y&l_x&-a^t\end{array}\right): a\in\frsl_3(\FF),x,y\in \FF^3\right\}
$$
is a Lie subalgebra of $\frgl_7(\FF)=(\Mat_7(\FF),[ \,,\,])$ of dimension 14.
\end{proposition}


\begin{proof}
The dimension of $L$ as a vector space is clear. We only have to see that $[A,B]=AB-BA\in L$ for any $A,B\in L$.
Let us denote $M_{(a,x,y)}= \left(\begin{array}{ccc}0&-2y^t&-2x^t\\x &a&l_y\\y&l_x&-a^t\end{array}\right)$.
We compute, for $a,b\in\frsl_3(\FF)$, and $x,y,u,v\in\cU$:  

\begin{itemize} 
\item[i)] $[M_{(a,0,0)},M_{(b,0,0)}]=M_{([a,b],0,0)}$;

Because $[a^t,b^t]=-[a,b]^t$ and $\tr([a,b])=0$.
\item[ii)] $[M_{(a,0,0)},M_{(0,u,0)}]=M_{(0,au,0)}$;

Because $l_{au}=-(l_ua+a^tl_u)$.
\item[iii)] $[M_{(a,0,0)},M_{( 0,0,y)}]=M_{( 0,0,-a^ty)}$; (the same argument as in ii).
\item[iv)] $[M_{(0,x,0)},M_{(0,u,0)}]=M_{(0,0,2x\times u)}$; 

Because obviously $l_xu-l_ux=2x\times u$, and not so obviously $l_{x\times  u}=-xu^t+ux^t$. This is a direct consequence of the identity of the cross product $(x\times  y)\times  z=(x\cdot z)y-(y\cdot z)x$ (also called  triple product expansion), since, relative to the canonical basis, the matrix of $(x\cdot -)y$ (it sends $e_j$ to $x_jy$) is $yx^t$ ($y_ix_j$ in the $(i,j)$th place). 
\item[v)] $[M_{(0,x,0)},M_{(0, 0,v)}]=M_{(\textrm{pr}_{\frsl_3(\FF)}(-3xv^t),0,0)}$;

When doing the bracket, we get $\left(\begin{array}{ccc}-2x^tv+2v^tx&0&0\\0 &-2xv^t-l_vl_x&0\\0&0&2vx^t+l_xl_v\end{array}\right)$.
There is $0$ in the $(1,1)$th position since $x^tv=x\cdot v=v^tx$. The matrix in position $(3,3)$ is the opposite-transpose of that one in $(2,2)$, since $ l_x$ and $l_v$ are skew-symmetric matrices. 
  Finally, observe that $$l_xl_v=(x_jv_i)_{i,j}-(\sum_k x_kv_k)I_3=vx^t-(v^tx)I_3,$$ so that $l_xl_v+2vx^t=3vx^t-\frac{\tr(3vx^t)}{\tr(I_3)}I_3=\textrm{pr}_{\frsl_3(\FF)}(3vx^t)$. (It is obvious that $\tr(vx^t)=v^tx$.)

\item[vi)] $[M_{(0,0,y)},M_{(0, 0,v)}]=M_{(0,2y\times v,0)}$; (analogous to iv).

\end{itemize}
\end{proof}

As before $\{e_i\}_{i=1}^3$ denotes  the canonical basis of $\FF^3$, and now $E_{ij}\in\Mat_{3}(\FF)$ denotes the \emph{unit} square matrix whose entries are all of them zero except for the $(i,j)$th, which is $1$. Take $\frh=\{ M_{(a,0,0)}: a\in\frsl_3(\FF)\ \textrm{diagonal matrix} \}$, which is an abelian subalgebra of $L$.
Define
$
\varepsilon_i\colon \frh\to\FF$, for $i=1,2,3$, as
$$
\varepsilon_i\left(h=\diag\{0,s_1,s_2,s_3,-s_1,-s_2,-s_3\}\right)=s_i\qquad (\sum_{i=1}^3s_i=0).
$$
Of course $\{\varepsilon_1,\varepsilon_2\}$ is a basis of $\frh^*$, since $\sum_{i=1}^3\varepsilon_i=0$.
We are computing now a basis of $L$ of eigenvectors for $\ad\, h$.
By using i), ii) and iii) above, 
$$\begin{array}{l}
{[}h,M_{(E_{ij},0,0)}]=M_{([\sum_k s_kE_{kk},E_{ij}],0,0)}=(s_i-s_j)M_{(E_{ij},0,0)}=(\varepsilon_i-\varepsilon_j)(h)M_{(E_{ij},0,0)},\\
{[}h,M_{(0,e_i,0)}]=M_{(0,\sum_k s_kE_{kk}e_i,0)}=M_{(0,s_ie_i,0)}=\varepsilon_i(h)M_{(0,e_i,0)},
\\
{[}h,M_{(0,0,e_i)}]=M_{(0,0,-(\sum_k s_kE_{kk})^te_i)}=M_{(0,0,-s_ie_i)}=-\varepsilon_i(h)M_{(0,e_i,0)};
\end{array}
$$
and hence the Lie algebra $L$ can be decomposed as
$$
L=\frh\oplus\left(\sum_{\alpha\in\Phi}L_\alpha\right),
$$
being $\Phi= \{\varepsilon_i-\varepsilon_j:i\ne j\}\cup\{\pm\varepsilon_i:i=1,2,3\}\subset\frh^*$, for
$$\begin{array}{ll}
L_{\varepsilon_i-\varepsilon_j}=\FF M_{(E_{ij},0,0)},&(i\ne j)\\
L_{\varepsilon_i}=\FF M_{(0,e_i,0)},&\\
L_{-\varepsilon_i}=\FF M_{(0,0,e_i)},&  (i=1,2,3).
\end{array}
$$
Thus, $L$ has a root decomposition (even for $\FF=\RR$). Does this mean that $L$ is simple? Yes, it does:
\begin{proposition}
L is a simple Lie algebra of type $G_2$.
\end{proposition}

\begin{proof}
Take $0\ne I$ an ideal of $L$. As $[\frh,I]\subset I$, then $I=I\cap \frh\oplus (\oplus_{\alpha\in\Phi}I\cap L_\alpha)$.
Consider $S=\{ M_{(a,0,0)}: a\in\frsl_3(\FF)\}$, which is a Lie subalgebra of $L$ isomorphic to $\frsl_3(\FF)$.
In case $I\cap S\ne0$, by the simplicity of $S$ we have $I\cap S=S$, so that $\frh\subset S\subset I$ and hence
$L=\frh+[\frh,L]\subset I+[I,L]\subset I$ and $I=L$.
The other possibility, $I\cap S=0$, does not occur, since in such a case there is some $\alpha\in\pm\{\varepsilon_i:i=1,2,3\}$ such that $I\cap L_\alpha\ne0$. As $L_\alpha$ has dimension 1, then $L_\alpha\subset I$ and so
$[L_\alpha,L_{-\alpha}]\subset [I,L]\subset I$. This is a contradiction  taking into account that\footnote{A well known property of the root decompositions of simple Lie algebras is $0\ne [L_\alpha,L_{-\alpha}]\subset\frh$. In this case, as
$e_ie_i^t=E_{ii}$, its projection on $\frsl_3(\FF)$ is $E_{ii}-\frac13 I_3$, so 
$[L_{\varepsilon_i},L_{-\varepsilon_i}]\ni[M_{(0,e_i,0)},M_{(0,0,e_i)}]=M_{(-2E_{ii}+E_{i+1,i+1}+E_{i+2,i+2},0,0)}\ne0$.
}
 $0\ne [L_\alpha,L_{-\alpha}]\subset\frh\subset S$, but $I$ does not contain any nonzero element of $S$. This finishes the simplicity.

Now we call $\alpha_1:=\varepsilon_2$ and $\alpha_2:=\varepsilon_1-\varepsilon_2$. Hence
$$
\begin{array}{ll}
\alpha_1+\alpha_2=\varepsilon_1,\qquad&
2\alpha_1+\alpha_2=\varepsilon_1+\varepsilon_2=-\varepsilon_3,\\
3\alpha_1+\alpha_2=\varepsilon_2-\varepsilon_3,\qquad\ &
3\alpha_1+2\alpha_2=\varepsilon_1-\varepsilon_3;
\end{array}
$$
so that 
$
\Phi=\pm\{\alpha_1,\alpha_2,\alpha_1+\alpha_2,2\alpha_1+\alpha_2,3\alpha_1+\alpha_2,3\alpha_1+2\alpha_2\}
$
is the root system of type $G_2$ described in Section~\ref{se:intro}.
\end{proof}

The above propositions are valid for both $\CC$ and $\RR$: if we denote $L_\FF$ for distinguishing the considered field, 
$L_\RR$ is a real form of $L_\CC$. 
From now on, we shall denote the
simple complex Lie algebra $L_\CC$ by $\frg_2$, and the real form $L_\RR$ by $\frg_{2,2}$. The second number in the notation makes reference to the signature of the Killing form, which is $2$, taking into account that $L_\RR$ is the split real form of $L_\CC$ since it admits a root decomposition, or alternatively  $\kappa\vert_{L_{\alpha}\oplus L_{-\alpha}}\equiv \left(\begin{array}{cc}0&1\\1&0\end{array}\right)$.
\medskip

I would like to remark that finding a \emph{model} (a description of the algebra independent from the roots and from the basis corresponding to such roots) has been an easy task for us now, but surprisingly, it took many years. A very suggestive paper about the exceptional group $G_2$, which includes many references and curiosities, is \cite{AgriG2}. There, the references about the first results on $G_2$ can be consulted. These results were not immediate after Killing's discovery, probably because the whole Lie theory was being developed. Both Élie Cartan and Friedrich Engel obtained, independently in 1893, that the Lie algebra of vector fields on $\CC^5$ whose local flows preserve a plane distribution given by certain Pfaffian system, is the Lie algebra of type
$\frg_2$ (that is, $\frg_2$ appears as the infinitesimal symmetries of this distribution). Both authors gave a second (this time different) geometrical realization of $\frg_2$; but an algebraical realization had to wait until 1900 \cite{EngelG2}, when Engel 
finally gave a  realization
including the missing compact form
 (which is besides a global version, for groups not only for algebras). His description based on 3-forms on $\RR^7$ will be carefully developed in Section~\ref{genericas}. The most popular description of $\frg_2$ as the derivation algebra of the octonions (both split and division octonions provide the split and the compact real forms, respectively) is due to Cartan
\cite{Cartanclasificareales} (in 1914, although already mentioned in a letter in 1908). We will arrive at it in Section~\ref{se:octo}. 

\begin{remark}\label{re_grad}
As a consequence of the computations made in the proof of Proposition~\ref{pr_defdeL}, if we define $\cL=\cL_{\bar 0}\oplus \cL_{\bar 1}\oplus \cL_{\bar 2}$
by
$$
\cL_{\bar 0}=\frsl_3(\FF),\quad \cL_{\bar 1}=\cU=\FF^3,\quad \cL_{\bar 2}=\cU^*,
$$
with the product given by $[\cL_{\bar i},\cL_{\bar j}]\subset \cL_{\bar i+\bar j}$ such that $\cL_{\bar 0}$ is subalgebra and:
\begin{itemize}
\item The actions of $\cL_{\bar 0}$ on $\cL_{\bar 1}$ and $\cL_{\bar 2}$ are the natural  one and the dual one, respectively;
\item $[x,y]=(2x\times y)^t$ and $[x^t,y^t]=2x\times y$;
\item $[x,y^t]=-3\textrm{pr}_{\frsl_3(\FF)}(xy^t)  $;
\end{itemize}
then $\cL$ is a Lie algebra of type $G_2$ (just $\frg_2$ if $\FF=\CC$ and $\frg_{2,2}$  if $\FF=\RR$). 
We are identifying the dual of $\FF^3$ with the row vectors $\{v^t:v\in\FF^3\}$,  thanks to the map $\cU\times \cU^*\to \FF$ given by $(x,u^t)\mapsto u^tx\in\FF$. \smallskip

A model equivalent to this one can be consulted, for instance, in \cite[\S22.4]{FultonHarris}, where some linear models with the same philosophy are exhibited for other exceptional Lie algebras, in particular for type $E_8$ as $\frsl_9(\FF)\oplus \cV\oplus \cV^*$ with $\cV=\wedge^3\FF^9$. 
The underlying fact is that, in the complex case,  any reductive subalgebra of maximal rank of a simple Lie algebra (in our case $\frg_2$, the subalgebra is $\frsl_3(\CC)$)
induces a grading on this algebra by means of an abelian group (in our case,  a $\ZZ_3$-grading), in such a way that the nontrivial components
of the grading are irreducible modules (in our case, the natural one and the dual of the natural one).   
This    philosophy is made explicit in \cite[Theorem~1]{modelosf4}, and applied  for describing some linear models of the complex -consequently, also the real split- exceptional Lie algebra of type $F_4$.
 This kind of models based in linear algebra turn out to be very comfortable to work with. 
 
 For models on the compact algebra of type $G_2$, one has to bear in mind that compact algebras are never graded by any group different from $\ZZ_2^r$. In spite of that,  one can find in
 \cite{modelosg2} some linear models of the compact real algebra of type $G_2$ (in a more general context, with few restrictions on the underground field), whose  underlying common feature is the existence of a \lq\lq nice\rq\rq\  
 (simple or semisimple) subalgebra  $\frh$ such that  $(\frg_2,\frh)$ is a reductive pair 
 -related to reductive homogeneous spaces-. We will exhibit such a compact model in Proposition~\ref{pr_modelocompacto} as a consequence of the relationship between $G_2$ and the six-dimensional sphere.
 
Other well known works are  \cite{Vinberg} on some unified models,    \cite{Kantor} on models based in $\ZZ$-gradings,
and specially 
\cite{ Tits}, which  constructs all the exceptional simple Lie algebras in a uniform way by using alternative and Jordan algebras, by a procedure which today is called \emph{Tits construction}.
\end{remark}

\subsection{A little bit of representation theory}\label{se:mod}

We consider the ground field $\FF$ to be either $\CC$ or $\RR$. Let $V=\FF^7$ throughout this and the next sections.

First we can observe that $V$ is an irreducible representation of $L_\FF$, with the natural action by matrix multiplication. Write the elements in $V$ in blocks of sizes $1+3+3$. Indeed, it is trivial that
$$M_{(a,x,y)}\left(\begin{array}{c}
s\\u\\v
\end{array}\right)=\left(\begin{array}{c}
-2y^tu-2x^tv\\sx+au+y\times v\\sy+x\times u-a^tv
\end{array}\right),
$$
if $s\in\FF$, $u,v\in\cU=\FF^3$.
In particular, for $h=\diag\{0,s_1,s_2,s_3,-s_1,-s_2,-s_3\}\in\frh$, 
$$
h\left(\begin{array}{c}
1\\0\\0
\end{array}\right)=\left(\begin{array}{c}
0\\0\\0
\end{array}\right),\quad
h\left(\begin{array}{c}
0\\e_i\\0
\end{array}\right)=s_i\left(\begin{array}{c}
0\\e_i\\0
\end{array}\right),\quad
h\left(\begin{array}{c}
0\\0\\e_i
\end{array}\right)=-s_i\left(\begin{array}{c}
0\\0\\e_i
\end{array}\right),
$$
so that $V=\FF^7$ decomposes as $V=\oplus_{\mu\in\frh^*}V_\mu$, for 
$$
V_0=\FF\left(\begin{array}{c}
1\\0\\0
\end{array}\right),\quad V_{\varepsilon_i}=\FF\left(\begin{array}{c}
0\\e_i\\0
\end{array}\right),\quad V_{-\varepsilon_i}=\FF\left(\begin{array}{c}
0\\0\\e_i
\end{array}\right),\quad i=1,2,3.
$$
The set $\Lambda=\{\mu\in\frh^*:V_\mu\ne0\}$ is called the set of \emph{weights},
for $V_\mu:=\{v\in V:hv=\mu(h)v\ \forall h\in\frh\}$. Now, any nonzero $L_\FF$-submodule $W\le V$ is homogeneous and  $V_\mu\cap W\ne0$ for some $\mu\in\frh^*$. Then $V_\mu\subset W$, and, by acting the root spaces $L_\alpha$ we get the whole $V\subset W$, and consequently the irreducibility of $V$.\smallskip

This behavior is not particular for $L_\FF$ and $V$, but any finite-dimensional representation $W$ of a complex semisimple Lie algebra $\frg$ decomposes as a direct sum of \emph{weight spaces} $W=\oplus_{\mu\in\Lambda}W_\mu$, which clearly satisfy $\frg_\alpha W_\mu\subset W_{\alpha+\mu}$ for all $\alpha\in\Phi$ and $\mu\in\Lambda$. If the module is irreducible, then it is generated by any nonzero element in $  W_\mu$ for a weight $\mu$ such that $\alpha+\mu\notin\Lambda$ for all $\alpha\in\Phi^+$. This weight is called \emph{maximal} and the module $W$ is usually denoted by $V(\mu)$.
In our example the weights of $V$ are just the short roots in $\Phi$ and $\{0\}$, so that the  maximal weight of $V$ is $-\varepsilon_3=2\alpha_1+\alpha_2=\omega_1$ (the short maximal root), where $\omega_i\in\frh^*$ is (always) defined by $\langle \omega_i,\alpha_j\rangle=\delta_{ij}$. (This is   possible for any semisimple  $\frg$, since Cartan matrix is invertible.)
The possible abstract weights   $\mu\in\frh^*$ such that  $\mu$ is the maximal weight of some irreducible module are precisely $\sum_{i=1}^n{s_i}\omega_i$, with $s_i\in\ZZ_{\ge0}$, which are called \emph{dominant weights} ($\Lambda^+$). The representations $V(\omega_i)$ are called the \emph{fundamental representations}, which in our $\frg_2$-case are $V(\omega_1)\cong\CC^7$ and $V(\omega_2)\cong \frg_2$, the natural and the adjoint representation respectively ($3\alpha_1+2\alpha_2=\omega_2$). 
This implies that the nontrivial $\frg_2$-representation of the least possible dimension is just $\CC^7$, so that, if one wants to see the complex Lie algebra of type $G_2$ as a subalgebra of $\frgl_m(\CC)$ for some $m$, then $m\ge7$, what explains in part  our election of model in the previous subsection.\smallskip

All this tells us how   representation theory works  for the complex Lie algebra $\frg_2=L_\CC$. Any finite-dimensional representation is completely reducible and hence direct sum of irreducible modules. Each summand is (isomorphic to) $V(s_1\omega_1+s_2\omega_2)$ for some $s_1,s_2\in\ZZ_{\ge0}$, so that it lives in $V(\omega_1)^{\otimes s_1}\otimes V(\omega_2)^{\otimes s_2}$,
that is, it is a submodule of a tensor product of copies of the adjoint module $\frg_2$ and of the natural module $\CC^7$.
\smallskip

In any case, as $L_\FF$ has dimension $14$ but $\dim\frgl_7(\FF)=49$ (a long \emph{distance}), there is still work to do until understanding what is  $L_\FF$ preserving, or what is characterizing $L_\FF$.

\medskip


\section{Understanding $G_2$  better}
\subsection{$G_2$ and the cross product in $\FF^7$}\label{se:cross}

Let $n\colon V\to \FF$ be the quadratic form whose related matrix in the canonical basis $\{e_i\}_{i=1,\dots,7}$ of $V=\FF^7$ is 
\begin{equation}\label{eq_lan}
N=\left(\begin{array}{ccc}
-1&0&0\\0&0&-2I_3\\0&-2I_3&0
\end{array}\right)\Rightarrow n\left(\begin{array}{c}
s\\u\\v
\end{array}\right)=-s^2-2u^tv-2v^tu=-s^2-4u\cdot v,
\end{equation}  
where $\cdot$ denotes the usual scalar product if $\FF=\RR$ and the usual inner product if $\FF=\CC$.
Observe that in the real case the signature of this norm $n$ is $(4,3)$, since in some orthogonal basis the matrix is $\diag\{-I_4,I_3\}$.

\begin{lemma}\label{le_antisim}
The algebra $L_\FF$ defined in Proposition~\ref{pr_defdeL} satisfies
$$
L_\FF\subset\frso(V,n)=\{f\in\frgl(V):n(f(X),Y)+n(X,f(Y))=0\ \forall X,Y\in V\}.
$$
\end{lemma}

\begin{proof}
It suffices to check $n(f(X),X)=0$ for all $X\in V$, by bilinearity. With our notation,
$$\begin{array}{l}n\left(M_{(a,x,y)}\left(\begin{array}{c}
s\\u\\v
\end{array}\right),\left(\begin{array}{c}
s\\u\\v
\end{array}\right)\right)= \left(\begin{array}{c}
-2y^tu-2x^tv\\sx+au+y\times v\\sy+x\times u-a^tv
\end{array}\right)^t \left(\begin{array}{c}
-s\\-2v\\-2u
\end{array}\right)\vspace{5pt}\\
=2s(y^tu+x^tv)-2\big(sx^tv+(au)^tv+(y\times v)^tv+sy^tu+(x\times u)^tu-(a^tv)^tu\big)\\
=-2((au)^tv-v^t(au))=0,
\end{array}
$$
since $(y\times v)^tv=\det(y,v,v)=0$ is the mixed product in $\FF^3$.
\end{proof}

Thus $n$ is $L_\FF$-invariant. But $\frso(V,n)$ has dimension $\left(\begin{array}{c}7\\2\end{array}\right)=21$, so that not every element preserving the norm belongs to $L_\FF$ ($\frg_2$ and $\frg_{2,2}$ have to preserve something besides the norm).

\begin{df}\label{def_crossproduct}
We will say that a bilinear map $\times\colon W\times W\to W$ is a \emph{cross product} if there exists a norm (nondegenerate quadratic form)
$q\colon W\to\FF$ such that
\begin{equation*}
\begin{array}{l}
\bullet\ q(u\times v,u)=q(u\times v,v)=0\vspace{3pt}\\
\bullet\ q(u\times v,u\times v)=\left|\begin{array}{cc}q(u ,u)&q(u ,v)\\q(v ,u)&q(v ,v)\end{array}\right|
\end{array}
\end{equation*} 
for all $u,v\in W$.
\end{df}

There are no cross products in  arbitrary dimension.
But, for dimension 7, there are:

\begin{lemma}\label{le_cross}
The following is a cross product in $V=\FF^7$:
$$
\left(\begin{array}{c}
s\\u\\v
\end{array}\right)
\times \left(\begin{array}{c}
t\\x\\y
\end{array}\right)
:=\left(\begin{array}{c}
2u^ty-2v^tx\\sx-tu-2v\times y\\-sy+tv+2u\times x
\end{array}\right),
$$
for the norm given in Eq.~\eqref{eq_lan}.
\end{lemma}

\begin{proof}
First, for $X=\left(\begin{array}{c}
s\\u\\v
\end{array}\right)$ and $Y= \left(\begin{array}{c}
t\\x\\y
\end{array}\right)$,
$$
\begin{array}{rl}
n(X,X\times Y)&=(-s\,-2v^t\,-2u^t)\left(\begin{array}{c}
2u^ty-2v^tx\\sx-tu-2v\times y\\-sy+tv+2u\times x\end{array}\right)\\
&=-2s(u^ty)+2s(v^ty)-2sv^tx+2tv^tu+4v^t(v\times y)\\&\quad+2su^ty-2tu^tv-4u^t(u\times x)\\
&=4\det(v,v,y)-4\det(u,u,x)=0.
\end{array}
$$
Similarly $n(Y,X\times Y)=0$. Finally, recall that, by linearizing the identity $(v\times y)\cdot(v\times y)=(v\cdot v)(y\cdot y)-(v\cdot y)(y\cdot v)$, we obtain the following  form of the Lagrange identity in $\FF^3$:
$$
(v\times y)\cdot(u\times  x)=(v\cdot u)(y\cdot x)-(v\cdot x)(y\cdot u).
$$
Hence
$$
\begin{array}{l}
n(X,Y)^2-n(X)n(Y)+n(X\times Y)=\vspace{3pt}\\
\ \ =(st+2v^tx+2u^ty)^2-(s^2+4u\cdot v)(t^2+4x\cdot y)-4(u\cdot y-v\cdot x)^2\\
\ \quad-4(sx-tu-2v\times y)\cdot (-sy+tv+2u\times x)\vspace{3pt}\\
\ \ =s^2t^2+4(v\cdot x)^2+4(u\cdot y)^2+4stv\cdot x+4stu\cdot y+8(v\cdot x)(u\cdot y)\\
\ \quad -s^2t^2-4s^2x\cdot y-4t^2u\cdot v-16(u\cdot v)(x\cdot y) -4(u\cdot y)^2-4(v\cdot x)^2+8(u\cdot y)(v\cdot x)\\
\ \quad +4s^2x\cdot y-4stx\cdot v-4stu\cdot y+4t^2u\cdot v+16(v\times y)\cdot (u\times x)\vspace{3pt}\\
\ \ =16(v\cdot  x)(u\cdot  y)-16(u\cdot  v)(x\cdot  y)+16(v\times y)\cdot  (u\times x)=0.
\end{array}
$$
\end{proof}
$L_\FF$ is precisely the Lie algebra making $\times $ invariant.
\begin{proposition}\label{pr_conservacross}
$L_\FF=\Der(V,\times).$
\end{proposition}

\begin{proof}
Denote by $\cD=\Der(V,\times)=\{f\in\frgl_7(\FF):f(X\times Y)=f(X)\times Y+X\times f(Y)\ \forall X,Y\in V\}$.
Let us prove first that if $f\in L_\FF$, then $f$ behaves well with respect the cross product, that is, $L_\FF\subset\cD$.

$\bullet$ For $f=M_{(a,0,0)}$, $X=\left(\begin{array}{c}
s\\u\\v
\end{array}\right)$ and $Y=
  \left(\begin{array}{c}
t\\x\\y
\end{array}\right)$,
$$
f(X)\times Y+X\times f(Y)=\left(\begin{array}{c}
0\\
-tau+2(a^tv)\times y + sax+2v\times 
(a^ty)\\
2(au)\times x-ta^tv+sa^ty+2u\times (ax)
\end{array}\right),
$$
$$
f(X\times Y)=f\left(\begin{array}{c}
*\\sx-tu-2v\times y\\-sy+tv+2u\times x
\end{array}\right)=\left(\begin{array}{c}
0\\a(sx-tu-2v\times y)\\a^t(sy-tv-2u\times x)
\end{array}\right),
$$
which coincide because for any $a=(a_{ij})\in\frsl_3(\FF)$, $au\times x+u\times ax=-a^t(u\times x)$ for all $u,x\in\FF^3$.
For instance, for elements in the canonical basis $u=e_i$, $x=e_{j}$, ($j=i+1$, $k=i+2$), $u\times x=e_k$ and $au\times x+u\times ax=-a_{ki}e_i-a_{kj}e_j+(a_{ii}+a_{jj})e_k=-a_{ki}e_i-a_{kj}e_j-a_{kk}e_k=-a^te_k$. (That is, $\cU\times\cU\to\cU^*$, $(u,v)\mapsto\det(u,v,-)$, is $\frsl_3(\FF)$-invariant.
At the group level, the matrices of determinant 1 preserve the volume form, since $\det(Pu,Pv,Pw)=\det P\,\det(u,v,w)$.)

$\bullet$  For $f=M_{(0,z,0)}$, $$
f(X)\times Y+X\times f(Y)=\left(\begin{array}{c}
2sz^ty-2(z\times u)^tx+2u^t(z\times x)-2tv^tz\\
-2(z^tv)x-2(z\times u)\times y+2(z^ty)u-2v\times(z\times x)\\
sz\times x-tz\times u+2(z^tv)y-2(z^ty)v
\end{array}\right),
$$
$$
f(X\times Y)=\left(\begin{array}{c}
-2z^t(-sy+tv+2u\times x)\\
2(u^ty-v^tx)z\\
z\times (sx-tu-2v\times y)
\end{array}\right).
$$
The scalars in the $(1)$-position coincide because $z^t(u\times x)=\det(z,u,x)$ is alternating.
The vectors in positions $(2)$ and $(3)$ coincide due to the identity of the double cross product 
$z\times (v\times y)=(z\cdot y)v-(z\cdot v)y.$

$\bullet$  The same arguments work to prove that $f=M_{(0,0,z)}\in\cD$.
\smallskip

Conversely, let us prove $\cD\subset L_\FF$. Recall that $\cD=\Der(V,\times)$  is the Lie algebra of the derivations of $V$ with respect to the product given by $\times$. Although with this product, $V$ is neither associative nor commutative, we can obtain a lot of information looking at $V$   as an algebra. First, $(V,\times)$ possesses a $\ZZ_3$-grading:
$$
V_{\bar0}=\left\{\left(\begin{array}{c}
s\\0\\0
\end{array}\right):s\in\FF\right\},
V_{\bar1}=\left\{\left(\begin{array}{c}
0\\u\\0
\end{array}\right):u\in\FF^3\right\}
,V_{\bar2}=\left\{\left(\begin{array}{c}
0\\0\\v
\end{array}\right):v\in\FF^3\right\}.
$$ 
This means that $V_{\bar i}\times V_{\bar j}\subset V_{\bar i+\bar j}$ (the sum of the indices in $\ZZ_3$), what is an straightforward computation.
    
Hence the Lie algebra $\frgl(V)$ is also $\ZZ_3$-graded ($\frgl(V)=\oplus_{\bar i\in\ZZ_3}\frgl(V)_{\bar i}$, where $f\in\frgl(V)_{\bar i}$ when $f(V_{{\bar j}})\subset V_{\bar i+\bar j}$ for all $\bar j\in\ZZ_3$) and $\cD=\Der(V,\times)$ is also  $\ZZ_3$-graded
($\cD=\oplus_{\bar i\in\ZZ_3}\cD_{\bar i}$, for $\cD_{\bar i}=\cD\cap\frgl(V)_{\bar i}$). This helps to the proof since we have only to check that $\cD_{\bar i}\subset L_{\FF}\cap\frgl(V)_{\bar i}$ (in fact, they coincide), even 
it suffices to prove that $\dim\cD_{\bar 0}=8$ and $\dim\cD_{\bar 1}=\dim\cD_{\bar 2}=3$.

Besides, take into account that the subspace $V_{\bar 1}$ generates $V$ as an algebra (since $V_{\bar 1}\times V_{\bar 1}=V_{\bar2}$ and $V_{\bar 1}\times V_{\bar 2}=V_{\bar 0}$  -not only contained-), what implies that if $d,d'\in\cD$ satisfy $d\vert_{V_{\bar 1}}=d'\vert_{V_{\bar 1}}$, then $d=d'$.

$\bullet$ Let $d\in\cD_{\bar 1}$. There is $z\in\cU$ such that $d\left(\begin{array}{c}
1\\0\\0
\end{array}\right)=\left(\begin{array}{c}
0\\z\\0
\end{array}\right)$ (since this element belongs to $V_{\bar 1} $). So, let us check that $d=M_{(0,z,0)}$, or equivalently 
$d(X)=M_{(0,z,0)}(X)$ for all $X\in V_{\bar 1}$. As $d(X)\in V_{\bar 2}$, there are $u,w\in\cU$ such that $X= \left(\begin{array}{c}
0\\u\\0
\end{array}\right)$ and $d(X)=\left(\begin{array}{c}
0\\0\\w
\end{array}\right)$.
Hence
$$
\left(\begin{array}{c}
0\\0\\w
\end{array}\right)=d\left( \left(\begin{array}{c}
1\\0\\0
\end{array}\right)\times X\right)=\left(\begin{array}{c}
0\\z\\0
\end{array}\right)\times \left(\begin{array}{c}
0\\u\\0
\end{array}\right)+\left(\begin{array}{c}
1\\0\\0
\end{array}\right)\times \left(\begin{array}{c}
0\\0\\w
\end{array}\right),
$$
so that $w=2z\times  u-w$ and $d(X)=\left(\begin{array}{c}
0\\0\\z\times  u
\end{array}\right)=M_{(0,z,0)}(X)$.

$\bullet$ Let $d\in\cD_{\bar 0}$. Then there are $s\in\FF$, $a,b\in\frgl_3(\FF)$ such that
$$
d\left(\begin{array}{c}
1\\0\\0
\end{array}\right)=\left(\begin{array}{c}
s\\0\\0
\end{array}\right),\quad d\left(\begin{array}{c}
0\\u\\0
\end{array}\right)=\left(\begin{array}{c}
0\\au\\0
\end{array}\right),\quad d\left(\begin{array}{c}
0\\0\\w
\end{array}\right)=\left(\begin{array}{c}
0\\0\\bw
\end{array}\right).
$$
Let us check that $a$ is a zero trace matrix and  $d=M_{(a,0,0)}$. First, note that $s=0$, since 
$$
 \left(\begin{array}{c}
0\\au\\0
\end{array}\right)=d\left(\begin{array}{c}
0\\u\\0
\end{array}\right)=d\left(\left(\begin{array}{c}
1\\0\\0
\end{array}\right)\times\left(\begin{array}{c}
0\\u\\0
\end{array}\right)\right)=\left(\begin{array}{c}
0\\su+au\\0
\end{array}\right).
$$
Second, $b=-a^t$, since
$$
0=d\left(\begin{array}{c}
2\\0\\0
\end{array}\right)=d\left(\left(\begin{array}{c}
0\\e_i\\0
\end{array}\right)\times\left(\begin{array}{c}
0\\0\\e_j
\end{array}\right)\right)=2\left(\begin{array}{c}
(ae_i)\cdot e_j+e_i\cdot(be_j)\\0\\0
\end{array}\right),
$$
and $(ae_i)\cdot e_j+e_i\cdot(be_j)=a_{ji}+b_{ij}$. Third, 
$$
\left(\begin{array}{c}
0\\0\\2be_3
\end{array}\right)=d\left(\begin{array}{c}
0\\0\\2e_3
\end{array}\right)=d\left(\left(\begin{array}{c}
0\\e_1\\0
\end{array}\right)\times  \left(\begin{array}{c}
0\\e_2\\0
\end{array}\right)\right)=\left(\begin{array}{c}
0\\0\\2(ae_1\times  e_2+e_1\times  ae_2)
\end{array}\right).
$$
Thus $-\sum a_{3i}e_i=-a^te_3=ae_1\times  e_2+e_1\times  ae_2$ and the coefficient of $e_3$ in both sides of this identity is $-a_{33}=a_{11}+a_{22}$, giving zero trace.
\end{proof}

\subsection{$G_2$ and the octonions}\label{se:octo}

Define $\cC=\FF\oplus V$ (isomorphic to $\FF^8$ as a vector space) with the product where $1\in\FF$ is a unit
and 
\begin{equation}\label{eq_procto}
XY:=-n(X,Y)1+X\times Y,
\end{equation}
for all $X,Y\in V$.
We extend our norm $n\colon V\to \FF$ to the norm $n\colon \cC\to \FF$ defined by $n(1)=1$, $n(1,V)=0$ and $n\vert_V=n$. We use the same symbol for the norm and for its polar form $n(X,Y)=\frac12(n(X+Y)-n(X)-n(Y))$.

\begin{df} An \emph{octonion algebra} is
\begin{itemize}
\item A unital algebra of dimension 8;
\item With a nondegenerate \emph{multiplicative} quadratic form $n$, that is, $n(XY)=n(X)n(Y)$ for all $X,Y$.
\end{itemize}
\end{df}

\begin{lemma} \label{le_octon}
$\cC$ is   an {octonion algebra}.
\end{lemma} 

\begin{proof}
We have only to prove that $n(XY)=n(X)n(Y)$.
If either $X$ or $Y$ is equal to $1$, it is clear. And if $X$ and $Y$ are in $V$, 
$$
\begin{array}{ll}
n(XY)&=n(-n(X,Y)1+X\times Y,-n(X,Y)1+X\times Y)\\
&=(-n(X,Y))^2n(1,1)+0+n(X\times Y)\\
&=n(X,Y)^2+\left|\begin{array}{cc}n(X ,X)&n(X ,Y)\\n(Y ,X)&n(Y ,Y)\end{array}\right|\\
&=n(X,X)n(Y,Y).
\end{array}
$$
\end{proof}

\begin{remark}
The relationship between cross products and multiplicative norms does not exist only in the octonion case. In fact, a  \emph{composition} algebra is an algebra with a nondegenerate multiplicative quadratic form. According to the classical Hurwitz's Theorem in 1898, there are unital finite-dimensional composition real\footnote{The theory   has subsequently been generalized to arbitrary quadratic forms and arbitrary fields, in particular $\CC$.} algebras only for dimensions $1$, $2$, $4$ and $8$. But whenever we have a cross product in an $\RR$-vector space $W$, then $\RR\oplus W$ is a composition algebra by the same argument as in the above proof. This is the way used by Brown and Gray \cite{Gray} to prove that cross products are only possible for   dimensions $1$, $3$ and $7$.
\end{remark}

\begin{remark}\label{re_cuadratica}
Note that  $\cC$ is a \emph{quadratic} algebra, that is, every element $X\in\cC$ satisfies a quadratic equation with coefficients in $\FF$:
\begin{equation}\label{eq_cuad}
X^2-2n(X,1)X+n(X)1=0.
\end{equation}
This is a straightforward computation: if $Y=s +X$ with $X\in V$, $s\in\FF$, then $Y^2=s^2 +2sX-n(X) $, $n(Y)=s^2+n(X)$ and $n(Y,1)=s$.

Now it is usual to denote 
$$\textrm{tr}(X):=2n(X,1)\in\FF,$$
 the \emph{trace}, so that $V$ can be identified with the zero trace elements in $\cC$, i.e. those ones orthogonal to $1$ (sometimes called \emph{imaginary} octonions by analogy with the  imaginary complex numbers).
 
 It will be useful to observe that the coefficients in the quadratic equation \eqref{eq_cuad} are determined, i.e. if $X\notin\FF1$ and $X^2-sX+t1=0$ then $s=\textrm{tr}(X)$ and $t=n(X)$.
\end{remark}

\begin{proposition}\label{pr_preservacross}
$L_\FF$ is isomorphic to the Lie algebra of derivations of the octonions $\Der(\cC)=\{d\in\frgl(\cC):d(XY)=d(X)Y+Xd(Y)\}$, by means of 
$$d\in L_\FF=\Der(V,\times)\mapsto \tilde d\in\Der(\cC),\quad\left\{\begin{array}{l}\tilde d(1)=0,\\\tilde d\vert_V=d.\end{array}\right.
$$
Furthermore,   $\Der(\cC)\subset\frso(\cC,n)$.
\end{proposition}

\begin{proof}
Let $d\in L_\FF$, $X,Y\in V$. So, $\tilde d(XY)=\tilde d(X\times Y)=d(X\times Y)=d(X)\times Y+X\times d(Y)$ is equal to
$\tilde d(X)Y+X\tilde d(Y)=-n(d(X),Y)1+d(X)\times Y-n(X,d(Y))1+X\times d(Y)$, since $d\in\frso(V,n)$ by Lemma~\ref{eq_lan}. Also $\tilde d(XY)=\tilde d(X)Y+X  \tilde  d(Y)    $ is obviously true if either $X$ or $Y$ is multiple of $1$, so that  $\tilde d\in\Der(\cC)$. 

Conversely, take $d\in\Der(\cC)$. Note first that $d(1)=2d(1)$, so that $d(1)=0$. 
Second, we are checking that $\Der(\cC)$ restricts to $V$. 
We apply $d\in\Der(\cC)$ to Eq.~\eqref{eq_cuad} getting $d(X^2)-\tr (X)d(X)=0 $ for all $X\in\cC$.
Thus
$$
\begin{array}{ll}
\tr (X)d(X)&=d(X^2)=d(X)X+Xd(X)=(X+d(X))^2-d(X)^2-X^2\\
&\stackrel{\eqref{eq_cuad}}{=}\tr(X+d(X))\,(X+d(X))-\tr(d(X))\,d(X)-\tr(X)\,X\\
&\qquad  -n(X+d(X))1+n(d(X))1+n(X)1\\
&=\tr(X)\,d(X)+\tr(d(X))\,X-2n(X,d(X))1,
\end{array}
$$
hence $\tr(d(X))\,X=2n(X,d(X))1$ for all $X\in\cC$. In particular, if $X\in V$, by projecting on $V$, we have  $\tr(d(X))=0$, so that $d(X)\in V$. As $d(1)=0$, then $d(\cC)\subset V$.
Third, and finally, if $X,Y\in V$, $d\in\Der(\cC)$,
$$
\begin{array}{ll}
d(X\times Y)&=d(XY)=d(X)Y+Xd(Y)\\
&=-n(d(X),Y)1+d(X)\times Y-n(X,d(Y))1+X\times d(Y).
\end{array}
$$
Here we have used that the image of $d$ is in $V$ for applying \eqref{eq_procto}. In particular (projecting on $V$), $d(X\times Y)=d(X)\times Y+X\times d(Y)$ and $d\in L_\FF$.
As a bonus  (projecting on $\FF$), we also get $n(d(X),Y)+n(X,d(Y))=0$ if $X,Y\in V$ (also consequence from Lemma~\ref{le_antisim}); but if both $X$ and $Y$ are scalars the statement is clear too; and if $X\in V$, $Y=1$, the equality consists of  $\tr(d(X))=0$. To summarize, $\Der(\cC)\subset\frso(\cC,n)$.
  \end{proof}

\section{Generic 3-forms in $\FF^7$}\label{sec_3formas}

In \cite{EngelG2}, Engel discovered an elegant description of the Lie algebra of type $G_2$ which enclosed the compact form. Before giving   proofs of his results about 3-forms (of course, not Engel's original proofs), we will try to go naturally to this concept from our previous situation. 

\subsection{$G_2$ and the triple product}\label{se:eformas}

Define a triple product in $V=\FF^7$ as follows:
\begin{equation}\label{eq_tripleproduct}
\{\ ,\ ,\ \}\colon V\times V\times V\to\FF,\quad\{X,Y,Z\}:=n(X\times Y,Z).
\end{equation}
It is an alternating (trilinear) map, taking into account that $n(X\times Y,Y)=0$ and $n(X\times Y,X)=0$, as in Lemma~\ref{le_cross}.
So, it can be identified with a 3-form $\wedge^3V\to \FF$.

Note that 
$$
L_\FF\subset\{f\in\frgl_7(\FF):\{f(X),Y,Z\}+\{X,f(Y),Z\}+\{X,Y,f(Z)\}=0
\ \forall X,Y,Z\in V\},
$$
which is a trivial consequence of   Proposition~\ref{pr_conservacross} 
and the fact $f\in L_\FF\subset\frso(V,n)$:
$$
\begin{array}{c}
n(f(X)\times Y,Z)+n(X\times f(Y),Z)+n(X\times Y,f(Z))=\qquad\qquad\qquad\\
\qquad\qquad\qquad\qquad\qquad\qquad= n(f(X\times Y),Z)+n(X\times Y,f(Z))=0.
\end{array}
$$
Our objective is to prove the equality of both algebras.\footnote{ 
This is frequently mentioned in the literature, for instance in  \cite[Example~30]{orbitas}, which just uses the model in \ref{pr_defdeL}.
}
To this aim, it is convenient to understand the 3-form  better:

\begin{remark}\label{re_3formaconcreta}
If $\{V_{\bar i},V_{\bar j},V_{\bar k}\}\ne0\Rightarrow \bar i+\bar j+\bar k=\bar 0.$
Denote by
$$
E_{ 0}= \left(\begin{array}{c}
1\\0\\0
\end{array}\right),\quad
E_{i}= \left(\begin{array}{c}
0\\e_i\\0
\end{array}\right),\quad
F_{i}= \left(\begin{array}{c}
0\\0\\e_i
\end{array}\right),
$$
so that the canonical basis of $V=\FF^7$ is $\cB_c=\{E_0,E_1,E_2,E_3,F_1,F_2,F_3\}$.
According to Lemma~\ref{le_cross}, the only nonzero $\times $-products of basic elements (taking into account skew-symmetry) are:
$$
\begin{array}{lll}
E_0\times E_i=E_i,
\qquad\ &
E_i\times E_{i+1}= 2F_{i+2},\\
E_0\times F_i=-F_i,&
F_i\times F_{i+1}=-2E_{i+2},\\
E_i\times F_{i}= 2E_{0}.&
\end{array}
$$ 
And, according to Eq.~\eqref{eq_lan}, the only nonzero $n$-products of basic elements (but taking into account symmetry) are:
$$
n(E_0, E_0)=-1,\qquad\quad n(E_i, F_i)=-2.
$$
Hence, the only nonzero $\{\,,\,,\,\}$-products of basic elements (taking into account it is alternating) are:
$$
\{E_0,E_i,F_i\}=-2,\quad
\{E_1,E_2,E_3\}=-4,\quad
\{F_1,F_2,F_3\}=4,\quad
$$
or, in other words, the triple product considered as a 3-form is
$$\begin{array}{ll}\{\,,\,,\,\}=&-2\,\big(E_0^*\wedge E_1^*\wedge F_1^*+E_0^*\wedge E_2^*\wedge F_2^*+E_0^*\wedge E_3^*\wedge F_3^*\big)\vspace{3pt}
\\&\qquad-4\,E_1^*\wedge E_2^*\wedge E_3^*+4\,F_1^*\wedge F_2^*\wedge F_3^*.
\end{array}$$
\end{remark}

\begin{proposition}\label{pr_enson}
$L_\FF=\Der(V, \{\,,\,,\,\} )=$
$$\begin{array}{c}
=\{f\in\frgl_7(\FF):\{f(X),Y,Z\}+\{X,f(Y),Z\}+\{X,Y,f(Z)\}=0
\ \forall X,Y,Z\in V\}.
\end{array}
$$
\end{proposition}

\begin{proof}
Let $f$ be a map preserving the triple product in Eq.~\eqref{eq_tripleproduct}. Our aim is  to prove that $f\in\frso(V,n)$,
because if we knew  $f\in\frso(V,n)$ then
$$
\begin{array}{rl}
0=& n(f(X\times Y),Z)+n(X\times Y,f(Z))=n(f(X\times Y),Z)+\{X,Y,f(Z)\}\\
 =&n(f(X\times Y),Z)-\{f(X),Y,Z\}-\{X,f(Y),Z\}\\
=&n(f(X\times Y) -f(X)\times Y-X\times f(Y),Z),
\end{array}
$$
and so, as $n$ is nondegenerate, this would imply $f\in\Der(V,\times)=L_\FF$ (Proposition~\ref{pr_conservacross}).

 If $A=(a_{ij})_{i,j=0}^6$ represents the matrix of $f$ relative to the basis $\cB_c$, then $f\in\frso(V,n)$ equivales to $A^tN+NA=0$, that is, to the fact that $A$ has, as a block-matrix, the shape
  \begin{equation}\label{eq_antisim}
  \left(\begin{array}{ccc}0&-2y^t&-2x^t\\x &d&b\\y&c&-d^t\end{array}\right),
  \end{equation}
 for some $x,y\in\FF^3$ and $b,c,d\in\Gl_3(\FF)$ such that  $b+b^t=0=c+c^t$.
According to Remark~\ref{re_3formaconcreta},
\begin{equation}\label{eq_hola}
\begin{array}{rl}
0=& \{f(E_0),E_i,F_i\}+\{E_0,f(E_i),F_i\}+\{E_0,E_i,f(F_i)\}\\
=&\{a_{00}E_0,E_i,F_i\}+\{E_0,a_{ii}E_i,F_i\}+\{E_0,E_i,a_{i+3,i+3}F_i\}\\
=&-2( a_{00}+a_{ii}+a_{i+3,i+3}),
\end{array}
\end{equation}

\noindent and also
$$
\begin{array}{rl}
0=& \{f(E_1),E_2,E_3\}+\{E_1,f(E_2),E_3\}+\{E_1,E_2,f(E_3)\}\\
=&-4( a_{11}+a_{22}+a_{33}),\\
0=& \{f(F_1),F_2,F_3\}+\{F_1,f(F_2),F_3\}+\{F_1,F_2,f(F_3)\}\\
=&4( a_{44}+a_{55}+a_{66}).
\end{array}
$$
If we now   sum Eq.~\eqref{eq_hola} for $i=1,2,3$,  we get  $-6a_{00}-2\sum_{i=1}^6a_{ii}=0$,  that joint with the two above identities imply $a_{00}=0$.

Now we apply
$\{f(X),Y,Z\}+\{X,f(Y),Z\}+\{X,Y,f(Z)\}=0$ to some various triplets $\{X,Y,Z\}$:
\begin{itemize}
\item[i)] $\{X,Y,Z\}=\{E_0,E_{i},F_{i+1}\}$ gives $0=a_{i+1,i}+a_{i',i+1'}$ for $i=1,2,3$; 

\noindent -In order to simplify the notation we have replaced  the indices $j+3$ with $j'$ when $j=1,2,3$ (so, if for instance $i=3$, then $i+1'$ means $4$, first sum modulo 3 and second sum $3$ in $\ZZ$);-

\noindent $\{X,Y,Z\}=\{E_0,E_{i},F_{i+2}\}$ gives $0=a_{i+2,i}+a_{i',i+2'}$,

\noindent  also $a_{ii}+a_{i+3,i+3}=0=a_{00}$ from \eqref{eq_hola} with the new notation is $0=a_{ii}+a_{i',i'}$, that gives $X_{33}=-X_{22}^t$ thinking of $A$ as a block-matrix $A=(X_{ij})_{i,j=1}^3$ of sizes as in \eqref{eq_antisim}.
\item[ii)] $\{X,Y,Z\}=\{F_{i},E_{i+1},E_{i+2}\}$ provides $a_{i,i'}=0$ (zero the whole diagonal of $b=X_{23}$),

\noindent $\{X,Y,Z\}=\{F_{i+1},E_{i+1},E_{i+2}\}$  implies $2a_{i,i+1'}=a_{0,i+2}$,

\noindent $\{X,Y,Z\}=\{F_{i+2},E_{i+1},E_{i+2}\}$   implies $-2a_{i,i+2'}=a_{0,i+1}$, and compiling the information $a_{i,i+1'}=\frac12a_{0,i+1+1'}=-a_{i+1,i'}$ (sum modulo 3) and $b$ is skew-symmetric.
\item[iii)]
$\{X,Y,Z\}=\{E_0,F_{i+1},F_{i+2}\}$ leads to $2a_{i',0}=a_{i+2,i+1'}-a_{i+1,i+2'}$,
so that by ii),  $2a_{i',0}=2a_{i+2,i+1'}=2a_{i+2,i+2+2'}=-a_{0,i+2+1}=-a_{0i}$, which confirms the relationship between the blocks $X_{31}$ and $X_{12}$. 
\end{itemize}
The blocks $(2,1),(3,2), (1,3)$ are obtained interchanging $E$'s and $F$'s in the previous series of triplets $\{X,Y,Z\}$.
(In fact we have directly checked $b=l_y$, $\tr(d)=0$ and $A\in L_\FF$.)
\end{proof}

\begin{remark}
The proof should be essentially based on the fact that the norm $n$ can be recovered from the triple product, although this fact seems hidden at a first sight (that is the reason why preserving the triple product is preserving the cross product, but the first object is   easier). Soon we will try to prove this in general, but now let us check it in our concrete example 
-of triple product as in Eq.~\eqref{eq_tripleproduct} and $n$ given by Eq.~\eqref{eq_lan}-:
If $X=(s\,u\,v)^t$, $n(X)=-s^2-4u^tv=\frac14\left( -\{X,E_1,F_1\}\{X,E_2,F_2\}+
\sum_{i=1}^3\{X,E_{i}, E_{i+1}\}\{X,F_{i},F_{i+1}\}
\right)$ (expression too reliant on our basis), which, multiplying with $\{E_{0},E_{i+2},F_{i+2}\}=-2$ and varying the possibilities, yields
$$
n(X)=\frac{-1}{2^63^2}\sum_{ \sigma\in S_7  }
(-1)^{\textrm{sg}\sigma} \{X,e_{\sigma({1})},e_{\sigma({2})}\}\{X,e_{\sigma({3})},e_{\sigma({4})}\}
 \{e_{\sigma({5})},e_{\sigma({6})},e_{\sigma({7})}\}.
$$
(If we take the sum over the permutations $\sigma\in S_7$ with
$\sigma(1)<\sigma(2)$,
$\sigma(3)<\sigma(4)$ and $\sigma(5)<\sigma(6) <\sigma(7)$, then the scalar is $\frac{-1}{2^33}$, since  $-8$ in our first expression appears $6$ times.)
This motivates Lemma~\ref{le_lanorma}.
\end{remark}   

Until now   we have worked at the level of the Lie algebra since, although it may seem less conceptual than working with the group,   the involved  equations are linear!
Anyway   the previous results can be read in terms of groups.

\begin{corollary}\label{co_conexossplit}
The group generated by $\exp(\frg_{2})\subset\GL_7(\CC)$ (resp. $\exp(\frg_{2,2})\subset\GL_7(\RR)$) is a complex simple Lie group of type $G_2$ (resp. real noncompact) contained in $\SO(V,n)$, which coincides, for $\FF=\CC$ (resp. $\RR$), with

\begin{equation}\label{eq_descrp1}
\Aut(V,\times)=\{f\in\GL_7(\FF): f(X)\times f(Y)=f(X\times Y) \  \forall X,Y\in V\}, 
\end{equation}
also with the connected component of the identity of  the   group
\begin{equation}\label{eq_preserva3forma}
\{f\in\GL_7(\FF):\{f(X),f(Y),f(Z)\}=\{X,Y,Z\}\  \forall X,Y,Z\in V\},
\end{equation}
and it is isomorphic -under the obvious extension to $\cC=\FF\oplus V$,  making $f(1)=1$- to
\begin{equation}\label{eq_descrp2}
\Aut(\cC)=\{f\in\GL(\cC): f(X  Y) =f(X) f(Y)\  \forall X,Y\in \cC\}.
\end{equation}
\end{corollary}

This result is not completely trivial. The groups in \eqref{eq_descrp1} and \eqref{eq_descrp2} are Lie groups, subgroups of $\GL_7(\FF)$  \cite[Theorem~3.54]{Warner} and also is the group in \eqref{eq_preserva3forma}, a closed subgroup of $\GL_7(\FF)$. As we have proved that all of them have the same Lie algebra, then it is true that the connected components of the identity coincide 
 \cite[Theorem~3.19]{Warner}. So the corollary will be clear when we see that the groups in \eqref{eq_descrp1} and \eqref{eq_descrp2} are connected. But proving the connectedness is not usually an easy task. We will check it carefully in Section~\ref{se_conexo}, obtaining it as a consequence of the diffeomorphism 
$ G_{2}/\mathrm{SL}_3(\RR)\cong H_3^6(1)$ (and of its complex version). The study of this homogeneous space and of the 6-dimensional sphere
 $\mathbb S^6\cong G_2/{\mathrm{SU}}(3)$   is  one of our purposes in this 
 paper. In the meantime, observe that the group in \eqref{eq_preserva3forma} has nontrivial center in the complex case because $\omega\id_V$ belongs to the  group if $\omega \in\CC$ is a   cubic root of the unit.\footnote{There are connected Lie groups with nontrivial centre, for instance $\widetilde G_{2,2}$ in Remark~\ref{re_elsc}. 
 But the complex group \eqref{eq_preserva3forma} is not that case,  its discrete center giving the number of connected components.
 }



\begin{remark}
S.~Garibaldi suggested to me some alternative arguments. For instance, Corollary~\ref{eq_descrp1} is immediate using this result: a closed subgroup of $\mathrm{SL}_7(\CC)$ (hence a Lie group) with identity connected component equal to $G_2$ should be contained in   the subgroup of $\mathrm{SL}_7(\CC)$ generated by $ G_2$ and the scalar
matrices. This is the case for $H$ any of the groups in \eqref{eq_descrp1}, \eqref{eq_preserva3forma} and \eqref{eq_descrp2}, 
%
whose identity connected component   coincides with $G_2$ by the above sections, or independently by using the argument in \cite[Lemma~5.3]{garibaldi} that allows to conclude that any -Zariski- closed connected subgroup lying properly between $G_2$ and $\mathrm{SL}_7(\CC)$  is simple and then, by checking the possibilities with root data, necessarily $H$ should be the orthogonal group $\SO(V,n)$, for $n$ the $G_2$-invariant quadratic form on $V$. But, of course,  $\SO(V,n)$   stabilizes neither $\times$, nor $\{\,,\,,\,\}$, nor the octonionic product.

In order to prove the above result, note that
if $g$ is in the normalizer of
$G_2$ in $\mathrm{SL}(V)$, then $\mathrm{Ad}\,g\in\Aut(\mathrm{SL}(V))$ given by conjugation restricts to $G_2$, so 
$\mathrm{Ad}\,g\vert_{G_2}\in \Aut(G_2)$. A well known fact is that every automorphism of $G_2$ is inner, so that there is $g'\in G_2$  such that $\mathrm{Ad}\,g\vert_{G_2}=\mathrm{Ad}\,g'\vert_{G_2}$. Hence
$ g(g')^{-1}$ commutes with $G_2$ and hence it belongs to $\mathrm{End}_{G_2}(V)=\CC\id_V$ by Schur's lemma ($V$ is $G_2$-irreducible). Thus $g$ is the product of an element in $G_2$ by a scalar map. But any   subgroup of $\mathrm{SL}(V)$ with identity connected component equal to $G_2$ lives in the normalizer.  

%

\end{remark}

\subsection{Generic 3-forms}\label{genericas}
Look at the previous subsection from a different approach. Consider the triple product    in Eq.~\eqref{eq_tripleproduct}
as a 3-form $\Omega_0\in\wedge^3V^*$ (as in Remark~\ref{re_3formaconcreta}). The action of $\GL_7(\FF)\equiv\GL(7)$ in the vector space of 3-forms is given by the pull-back
\begin{equation}\label{eq_accion}
\GL(7)\times\wedge^3V^*\to \wedge^3V^*,\quad (f,\Omega)\mapsto f\cdot\Omega:=f^*\Omega
\end{equation}
where $f^*\Omega(X,Y,Z)=\Omega(f^{-1}(X),f^{-1}(Y),f^{-1}(Z))$. Equation~\eqref{eq_preserva3forma} precisely says that 
$$
G_{\Omega_0}=\{f\in\GL(7):f\cdot\Omega_0=\Omega_0\},
$$
is a Lie group of type $G_2$, the isotropy group (the stabilizer) of the element $ \Omega_0\in\wedge^3V^*$.
This implies that the orbit of $\Omega_0$ (see Remark~\ref{re_orbitas} below) 
$\cO_{\Omega_0}=\{ f\cdot\Omega_0:f\in\GL(7)\}$
is a homogeneous space diffeomorphic to $\GL(7)/G_{\Omega_0}$, in particular of dimension $49-14=35$. But $\dim\wedge^3V^*=\left(\begin{array}{l}7\\3\end{array}\right)=35$, so that $\cO_{\Omega_0}$ should be an open set of $\wedge^3V^*$. In other words, $ {\Omega_0}$ is a generic $3$-form in $\wedge^3V^*$.

\begin{remark}\label{re_orbitas}
We add some references about why the orbits are submanifolds and they are open sets or not according to the dimension of the isotropy group. (Although here we have a polynomial action on a vector space, this is not necessary at all.)

If $G$ is a Lie group acting (smoothly) on a connected manifold $M$, then every orbit 
$\cO_{x_0}$ of a point $x_0\in M$ is a submanifold \cite[Lemma~2.13]{Sharpe}. In the proof, Sharpe takes $H_{x_0}=\{g\in G: g\cdot x_0=x_0\}$ the isotropy group and he shows that it is a closed subgroup of $G$, and that the induced map
$$
G/H_{x_0}\to M,\qquad gH\mapsto g\cdot x_0
$$
is an injective immersion with image $\cO_{x_0}$. So the inclusion $\iota\colon \cO_{x_0}\hookrightarrow M$ is also an immersion. Let us check that $\cO_{x_0}$ is open if and only if $\dim \cO_{x_0}$ (which coincides with $\dim G-\dim H_{x_0}$) equals $\dim M$. Of course, this condition is necessary.  But it is sufficient too, because if $\dim M=\dim \cO_{x_0}$ and $x\in \cO_{x_0}$, the differential $(\iota_*)_x\colon T_x\cO_{x_0}\to T_xM$ is not only a monomorphism but an isomorphism. Hence, Inverse Function Theorem \cite[1.30 Corollary (a)]{Warner} says that there is a neighborhood $U$ ($\subset \cO_{x_0}$) of $x$ such that $\iota\vert_U\colon U\to \iota(U)=U$ is a diffeomorphism onto the open set $U$ of $M$. The fact that each $x\in \cO_{x_0}$ possesses $U$ an open set of $M$ such that $x\in U\subset \cO_{x_0}$ means obviously  that $\cO_{x_0}$ is open.
\end{remark}

\begin{df}
A $p$-form $\Omega\in\wedge^p(\FF^n)^*$ ($p\le n$) is said a \emph{generic} form when its orbit under the action of $\GL(n)$ is open, what happens if and only if $$\dim G_{\Omega}=\dim \GL(n)-\dim \wedge^p(\FF^n)^*=n^2-\left(\begin{array}{l}n\\p\end{array}\right),$$
for $G_{\Omega}=\{f\in\GL(n):f\cdot\Omega=\Omega\}$.
\end{df}

The existence of generic $p$-forms is not possible in most of the dimensions. For instance, by dimension count, if $p=3$, this forces $n$ to be   at most $8$.
But, according to the above arguments, the existence in case  $(n,p)=(7,3)$ is guaranteed since just $\Omega_0$ is a generic $3$-form, and of course all the $3$-forms in the orbit  $\cO_{\Omega_0}$ are   generic too.  Some natural questions arise: how many orbits of generic $3$-forms will  there be in $\wedge^3V^*$? And, in case there are more than one, which is the isotropy group of an element in a second orbit? Of course, this group should have dimension $14$, 
but, is it also of type $G_2$? Although  the only  -complex and real-  simple Lie groups of dimension $14$ have type $G_2$,  nobody ensures us (a priori) that the isotropy group of an arbitrary $3$-form is simple. The answers were given so soon as in 1900 \cite{EngelG2}.\footnote{The classification of the orbits of the $3$-forms appears in   
Walter Reichel's thesis -see references in \cite{AgriG2}- and later in \cite{Schouten}. From an Algebraic Geometry approach, \cite{orbitas} classifies the connected linear algebraic groups over the complex numbers such that there is a rational irreducible representation
of the group on a finite-dimensional vector space   whose orbit is 
 Zariski-dense, and also their   invariants. Trilinear alternating forms on a 7-dimensional   vector space are classified in \cite{tresformas}.}

%

\begin{theorem}\label{teo_casoC}
There is only one orbit of generic 3-forms in $\wedge^3(\CC^7)^*$.
A representative is
$$\begin{array}{ll}\Omega_0=&-2\,\big(E_0^*\wedge E_1^*\wedge F_1^*+E_0^*\wedge E_2^*\wedge F_2^*+E_0^*\wedge E_3^*\wedge F_3^*\big)\vspace{3pt}
\\&\qquad-4\,E_1^*\wedge E_2^*\wedge E_3^*+4\,F_1^*\wedge F_2^*\wedge F_3^*.
\end{array}$$
\end{theorem}
\begin{theorem}\label{teo_casoR}
  There are just two orbits of generic 3-forms in $\wedge^3(\RR^7)^*$. A pair of representatives are $\Omega_0$ and
 \begin{equation}\label{eq_omegandefinida}
 \Omega_1=e^{ 147}+e^{257 }+e^{367 }+e^{123 }-e^{156 }+e^{246 }-e^{345 },
 \end{equation} 
where the used notation is $e^{ ijk}:=e_i^*\wedge e_j^*\wedge e_k^*$ for  $\{e_i^*\}_{i=1}^7$ a   basis of $(\RR^7)^*$.
\end{theorem}

Before proving these theorems, we need some auxiliary results.

\begin{lemma}\label{le_lanorma}
Let $\Omega\in\wedge^3V^*$ be a $3$-form.
For any  basis  $\cB=\{b_i:i=1,\dots,7\}$    of $V=\FF^7$, define $n_{\Omega,\cB}\colon V\times V\to \FF$ the bilinear symmetric form given  by
$$
n_{\Omega,\cB}(X,Y)= \sum_{ \sigma\in S_7  }
(-1)^{\textrm{sg}\sigma} \Omega(X,b_{\sigma({1})},b_{\sigma({2})})\Omega(Y,b_{\sigma({3})},b_{\sigma({4})})
 \Omega(b_{\sigma({5})},b_{\sigma({6})},b_{\sigma({7})}).
$$
Then,
\begin{itemize}
\item[a)] $n_{\Omega,\cB'} =(\det P)\,n_{\Omega,\cB}$ for $P$ the matrix of the change of basis from $\cB'$ to $\cB$, so that the norm is determined by the $3$-form up to scalar multiple;
\item[b)] If $\Omega$ is  a generic 3-form, then $n_{\Omega,\cB}$ is a nondegenerate bilinear symmetric form.\footnote{In fact, this is an equivalence, see Remark~\ref{re_reciproco} below.}  
\end{itemize}
\end{lemma}

\begin{proof}
The symmetry comes from $\textrm{sg}\,\sigma=\textrm{sg}(\sigma(3)\,\sigma(4)\,\sigma(1)\,\sigma(2)\,\sigma(5)\,\sigma(6)\,\sigma(7))$ for any permutation $\sigma$. The  bilinearity is obvious,  consequence of being $\Omega$ trilinear. 

For proving a), if $P=(a_{ij})$ is the matrix of the change of basis from $\cB'$ to $\cB$, then the $k$th element in $\cB'$ is 
$\sum_i a_{ik}b_i$. For   any $X,Y\in V$, the trilinearity implies that $n_{\Omega,\cB'}(X,Y)$ is equal to
\begin{equation}\label{eq_sigmas}
 \sum_{\tiny{ \begin{array}{l}\sigma\in S_7\\i_j=1,\dots,7\end{array}  }}
(-1)^{\textrm{sg}\sigma}a_{i_1\sigma(1)}\dots a_{i_7\sigma(7)} \Omega(X,b_{i_1},b_{i_{2}})\Omega(Y,b_{i_{3}},b_{i_{4}})
 \Omega(b_{i_{5}},b_{i_{6}},b_{i_{7}}).
\end{equation}
But this sum can be considered indexed in $(i_1,\dots,i_7)\in S_7$ because, if $(i_1,\dots,i_7)$ is a fixed $7$-tuple which is not a permutation of $(1,\dots,7)$, then 
\begin{equation*}  
 \sum_{\tiny{  \sigma\in S_7    }}
(-1)^{\textrm{sg}\sigma}a_{i_1\sigma(1)}\dots a_{i_7\sigma(7)} \Omega(X,b_{i_1},b_{i_{2}})\Omega(Y,b_{i_{3}},b_{i_{4}})
 \Omega(b_{i_{5}},b_{i_{6}},b_{i_{7}})=0.
\end{equation*}
Indeed, there will be $j\ne k\in\{1,\dots,7\}$ such that $i_j=i_k$, and for each $\sigma\in S_7$ the related summand is just the opposite of the summand for the permutation $\mu=\sigma\circ (j\,k)$ 
($\mu(j)=\sigma(k)$, $\mu(k)=\sigma(j)$ and $\mu(i)=\sigma(i)$ for the remaining indices), since 
$(-1)^{\textrm{sg}\sigma}=-(-1)^{\textrm{sg}\mu}$ and the rest of the expression does not change.
 Thus, we obtain that \eqref{eq_sigmas} equals $(\det P)\,n_{\Omega,\cB}(X,Y)$ by writing 
 $\det P =\sum_{\sigma\in S_7}(-1)^{\textrm{sg}\sigma\circ i}a_{i_1\sigma(1)}\dots a_{i_7\sigma(7)}$.

In particular, the  nondegeneracy does not depend on the choice of $\cB$. Of course, also nondegeneracy does not depend on the choice of the $3$-form in a fixed orbit, since for every $f\in\GL(V)$,
$$
n_{f\cdot \Omega,\cB}(X,Y)=n_{\Omega,f^{-1}(\cB)}(f^{-1}(X),f^{-1}(Y)).
$$
So, let us fix $\cB$ and also fix $\psi\colon V^*\to V$ an isomorphism (there is no \emph{natural} isomorphism, but both spaces have the same dimension, so choose a basis $\{e_i^*\}_{i=1}^7$ of $V^*$ and take $\psi(e_i^*)=e_i$,  for instance the dual basis of $\cB$).  
For each $3$-form $\Omega$, define $\varphi_{\Omega}\colon V\to V^*$   by $\varphi_\Omega(X)(Y):=n_{\Omega,\cB}(X,Y)$. Thus, $n_{\Omega,\cB}$ is nondegenerate if and only if $\varphi_\Omega$ is an injective (linear) map, or equivalently if $\det\varphi_\Omega\ne0$, where we understand by $\det\varphi_\Omega:=\det\psi\varphi_\Omega$.
We have all the ingredients to see that if   $n_{\Omega,\cB}$ degenerates, then the orbit $\cO_{\Omega}$ cannot be open. Simply note that, as $\det\varphi_\Omega=0$, thus
$$
\cO_{\Omega}=\{f\cdot \Omega:f\in\GL(V)\}\subset \cM=\{\omega\in \wedge^3V^*:\det\varphi_\omega=0\},
$$
but $\cM$ cannot contain any open subset of the vector space $\wedge^3V^*$, because if $\omega=\sum_{1\le i<j<k\le 7} a_{ijk}e_i^*\wedge e_j^*\wedge e_k^*$ (we identify the form with its coordinates $(a_{ijk})\in\FF^{35}$), then $\det\varphi_\omega$ is  a homogeneous of degree 21 polynomial with coefficients in $\FF$ and variables $a_{ijk}$. 
\end{proof}

\begin{remark}\label{re_elijon}
It will be useful the fact that  
one can choose a basis $\cB$ such that $\cB$ is orthonormal\footnote{In the real case, we understand that a basis is   orthonormal when the vectors are  pairwise  orthogonal and each of them has length $\pm1$. }  for $n_{\Omega,\cB}$ (when it is nondegenerate).
Indeed, take $\cB_0 $ an arbitrary basis of $V$ and $n=n_{\Omega,\cB_0}$. Take $\cB_1=\{e_i\}_{i=1}^7 $ an orthonormal basis for $n$. Take $\cB_2=\{te_i\}_{i=1}^7 $, where $t\in\FF$ is a fixed nonzero scalar which we will determine next.
The matrix of the change of basis of $\cB_2$ to $\cB_1$ is $tI_7$ with determinant $t^7$, so
$$
n_{\Omega,\cB_2}(te_i,te_i)=t^7n_{\Omega,\cB_1}(te_i,te_i)=t^9n_{\Omega,\cB_1}(e_i,e_i)=t^9\alpha, 
$$
being $\alpha$ the determinant of the change of basis of $\cB_1$ to $\cB_0$.
 Of course $\cB_2$ is orthogonal for $n_{\Omega,\cB_2}$ (the orthogonality does not change by multiplying by nonzero scalars), so in order to obtain the desired basis we have only to choose $t\in\FF$ such that $t^9\alpha=1$. 
  Of course this is possible  whether $\FF=\RR$ or $\CC$.\smallskip
 
 Furthermore, an argument similar to the previous one shows that a norm $n$ related to $\Omega$ satisfying such property ($\cB$ is orthonormal for $n=n_{\Omega,\cB}$) is determined up to a ninth root of the unit (in particular, completely determined in the real case).   Simply observe that if $\cB$ and $\cB'$ are orthonormal for $n=n_{\Omega,\cB}$ and $n'=n_{\Omega,\cB'}$ respectively, then $n'=\alpha n$ for $\alpha=\det P$ being  $P=C_{\cB\cB'}$ the change of basis, so that $\alpha P^tP= I_7$ and, taking determinants, $\alpha^9=1$.
\end{remark}

\begin{lemma}\label{le_parteconexa}
Let $\Omega$ be a   3-form on $V=\FF^7$ such that the  
related  bilinear form $n$ as in Lemma~\ref{le_lanorma} is nondegenerate. Then
\begin{itemize}
\item[a)] $(\det g)^9=1$ for all $g\in G_\Omega$;
\item[b)] $G_\Omega^+\subset G_\Omega\cap\SL(V)=G_\Omega\cap\mathrm{O}(V,n)$;
\item[c)] If $\FF=\RR$, then $G_\Omega \subset\SO(V,n)$.
\end{itemize}
\end{lemma}

\begin{proof} 
a) Fix $\cB$ any basis of $V=\FF^7$ orthonormal for $n=n_{\Omega,\cB} $.
If $g\in G_\Omega$, as 
$$
n_{\Omega,\cB} (X,Y)=n_{g^{-1}\cdot\Omega,\cB} (X,Y)=n_{\Omega,g(\cB)} (g(X),g(Y)),
$$
 then $g(\cB)$ satisfies the same property as $\cB$, and, by the above remark, $(\det g)^9=1$.

b) For any $g\in G_\Omega$, by Lemma~\ref{le_lanorma}\,a),
$$
n_{\Omega,\cB}(X,Y)=n_{\Omega,g(\cB)} (g(X),g(Y))=(\det g) \,n_{\Omega, \cB}(g(X),g(Y)),
$$
so, if $\det g=1$ then $g\in \mathrm{O}(V,n)$ and the converse is also true. This proves $G_\Omega\cap\SL(V)=G_\Omega\cap\mathrm{O}(V,n)$.
Moreover, by a), the map 
$$
\det\colon G_\Omega\to\{\xi\in\CC:\xi^9=1\}
$$
is   well defined. But it is a continuous map, what implies that the image of $G_\Omega^+$ is constant, equal to 1 
since $\id_V\in G_\Omega^+$. This means that
$G_\Omega^+$ is a subgroup of $\SL(V)$, so that hence a subgroup of $\SO(V,n)$.

c) This is trivial now, since $1$ is the only ninth root of the unit in $\RR$.
\end{proof}

 \begin{remark}\label{re_cotasdimensiones}
Consider the action $\GL_7(\FF)\times\wedge^3V^*\to \wedge^3V^*$ as in Eq.~\eqref{eq_accion}, and a  $3$-form $\Omega$  such that the related bilinear form $n$ is nondegenerate. We know that $\dim G_\Omega\ge14$, and it is just 14 when $\Omega$ is generic (afterwards we will prove that this will be the case, since $n$ is nondegenerate). 
As in the above lemma, $G_\Omega^+ \subset\SO(V,n)$. We will prove in Lemma~\ref{pr_concomplejo} and in Lemma~\ref{pr_conreal} (complex and real cases, respectively) that $G_\Omega^+=G_\Omega\cap  \mathrm{O}(V,n)$, but, in the meantime we will denote by $G_\Omega^0:= G_\Omega \cap \mathrm{O}(V,n)$, for working with this group of dimension equal to $\dim G_\Omega$.
 %
Thus, the orbit $G_\Omega^0\cdot X$ for a nonisotropic vector $X\in V$ of the action $G_\Omega^0\times V\to V$ will be fully contained in $\{Y\in V: n(Y)=n(X)\}$, a manifold which has dimension $6$ whenever $n\ne0$. This implies that the isotropy group
\begin{equation*}\label{eq_isotgr}
H_X=\{g\in G_\Omega^0: g(X)=X\}
\end{equation*}
has dimension $\dim H_X=\dim G_\Omega^0-\dim G_\Omega^0\cdot X\ge 14-6=8$.

Define $Y\wedge Z\in V$ to satisfy $n(Y\wedge Z,-)=\Omega(Y,Z,-)$, which is well defined by nondegeneracy  of $n$. Then $H_X$ is a subgroup of
\begin{equation}\label{eq_hx}
H_X':=\{g\in\GL(V): n(g(Y))=n(Y),\ g(X)=X,\ g(X\wedge Y)=X\wedge g(Y)\ \forall Y\in V\},
\end{equation}
which will  also be a group of dimension at least 8. This fact $H_X\subset H_X' $ is a direct consequence of 
$$
n(X\wedge gY,gZ)=\Omega(gX,gY,gZ)=n(X\wedge Y,Z)=n(g(X\wedge Y),gZ)
$$
if $g\in H_X$, and of $n$ nondegenerate.
\end{remark}

Now, we can almost exhibit a proof  of the theorem, 
using arguments which are a mixture of  \cite{ultimo,Hitchin,Herz}  since the original Engel's proof is difficult to find.
(Many ideas also appear in \cite[Chapter~6]{Harvey}, but the proofs are not complete there.)
The key of the argument is to prove first the next result.    


\begin{proposition}\label{pr_uncrossproduct}
Let $\Omega\in\wedge^3V^*$ be a   3-form on $V=\CC^7$ such that 
the  bilinear symmetric form   associated by Lemma~\ref{le_lanorma} is nondegenerate.
Then one of the related bilinear forms, $n$, satisfies that      
the multiplication 
$
\wedge\colon V\times V\to V
$ defined by 
$$
n(X\wedge Y,Z)=\Omega(X,Y,Z),
$$
  is a cross product. 
\end{proposition}

\begin{proof}  
Take $n=  n_{\Omega,\cB}$  for $\cB$ a basis of $V$ such that $\cB$ is orthonormal for
$ n$ as in Remark~\ref{re_elijon}.  Afterwards we will replace $n$ with a suitable nonzero multiple. We are going  to prove that $\wedge$ is a cross product relative to $n$, but, as $\Omega$ is alternating, 
clearly $n(X\wedge Y,X)=n(X\wedge Y,Y)=0$, and our concrete  purpose will be to prove the second equation in Definition \ref{def_crossproduct}. 
  
Let $X\in V$ be a nonisotropic element, that is, $n(X)\ne0$. Thus $V$ decomposes as a direct sum of $\CC X$ and $X^\perp$. Let us prove that for any $Y\in V$ orthogonal to $X$, then $X\wedge(X\wedge Y)=-n(X)Y$. Consider 
the linear map $f\colon X^\perp\to X^\perp$ given by $f(Y)=X\wedge Y$, which is well defined since $n(X\wedge Y,X)=0$. It is linear and skew-adjoint,\footnote{Our purpose is to check that $f^2$ is diagonalizable, but this cannot be concluded of the fact of being $f^2$ selfadjoint, since we are working on $\CC$. We need some extra-information. For $f$ we have   eigenvectors, but, as they are necessarily isotropic, there is no guarantee of an invariant complementary subspace.}
 since $n(f(Y),Z)=\Omega(X,Y,Z)=-\Omega(X,Z,Y)=-n(Y,f(Z))$.

 Let us check first that $0$ is not an eigenvalue, that is, $f$ bijective.
Indeed, if $f(Y)=0$ for $Y\ne0$, then $X\wedge Y=0$, that is, $\Omega(X,Y,-)=0$.  
As $X$ and $Y\in X^\perp$ are linearly independent, complete $\{X,Y\}$  to a basis $\cB'=\{ b_i\}_{i=1}^7$ of $V$ such that $b_1=X$ and $b_2=Y$. As $X$ is not isotropic, $0\ne  n_{\Omega,\cB'}(X,X)$ (the isotropic cone does not depend on the choice of the basis, by Lemma~\ref{le_lanorma}a)), so that there is $\sigma\in S_7$ such that all $\Omega(X,b_{\sigma(1)},b_{\sigma(2)})$,
$\Omega(X,b_{\sigma(3)},b_{\sigma(4)})$ and $\Omega(b_{\sigma(5)},b_{\sigma(6)},b_{\sigma(7)})$ are nonzero complex numbers. Hence the index $1\in\{\sigma(5),\sigma(6),\sigma(7)\}$. But $2$ has to be another of the six remaining indices, so that  $\Omega(X,Y,b_i)\ne 0$ for some $i$, a contradiction.  

Also, since $f\in\frso(X^\perp,n)$,
 \begin{equation}\label{eq_juegon}
 n(Y,(f+\lambda\id)(Z))=-n((f-\lambda\id)(Y),Z)
 \end{equation}
  for all $Y,Z\in X^\perp$.

We start with the Jordan decomposition of the endomorphism: 
$$
X^\perp=\ker(f-\lambda_1\id)^{s_{\lambda_1}}\oplus\dots
\oplus\ker(f-\lambda_k\id)^{s_{\lambda_k}}
$$
 and try to prove that all $s_{\lambda_i}$'s are $1$. Note that $\ker(f-\lambda\id)^s$ is orthogonal to $\ker(f-\alpha\id)^{t}$   if $\lambda+\alpha\ne0$. Indeed, for $0\ne Y\in \ker(f-\lambda\id)^s$, 
$$
0=n(Z,(f-\lambda \id)^s(Y)) \stackrel{\eqref{eq_juegon}}=(-1)^sn((f+\lambda\id)^s(Z),Y) 
$$
so that $Y$ is orthogonal to $\textrm{Im}(f+\lambda\id)^s$. But $\ker(f-\alpha\id)^{t}\subset \textrm{Im}(f+\lambda\id)^s$, because $ (f+\lambda\id)^s\vert_{\ker(f-\alpha\id)^{t}}$ is a bijective endomorphism (if $g=(f+\lambda\id)^s$, $W=\ker(f-\alpha\id)^{t}$, then $g(W)\subset W$ since $g$ commutes with $(f-\alpha\id)^{t}$ so $g$ is well defined, but
$\ker(f-\alpha\id)^{t}\cap\ker(f+\lambda\id)^s=0$ for $\lambda+\alpha\ne0$, which gives the injectivity and hence $W=g(W)\subset \textrm{Im}(g)$).

The nondegeneracy of $n$ implies now that $s_{\lambda}$ coincides with  $s_{-\lambda}$ for any eigenvalue $\lambda$, since the corresponding blocks are dual for $n$. The only possibilities are $s_{\lambda}\in\{1,2,3\}$. Let $s$ be the maximum of $s_{\lambda}$ for all $\lambda $ eigenvalue.\smallskip

$\boxed{s=3}$\ \ 
Take $u_3\in\ker(f-\lambda\id)^3\setminus \ker(f-\lambda\id)^2$.
Take $u_i=(f-\lambda\id)^{3-i}(u_3)$ for $i=1,2$. Choose any $v_1\in\ker(f+\lambda\id)^3$ such that $n(u_1,v_1)=1$.
It is not difficult to check that $v_1\notin\ker(f+\lambda\id)^2$.
Take $v_2=-(f+\lambda\id)v_1$ and $v_3=-(f+\lambda\id)v_2$.
Thus $\{u_1,u_2,u_3,v_1,v_2,v_3\}$ is a basis of $X^\perp$ relative to which the matrix of $f$ is
$$ \left(
\begin{array}{cccccc}
 \lambda  & 1 & 0 & 0 & 0 & 0 \\
 0 & \lambda  & 1 & 0 & 0 & 0 \\
 0 & 0 & \lambda  & 0 & 0 & 0 \\
 0 & 0 & 0 & -\lambda  & 0 & 0 \\
 0 & 0 & 0 & -1 & -\lambda  & 0 \\
 0 & 0 & 0 & 0 & -1 & -\lambda  \\
\end{array}
\right)=\left(
\begin{array}{cc}J_1&0\\0&-J_1^t
\end{array}
\right).$$
Besides the matrix of $n$ relative to such basis is $\left(
\begin{array}{cc}0&I_3\\I_3&0
\end{array}
\right)$: 

\noindent On one hand, $n(u_1,v_i)=n(u_3,(f+\lambda\id)^2(v_i))=n(u_3,0)=0$ if $i=2,3$;
also $n(u_2,v_3)=n(u_3,(f+\lambda\id)(v_3))=n(u_3,0)=0$,
and analogously $n(v_i,u_j)=0$ if $j>i$. On the other hand, $n(u_1,v_1)=n((f-\lambda\id)(u_2),v_1)=n(u_2,-(f+\lambda\id)(v_1))=n(u_2,v_2)$ and similarly it is equal to $n(u_3,v_3)$.

Now the Lie algebra $\frh_X'\subset \frso(X^\perp,n)$ of the Lie group $H_X'$ in Eq.~\eqref{eq_hx}
 in Remark~\ref{re_cotasdimensiones}   has dimension greater than $8$. An element $g\in\frh_X'$ satisfies $g(X)=X$,
 $g\vert_{X^\perp}\in\frso(X^\perp,n)$ and $gf=fg$.
The matrix related to any element $g\in\frso(X^\perp,n)$ relative to the above basis should be a block-matrix 
$\left(
\begin{array}{cc}a&b\\c&-a^t
\end{array}
\right)$
with $b$ and $c$ skew-symmetric matrices of size $3$ ($\dim\frso(6)=15$), so commuting with $f$ is equivalent to 
\begin{equation}\label{eq_buscandocontradi}
J_1b+bJ_1^t=0,\ J_1^tc+cJ_1=0,\ aJ_1=J_1a.
\end{equation}
By making these computations, we get $b=c=0$ and $a\in\langle\{I_3,J_1,J_1^2\}\rangle$. This means that $\dim\frh_X'=3$, a contradiction with the fact of containing a subalgebra of dimension greater than 8. That is, $s=3$ is not a true possibility.\smallskip

$\boxed{s=2}$\ \ 
With analogous arguments to the previous case, it is possible to find a basis of $X^\perp$ relative to which   
the matrix of $f$ is 
$$ \left(
\begin{array}{cccccc}
 \alpha  & 0 & 0 & 0 & 0 & 0 \\
 0 & \lambda  & 1 & 0 & 0 & 0 \\
 0 & 0 & \lambda  & 0 & 0 & 0 \\
 0 & 0 & 0 & - \alpha  & 0 & 0 \\
 0 & 0 & 0 & 0 & -\lambda  & 0 \\
 0 & 0 & 0 & 0 & -1 & -\lambda  \\
\end{array}
\right)=\left(
\begin{array}{cc}J_2&0\\0&-J_2^t
\end{array}
\right),$$
($\alpha$ and $\lambda$ nonzero complex numbers) and the matrix of $n$   is $\left(
\begin{array}{cc}0&I_3\\I_3&0
\end{array}
\right)$. Again this case does not occur, since any element $g\in\frh_X'$ would have as a matrix relative to such basis 
$\left(
\begin{array}{cc}a&b\\c&-a^t
\end{array}
\right)$
with $b$ and $c$ skew-symmetric $3\times 3$ matrices satisfying the relations in
\eqref{eq_buscandocontradi} for  $J_2$. This forces $b$ and $c$ to be zero, and $a$ commuting with $J_2$ makes
$$ 
a=\left(
\begin{array}{ccc}a_{11}&0&a_{13}\\a_{21}&a_{22}&a_{23}\\0&0&a_{22}
\end{array}
\right) \textrm{if $\alpha=\lambda$},\qquad
a=\left(
\begin{array}{ccc}a_{11}&0&0\\0&a_{22}&a_{23}\\0&0&a_{22}
\end{array}
\right)  \textrm{if $\alpha\ne\lambda.$}
$$
In any case, $\dim\frh_X'\le5$ cannot be greater than 8. \smallskip

$\boxed{s=1}$\ \ 
 In the third case (a posteriori the only possible case), we can find $\{u_1,u_2,u_3,v_1,v_2,v_3\}$  a basis of $X^\perp$ relative to which the matrix of $f$ and $n$ are, respectively,
$$ \left(
\begin{array}{cccccc}
 \lambda_1  & 0 & 0 & 0 & 0 & 0 \\
 0 & \lambda_2  & 0 & 0 & 0 & 0 \\
 0 & 0 & \lambda_3  & 0 & 0 & 0 \\
 0 & 0 & 0 & -\lambda_1  & 0 & 0 \\
 0 & 0 & 0 & 0 & -\lambda_2  & 0 \\
 0 & 0 & 0 & 0 & 0 & -\lambda_3  \\
\end{array}
\right),\qquad\left(
\begin{array}{cc}0&I_3\\I_3&0
\end{array}
\right).$$
We can assume (changing if necessary $u_i$ with $v_i$)
that $\lambda_1=\lambda_2=\lambda_3$. Indeed,  if $\lambda_1\notin\{\pm\lambda_2,\pm\lambda_3\}$,
any $g\in\frh_X'$ commutes with $f$ and leaves the $f$-eigenspaces invariant. In particular, 
$g$ leaves $\langle u_1,v_1\rangle$ and   $\langle u_2,v_2,u_3,v_3\rangle$ invariant (again its related matrix is a block-matrix) and, as $n$ does not degenerate in those subspaces, $g\vert_{X^\perp}\in\frso(2)\oplus\frso(4)$. As this space has dimension $1+6=7$,   we would again   obtain a  contradiction. 

Then the matrix of $f^2$ in this basis is $\lambda^2 I_6$ (for $\lambda=\lambda_1$), which means that
$$
X\wedge(X\wedge Y)=f^2(Y)=\lambda^2 Y
$$
for all $Y\in X^\perp$. 
We would like to prove   $\lambda^2=- n(X)$, and, although it is not true,  we will check that  $\lambda^2/n(X)$ does not depend on the chosen nonisotropic $X$.
On one hand, consider the basis $\cB_1=\{X,u_1,u_2,u_3,v_1,v_2,v_3\}$ of $V$ and let us compute $n_{\Omega,\cB_1}(X,X)$. As
$\Omega(X,u_i,v_i)=n(X\wedge u_i,v_i)=n(f(u_i),v_i)=n(\lambda u_i,v_i)=\lambda$ but $\Omega(X,u_i,v_j)=0=\Omega(X,u_i,u_j)$ if $i\ne j$, then
$$
n_{\Omega,\cB_1}(X,X)=2\cdot2\cdot 3!\cdot3!  \Omega(X,u_1,v_1)\Omega(X,u_2,v_2)\Omega(X,u_3,v_3)=144\lambda^3.
$$
On the other hand,  $\cB_2=\{\frac{X}{\sqrt{n(X)}},\frac{u_1+v_1}{\sqrt2},\frac{u_1-v_1}{\mathbf{i}\sqrt2},
\frac{u_2+v_2}{\sqrt2},\frac{u_2-v_2}{\mathbf{i}\sqrt2},
\frac{u_3+v_3}{\sqrt2},\frac{u_3-v_3}{\mathbf{i}\sqrt2} \}$ is
an  orthonormal basis for $n$, so that its change of basis to the original basis $\cB$ has determinant $\pm1$ and 
$n_{\Omega,\cB_2}(X,X)=\pm n_{\Omega,\cB}(X)=\pm n(X)$. The last ingredient is
$$
\det C_{\cB_1\cB_2}= \frac{1}{\sqrt{n(X)}}\left(\det\left(\begin{array}{cc}\frac{1}{\sqrt2}&\frac{1}{\mathbf{i}\sqrt2}\\\frac{1}{\sqrt2}&\frac{-1}{\mathbf{i}\sqrt2}
\end{array}
\right) \right)^3=\ \frac{-1}{\mathbf{i}^3\sqrt{n(X)}}, 
$$ 
so $144\lambda^3=\det C_{\cB_2\cB_1}\,n_{\Omega,\cB_2}(X)=\pm{\mathbf{i}^3\sqrt{n(X)}} n(X)$
and $\left(\frac{\lambda}{\mathbf{i}\sqrt{n(X)}}\right)^3=\pm\frac1{12^2}$. Hence $\lambda^6=-\frac1{12^4}n(X)^3$ and $\lambda^2=-\epsilon 12^{-4/3}n(X)$ for $\epsilon$ a cubic root of the unit.
We have proved that for each nonisotropic $X$, there is $\epsilon_X\in\CC$, $\epsilon_X^3=1$  such that  $X\wedge(X\wedge Y)= -\epsilon_X 12^{-4/3}n(X) Y$ for all $Y\in X^\perp$. 
The map $\{X\in V\mid n(X)\ne0\}\to\{1,\frac{-1\pm\mathbf{i}\sqrt3}{2}\}\subset\CC$ which sends $X\mapsto \epsilon_X$ is continuous, and the first set is connected and the second one is discrete, so that the map is   constant   and $\epsilon=\epsilon_X$ does not depend on $X$. 

If we replace $n$ with $n'=sn$, then the wedge product defined by $n'(X\wedge' Y,Z)=\Omega(X,Y,Z)$ is $\wedge'=\frac1s\wedge$, so 
$$
X\wedge'(X\wedge' Y)= - \frac\epsilon{s^212^{4/3}} n(X) Y= - \frac\epsilon{s^312^{4/3}} n'(X) Y=-n'(X) Y,
$$
choosing $s\in\CC$ such that $s^312^{4/3}=\epsilon$. 
For the rest of the proof, we consider this $n'$ as our new norm $n$.

Therefore, for any nonisotropic $X$, it holds
\begin{equation}\label{eq_mainidentity}
X\wedge(X\wedge Y)= n(X,Y)X-n(X) Y
\end{equation}
 for all $Y\in V$, since it is true if $Y\in X^\perp$, also if $Y=X$ and this expression is linear in $Y$. It is not difficult to conclude that 
 the identity \eqref{eq_mainidentity}
holds for any $X\in V$ by using only linear algebra arguments, but an argument using continuity is faster here, since $\{X\in V\mid n(X)\ne0\}$ is an open dense set. Now, we apply $n(-,Y)$ to Eq.~\eqref{eq_mainidentity}
to get
$$\begin{array}{ll}
n(X,Y)^2-n(X)n(Y)&=n(n(X,Y)X-n(X) Y,Y)\\&=n( X\wedge(X\wedge Y),Y) 
 =-n(X\wedge Y,X\wedge Y),
\end{array}
$$
and hence we infer  that $\wedge$ is a cross product.\smallskip
\end{proof}

\begin{proof} (Of Theorem~\ref{teo_casoC}.)
Take $n$ given by Proposition~\ref{pr_uncrossproduct} such that the related product $\wedge\colon V\times V\to V$ is a cross product. 
Linearizing\footnote{
That is,   applying three times  Eq.~\eqref{eq_mainidentity}
to $(X+Y)\wedge  ((X+Y)\wedge  Z)-X\wedge  (X\wedge  Z)-Y\wedge  (Y\wedge Z)$.}
 Eq.~\eqref{eq_mainidentity}, we get 
\begin{equation}\label{eq_linealizada}
X\wedge  (Y \wedge  Z)+Y\wedge  (X\wedge   Z)=n(X,Z)Y+n(Y,Z)X-2n(X,Y)Z
\end{equation}
for any $X,Y,Z\in V$. 

Take a nonisotropic vector $X\in V$ such that $n(X)=-1$. We have checked that
$$X^\perp=W_1\oplus W_{-1},\quad n(W_1)=n(W_{-1})=0
$$
for the eigenspaces $W_{\pm1}=\{Y\in X^\perp \mid X\wedge Y=\pm Y\}$.
Observe that, if $Y\in W_1$, $Z\in W_{-1}$, then 
\begin{equation}\label{eq_multiplic}
Y\wedge Z=-n(Y,Z)X.
\end{equation}
Indeed, by Eq.~\eqref{eq_linealizada}, 
$X\wedge  (Y \wedge  Z)-Y\wedge  Z=n(Y,Z)X$. Changing the role of $Y$ and $Z$ in Eq.~\eqref{eq_linealizada},
we also get 
$X\wedge  (Z \wedge  Y)+Z\wedge  Y=n(Z,Y)X$ (recall $X\wedge Y=Y$ and  $X\wedge Z=-Z$). Summing both identities, we have
$-2Y\wedge Z=2n(Y,Z)X.$ 

In particular, it holds $\Omega( W_1, W_1, W_{-1})=n(W_1\wedge W_{-1},W_1)\subset n(X,W_1)=0$. 
If $\Omega(X^\perp,X^\perp,X^\perp)=0$, take $Y\in X^\perp$ nonisotropic, which will satisfy $L_Y(X^\perp)$ orthogonal to $X^\perp$, for $L_Y=Y\wedge-$. But $L_Y\colon Y^\perp\to Y^\perp$ is bijective (recalling the previous proof), so $L_Y(Y^\perp\cap X^\perp)$ is a 5-dimensional vector subspace of $(X^\perp)^\perp$; a contradiction. Thus 
 $\Omega(X^\perp,X^\perp,X^\perp)\ne0$, what implies that either  $\Omega( W_1, W_1, W_1)\ne0$ or  $\Omega( W_{-1}, W_{-1}, W_{-1})\ne0$. But one of the assertions implies the other one.
For instance, take $Y_1,Y_2,Y_3\in W_1$ such that $t=\Omega(Y_1,Y_2,Y_3)\ne0$. Take $Z_i=s Y_{i+1}\wedge Y_{i+2}$ (indices modulo 3) for some $0\ne s\in\CC$ which we determine next. Thus,
\begin{itemize}
\item[i)] $Z_i$  is orthogonal to $X$ since $n(X,  Y_{i+1}\wedge Y_{i+2})=n(X\wedge Y_{i+1},Y_{i+2})\subset n(W_1,W_1)=0$;
\item[ii)] $Z_i\in W_{-1}$ since $X\wedge (Y_{i+1}\wedge Y_{i+2})\stackrel{\eqref{eq_linealizada}\& i)}{=}-Y_{i+1}\wedge (X \wedge Y_{i+2})=-Y_{i+1}\wedge Y_{i+2};$
\item[iii)] $n(Y_i,Z_i)=s\Omega(Y_i,Y_{i+1},Y_{i+2})=st$;
\item[iv)] If $i\ne j$, $n(Y_i,Z_j)=s\Omega(Y_i,Y_{j+1},Y_{j+2})=0$ (some repeated index).
\end{itemize}
In particular, $Z_1,Z_2,Z_3$ is a basis of $W_{-1}$. Besides,
\begin{itemize}
\item[v)] As  $Z_3\wedge (Y_2\wedge Y_3)=-Y_2\wedge (Z_3\wedge Y_3)+n(Z_3,Y_3)Y_2+0=-Y_2\wedge n(Y_3,Z_3)X+n(Y_3,Z_3)Y_2=2n(Y_3,Z_3)Y_2$, then
$$\begin{array}{rl}\Omega(Z_1,Z_2,Z_3)&=s\Omega(Y_2\wedge Y_3,Z_2,Z_3)=sn(Z_3\wedge (Y_2\wedge Y_3),Z_2)\\&=s2n(Y_3,Z_3)n(Y_2,Z_2)=2s^3t^2.\end{array}
$$
\end{itemize}
Thus, also $\Omega( W_{-1}, W_{-1}, W_{-1})\ne0$. For $t=-4$ and $s=\frac12$, we get ($2s^3t^2=4$)
\begin{equation}\label{transisplit}
\begin{array}{ll}
\Omega(X,Y_i,Z_i)=-2,\qquad\qquad& n(X,X)=-1,\\
\Omega(Y_1,Y_2,Y_3)=-4,&n(Y_i,Z_i)=-2,\\
\Omega(Z_1,Z_2,Z_3)=4.&
\end{array}
\end{equation}
Hence, if we take $\varphi\colon V\to V$ the isomorphism given by $\varphi(E_0)=X$, $\varphi(E_i)=Y_i$ and $\varphi(F_i)=Z_i$, it is clear from Remark~\ref{re_3formaconcreta}   that $\varphi^*\Omega_0=\Omega$. Thus any generic 3-form belongs to the orbit of $\Omega_0$.
\end{proof}

\begin{proof} (Of Theorem~\ref{teo_casoR}.)

Let $\Omega\in\wedge^3V^*$ be a generic 3-form on $V=\RR^7$.
Fix $\cB$ any basis of $V$ and  $n\equiv n_{\Omega,\cB}$ the nondegenerate bilinear symmetric form in Lemma~\ref{le_lanorma} and take  
$
\wedge\colon V\times V\to V
$
the multiplication defined by the nondegeneracy of $n$ as
$
n(X\wedge Y,Z)=\Omega(X,Y,Z).
$
If we complexify, $\Omega^\CC\in\wedge^3(\CC^7)^*$ is again a  generic 3-form, due to the fact that 
$\dim_\CC G_{\Omega^\CC}=\dim_\RR G_\Omega=14$.  Thus Theorem~\ref{teo_casoC}  says that  there is $\alpha\in\CC$ such that   
\begin{equation*}
X\wedge(X\wedge Y)= \alpha(n(X,Y)X-n(X) Y)
\end{equation*}
for all $X,Y\in V$. But $X\wedge(X\wedge Y)$ and $n(X,Y)X-n(X) Y$ both belong to $V$, so that $\alpha\in\RR$.
If we take $s\in\RR$ such that $s^3=\alpha$ and  we replace $n$ with $s n$ and consequently $\wedge$ with $\frac1s\wedge$, we obtain that the new $\wedge$ is a cross product on $V=\RR^7$ relative to the new $n$.\medskip

\boxed{\textrm{Case $n$ is not positive definite}}
Take $X\in V$ such that $n(X)=-1$. Thus $X\wedge(X\wedge Y)=Y$ for all $Y\in X^\perp$ and the endomorphism $f\colon X^\perp\to X^\perp$ given by $f(Y)=X\wedge Y$ satisfies $f^2=\id$. Its minimum polynomial is $x^2-1$, so that it is diagonalizable with eigenvalues $\pm1$, that is,  $X^\perp=W_1\oplus W_{-1}$ for the eigenspaces $W_{\pm1}=\{Y\in X^\perp \mid X\wedge Y=\pm Y\}$. Besides $n(W_1)=n(W_{-1})=0
$ ($f$ skew-adjoint for $n$ and the eigenspaces are totally isotropic). Now the proof continues exactly in the same way than the previous one until proving that, after a change of basis, $\Omega=\Omega_0$.

(In particular, this proves that there is a basis relative to which $n$ is given by \eqref{eq_lan} and the signature is just $(4,3)$.)\smallskip

\boxed{\textrm{Case $n$ is  positive definite}}
Now let us take $X\in V$ such that $n(X)=1$ and, as before, the endomorphism $f\colon X^\perp\to X^\perp$ given by $f(Y)=X\wedge Y$, which satisfies $f^2=-\id$ and it is skew-adjoint (in particular without any real eigenvalue). 
We choose
\begin{itemize}
\item $X_1\in X^\perp$ such that $n(X_1)=1$, $Y_1=f(X_1)$;
\item $X_2\in \langle X,X_1,Y_1\rangle^\perp$ such that $n(X_2)=1$, $Y_2=f(X_2)$;
\item $X_3=X_1\wedge X_2$,  $Y_3=f(X_3)$.
\end{itemize}
Note that $n(Y_1,Y_1)=n(f(X_1),f(X_1))=-n(X_1,f^2(X_1))=n(X_1)$ and $n(X_1,Y_1)=n(X_1,X\wedge X_1)=0$. Similarly $n\vert_{ \langle  X_s,Y_s\rangle}\equiv I_2$ for $s=2,3$ and also $Y_2$ belongs to $ \langle X,X_1,Y_1\rangle^\perp$, since $n(Y_2,X_1)=-n(X_2,Y_1)$ and $n(Y_2,Y_1)=n(X_2,X_1)$.
We are checking now that $X_3\in \langle X,X_1,Y_1,X_2,Y_2\rangle^\perp$. Indeed, $n(X_3,Z)=\Omega(X_1,X_2,Z)=0$ for $Z=X_1,X_2$ and also for $Z=X$, since $\Omega(X,X_1,X_2)=n(Y_1,X_2)=0$. Besides $X\wedge X_3\stackrel{\eqref{eq_linealizada}}{=}-X_1\wedge(X\wedge X_2)+0=-X_1\wedge Y_2$ is perpendicular to $X_1$, so that $n(X_3,Y_1)=n(X_3,X\wedge X_1)=-n(X\wedge X_3,X_1)=0$.
Analogously, $X\wedge X_3=-X\wedge(X_2\wedge X_1)\stackrel{\eqref{eq_linealizada}}{=}X_2\wedge(X\wedge X_1)$ is orthogonal to $X_2$ so that  $n(X_3,Y_2)=-n(X\wedge X_3,X_2)=0$.

As $n(X_3)=n(X_1)n(X_2)=1$, this implies that 
$\cB_{X^\perp}=\{X_1,Y_1,X_2,Y_2,X_3,Y_3\}$ is an orthonormal basis for $n\vert_{X^\perp}$ such that the matrix of $f$ is 
$$ \left(
\begin{array}{cccccc}
 0  & -1 & 0 & 0 & 0 & 0 \\
 1 & 0  & 0 & 0 & 0 & 0 \\
 0 & 0 & 0 & -1 & 0 & 0 \\
 0 & 0 & 1 & 0  & 0 & 0 \\
 0 & 0 & 0 & 0 & 0  & -1 \\
 0 & 0 & 0 & 0 & 1 & 0  \\
\end{array}
\right).$$
We are going to prove that

\begin{equation}\label{eq_omegandefinida1}
\begin{array}{rl}
\Omega=&X^*\wedge(X_1^*\wedge Y_1^*+X_2^*\wedge Y_2^*+X_3^*\wedge Y_3^*)\\+&
X_1^*\wedge X_2^*\wedge X_3^*-X_1^*\wedge Y_2^*\wedge Y_3^*
-Y_1^*\wedge X_2^*\wedge Y_3^*-Y_1^*\wedge Y_2^*\wedge X_3^*.
\end{array}
\end{equation} 
With that purpose, let us check first that
\begin{equation}\label{eq_tec}
\begin{array}{l}
X_s\wedge Y_s=X, \\
X_s\wedge Y_k=Y_s\wedge X_k \textrm{ if }s\ne k,\\
X_s\wedge X_k=-Y_s\wedge Y_k  \textrm{ for all }s,k.
\end{array}
\end{equation}
For instance, we can use the proof of Theorem~\ref{teo_casoC} by complexifying, since $n(\mathbf{i}X)=-1$, 
$W_{1}=\langle X_s+\mathbf{i}Y_s\mid s=1,2,3\rangle$ and $W_{-1}=\langle X_s-\mathbf{i}Y_s\mid s=1,2,3\rangle$.
According to Eq.~\eqref{eq_multiplic}, $(X_s+\mathbf{i}Y_s)\wedge(X_k-\mathbf{i}Y_k)=-n(X_s+\mathbf{i}Y_s,X_k-\mathbf{i}Y_k)\mathbf{i}X$. The real part gives $X_s\wedge X_k+Y_s\wedge Y_k =(n(Y_s,X_k)+n(X_s,Y_k))X=0$ and the imaginary part
gives $Y_s\wedge X_k-X_s\wedge Y_k=-(n(X_s,X_k)+n(Y_s,Y_k))X=-2\delta_{sk}X$, and hence Eq.~\eqref{eq_tec} holds.

If $Z_i,S_i,T_i\in\{X_i,Y_i\}$, we would like to know  $\Omega(Z_i,S_j,T_k)$ for all $i,j,k$, but it is always $0$ except if $(i,j,k)$ is a permutation of $(1,2,3)$, because $\Omega(X_i,X_i,-)=0$ (evaluating in $X^\perp$) and $\Omega(X_i,Y_i,-)=n(X,-)=0$. 
Then we can assume  $i=1$, $j=2$, $k=3$. Thus, by using repeatedly Eq.~{\eqref{eq_tec}},
\begin{itemize}
\item[$\circ$] $\Omega(X_1,X_2,X_3)=n(X_3)=1$;
\item[$\circ$] $\Omega(X_1,Y_2,Y_3)=n(X_1,Y_2\wedge Y_3)=-n(X_1,X_2\wedge X_3)=-1$;
\item[$\circ$] $\Omega(X_1,X_2,Y_3)=n(X_3,Y_3)=0$;
\item[$\circ$] $\Omega(X_1,Y_2,X_3)=n(X_1,Y_2\wedge X_3)=n(X_1,X_2\wedge Y_3)=0$;
\item[$\circ$] $\Omega(Y_1,X_2,X_3){=}\Omega(X_1,Y_2,X_3)=0$;
\item[$\circ$] $\Omega(Y_1,X_2,Y_3)=\Omega(X_1,Y_2,Y_3)=-1$;
\item[$\circ$] $\Omega(Y_1,Y_2,X_3)=-\Omega(X_1,X_2,X_3)=-1$;
\item[$\circ$] $\Omega(Y_1,Y_2,Y_3)=-\Omega(X_1,X_2,Y_3)=0$;
\end{itemize}
that proves Eq.~\eqref{eq_omegandefinida1}.
Hence we have shown that, if we consider the basis of $V$ given by   $\{e_i\}_{i=1}^7=\{X_1,X_2,X_3,Y_1,Y_2,Y_3,X\}$ and $\{e_i^*\}_{i=1}^7$ the dual basis in $V^*$ ($e_i^*(e_j)=\delta_{ij} $), then our $3$-form $\Omega$ is just $\Omega_1$ in Eq.~\eqref{eq_omegandefinida}.\smallskip

  This would complete the proof that there are at most two orbits, related respectively with a norm of signature $(4,3)$ and $(0,7)$. For completing that the number of orbits is just two, we have only to be sure that $\Omega_1$ as in Eq.~\eqref{eq_omegandefinida} is generic.
It suffices to see that the complexification $\Omega_1^\CC$ is generic. 
But, by repeating the proof of   Theorem~\ref{teo_casoC},
we conclude that the wedge-product defined by $\Omega_1^\CC$ and $n^\CC$ is a cross product for a suitable multiple of $n^\CC$ and then $\Omega_1^\CC$ is in the orbit of $\Omega_0$, in other words, $\Omega_1^\CC$ is generic 
and $\Omega_1 $ is too ($14=\dim_\CC G_{\Omega_1^\CC}=\dim_\RR G_{\Omega_1}$). Although   the hypothesis of Theorem~\ref{teo_casoC} was $\Omega$ being generic, the only point we really needed throughout the proof  was that the related bilinear form was nondegenerate to apply then Proposition~\ref{pr_uncrossproduct}.
\end{proof}

\begin{remark}\label{re_reciproco}
Following the   arguments in the last paragraph, the converse of Lemma~\ref{le_lanorma}b) is also true and the nondegeneracy of $n_{\cB,\Omega}$ implies that $\Omega$ is generic.

Besides, as a consequence of Theorem~\ref{teo_casoC}, in the complex case 
$\cO_{\Omega_0}= \{\omega\in \wedge^3V^*:\det\varphi_\omega\ne 0\},$
so that the complementary is the set of zeros of a polynomial (closed in the Zarisky topology) and 
$\cO_{\Omega_0}$ is also open (hence dense) in the Zarisky topology (topology coarser than the usual one). 
\end{remark}   

\begin{remark}\label{re_S2enwedge3}
We add just one brief comment   on representation theory   
related again to 3-forms and quadratic forms.
For   $\Omega_0$ our generic 3-form on $V=\FF^7$ with $\cB_c=\{b_i:i=1,\dots,7\} =\{E_0,E_1,E_2,E_3,F_1,F_2,F_3\}$ our basis of $V $, the map given by
$$
\wedge^3V^*\to S^2V^*,\qquad \Omega\mapsto n_\Omega\colon V\times V\to\FF 
$$
being
$$
n_ \Omega (X,Y)= \sum_{ \sigma\in S_7  }
(-1)^{\textrm{sg}\sigma} \Omega_0(X,b_{\sigma({1})},b_{\sigma({2})})\Omega_0(Y,b_{\sigma({3})},b_{\sigma({4})})
 \Omega(b_{\sigma({5})},b_{\sigma({6})},b_{\sigma({7})}),
$$
is 
$G_2$-invariant and $\frg_2$-invariant, where by $G_2$ we again mean $G_{\Omega_0}^0=G_{\Omega_0}\cap\SL(V)=\{f\in\SL_7(\FF):f\cdot \Omega_0=\Omega_0\}$. Written using the notation  in Section~\ref{se:mod}, this means that the $\frg_2$-module $S^2V^*\equiv V(2\omega_1)\oplus V(0)$ should be a submodule of  the $\frg_2$-module $\wedge^3V^*\equiv V(2\omega_1)\oplus   V(\omega_1)\oplus V(0)\equiv S^2V^*\oplus V$, as  does occur. \end{remark}

\subsection{Brief comments on $G_2$-structures}\smallskip

A seven-dimensional smooth manifold M is said
to admit a $G_2$-structure if there is a reduction of the
structure group of its frame bundle from $\GL_7(\mathbb R)$
to the group $G_2$, viewed naturally as a subgroup of
$\SO(7)$. This is equivalent to fix a 3-form $\omega$ on $M$ such that $\omega_p\in\wedge^3(T_pM)^*$ is generic for all $p\in M$.  So each tangent space $T_pM$ can be identified with the imaginary octonions (zero trace octonions)
because $\omega_p$ endows $T_pM$  with a cross product as in Proposition~\ref{pr_uncrossproduct}.  Also a scalar product and an orientation in  $T_pM$  are determined 
as a consequence of Lemmas~\ref{le_lanorma} and  \ref{le_parteconexa} respectively.  Note that the signature of the scalar product does not depend on the choice of a basis in $T_pM$, and moreover, for connected manifolds, it does not depend on the point $p\in M$: under any identification of $T_pM$ with $\mathbb R^7$, either $\omega_p$ is in the orbit of $\Omega_1$ for all $p$ or in the orbit of $\Omega_0$ for all $p\in M$, and consequently $T_pM$ can be identified with the imaginary part of either the 
division octonion algebra or the split octonion algebra respectively. Now all this goes  from the vector spaces to the manifold, 
what explains in part why the two real forms of $G_2$
are of such eminent importance in differential geometry: any manifold with a reduction to one of the real forms of $G_2$ carries a
Riemannian metric or respectively a semi-Riemannian metric of signature $(3,4)$, jointly with an
orientation.


%

An introduction to modern
$G_2$-geometry and its relevance in   superstring theory  can be found in \cite{AgriG2,agrilecturenotes}.
An accessible paper about $G_2$-manifolds (that is, $\nabla^g\omega=0$)
is \cite{Karigiannis}.


\section{$G_2$ and the spheres in dimension 6}\label{se_esferadim6}

The compact (resp. split) group of type $G_2$ acts transitively on the six-dimensional sphere (resp. on certain pseudohyperbolic space). It is easy to understand this relationship by going on with the arguments in the previous section.

\subsection{Case the norm is  positive definite}\smallskip

First, consider $\Omega_1\in\wedge^3(\RR^7)^*$ in Eq.~\eqref{eq_omegandefinida}, the corresponding positive definite norm $n$ and the related cross product $\times$ in $V=\RR^7$.
If we take, as in Section~\ref{se:octo}, $\OO=\RR\oplus V$  with the product where $1\in\RR$ is a unit
and 
$$
XY:=-n(X,Y)1+X\times Y,
$$
for all $X,Y\in V$, and with the positive definite norm $n\colon \OO\to \RR$ defined by $n(1)=1$, $n(1,V)=0$ and $n\vert_V=n$, we obtain
  that $\OO$ is   an {octonion algebra}, with a proof completely similar to that one of Lemma~\ref{le_octon}, since $\times $ is again a cross product. This algebra $\OO$ is a division algebra ($\cC$ is not), since, for $\alpha\in\RR$ and $X\in V$,
$$
(\alpha+X)(\alpha-X)=(\alpha^2+n(X))1
$$
is nonzero if $\alpha+X\ne0$. 
Sometimes $\OO$ is called the \emph{(division) octonion} algebra and $\cC$   the \emph{split octonion} algebra (often denoted by $\OO_s$ in the real case, and by $\OO_s^\CC\cong\OO^\CC$ in the complex one). 

The multiplication in $\OO=\RR\langle1, e_i:1\le i\le 7\rangle$ can be recalled by noting  that $e_i^2=-1$    and that the only nonzero products of basic elements (up to the unit) are
$$\begin{array}{c}
e_1e_4=e_2e_5=e_3e_6=e_7;\\
e_1e_2=e_3; \ e_1e_6=e_5;  \ e_2e_4=e_6; \ e_3e_5=e_4;
\end{array}
$$
taking into account that if $ e_{ i}e_{j }=e_{k }$, also $e_{j }e_{ k}=e_{ i}$ and $e_{k }e_{ i}=e_{ j}$, and $e_ie_j=-e_je_i$ if 
$i\ne j$.  Another easy way of recalling these products is Figure~\ref{fig_fano}, Fano plane: in any of the seven \emph{lines} (the circle is also a line) the basic elements multiply according to the arrows. Thus, each triplet $ \{e_{ i},e_{j },e_{k }\}$ in one of these lines spans a subalgebra isomorphic to $\RR^3$ with the usual cross product. (And of course $ \RR\langle1,e_{ i},e_{j },e_{k }\rangle$ is isomorphic to the -division- quaternion algebra $\HH$ when $e_{ i}$, $e_{j }$ and $e_{k }
$ are the three different elements in one line.)
  
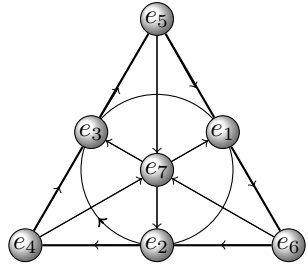
\begin{figure}
 \begin{center}  
    \begin{tikzpicture}
\tikzstyle{point}=[ball color=white, circle, draw=black, inner sep=0.02cm]
\node (v7) at (0,0) [point] {$e_7$};
\draw[->] (0,0) circle (1cm);
 \draw[thick,-<] ({-1/sqrt(2)},{-1/sqrt(2)}) -- ({-1/sqrt(2)+0.001},{-1/sqrt(2)-0.001});
\node (v1) at (90:2cm) [point] {$e_5$};
\node (v2) at (210:2cm) [point] {$e_4$}; 
\node (v4) at (330:2cm) [point] {$e_6$};
\node (v3) at (150:1cm) [point] {$e_3$};
\node (v6) at (270:1cm) [point] {$e_2$};
\node (v5) at (30:1cm) [point] {$e_1$};  
\draw [->, color=black] (v4) -- (v7);
\draw [->, color=black] (v1) -- (v7);
\draw [->, color=black] (v2) -- (v7);
\draw [->, color=black] (v7) -- (v3);
\draw [->, color=black] (v7) -- (v6);
\draw [->, color=black] (v7) -- (v5);
\draw [->, color=black] (v6) -- (230:13mm);\draw [->, color=black] (v4) -- (310:13mm);
\draw [->, color=black] (v1) -- (65:12mm);\draw [->, color=black]  (v3) --  (115:12mm) ; 
\draw [->, color=black] (v5) -- (350:13mm);\draw [->, color=black] (v2) -- (190:13mm);
\draw[-,thick] (v1) -- (v3) -- (v2);
\draw[-,thick] (v2) -- (v6) -- (v4);
\draw[-,thick] (v4) -- (v5) -- (v1);
\draw (v3) -- (v7) -- (v4);
\draw (v5) -- (v7) -- (v2);
\draw (v6) -- (v7) -- (v1);
\end{tikzpicture}
\caption{Fano plane.}\label{fig_fano}
\end{center}
\end{figure}

Then,  
$$
\begin{array}{rl}
\frg_c&:=\{f\in\frgl_7(\RR):\Omega_1(f(X),Y,Z)+\Omega_1(X,f(Y),Z)+\Omega_1(X,Y,f(Z))=0\}
\\&=\{f\in\frgl_7(\RR):f(X\times Y)=f(X)\times Y+X\times f(Y)\quad \forall X,Y\in V\}
\\
&\cong \Der(\OO)=\{d\in\frgl(\OO)\equiv\frgl_8(\RR):d(XY)=d(X)Y+Xd(Y)\}
\end{array}
$$
is a Lie algebra of type $G_2$ (a real form).  %
\begin{corollary}\label{co_conexoscompactos} The group
\begin{equation*}  
\begin{array}{rl}
G_{\Omega_1}:&=\{f\in\GL_7(\RR):\Omega_1(f(X),f(Y),f(Z))=\Omega_1(X,Y,Z)\  \forall X,Y,Z\in V\},\\
&=\{f\in\GL_7(\RR): f(X)\times f(Y)=f(X\times Y) \  \forall X,Y\in V\},
\end{array}
\end{equation*}
is a   connected  Lie group of type $G_2$,  subgroup of $\SO(7)$. 
It is isomorphic -under the obvious extension to $\OO$ making $f(1)=1$- to
$$
\Aut(\OO)=\{f\in\GL(\OO): f(X  Y) =f(X) f(Y)\  \forall X,Y\in \OO\}.
$$
\end{corollary}

We will prove this corollary in Section~\ref{se_conexo}.

Recall, as in Lemma~\ref{le_parteconexa}c), that the group $G_{\Omega_1}\subset\SL(V)$ acts naturally in $V$, and the orbit of an element $X$ of norm 1 is contained in  $\{Y\in\RR^7:n(Y)=1\}\equiv\mathbb{S}^6$. Moreover, that orbit $G_{\Omega_1}\cdot X$ equals $\mathbb{S}^6$, that is, the action of $G_{\Omega_1}$ on the 6-dimensional   sphere is transitive. This is consequence of
  a more general fact that can be deduced from  the proof of Theorem~\ref{teo_casoR}: if we have $\{X,X_1,X_2\}$ and $\{X',X'_1,X'_2\}$ two pairs of three orthonormal vectors (hence linearly independent) such that $\Omega_1(X,X_1,X_2)=0=\Omega_1(X',X'_1,X'_2)$, then there is $g\in G_{\Omega_1}$ such that $g(X)=X'$, $g(X_1)=X'_1$ and $g(X_2)=X'_2$.
  To be precise, $g\in\GL(V)$ is the change of basis from the basis of $V$ given by
  $$\{X,X_1,X_2,X\times X_1,X\times X_2,X_1\times X_2,X\times(X_1\times X_2)\}$$
  to
  \begin{equation}\label{eq_transi}
  \{X',X'_1,X'_2,X'\times X'_1,X'\times X'_2,X'_1\times X'_2,X'\times(X'_1\times X'_2)\},
  \end{equation}
which preserves $\Omega_1$. Now we compute the isotropy group of $X$.


\begin{figure}
      \psfrag{x}{$X$ }\psfrag{TxS}{$T_X\mathbb{S}^2$}\psfrag{v}{$Y$}
\psfrag{Jv}{$J(Y)=X\times Y$}\psfrag{S2}{$\mathbb{S}^2$}
\psfrag{titel}{}
\includegraphics[width=5cm]{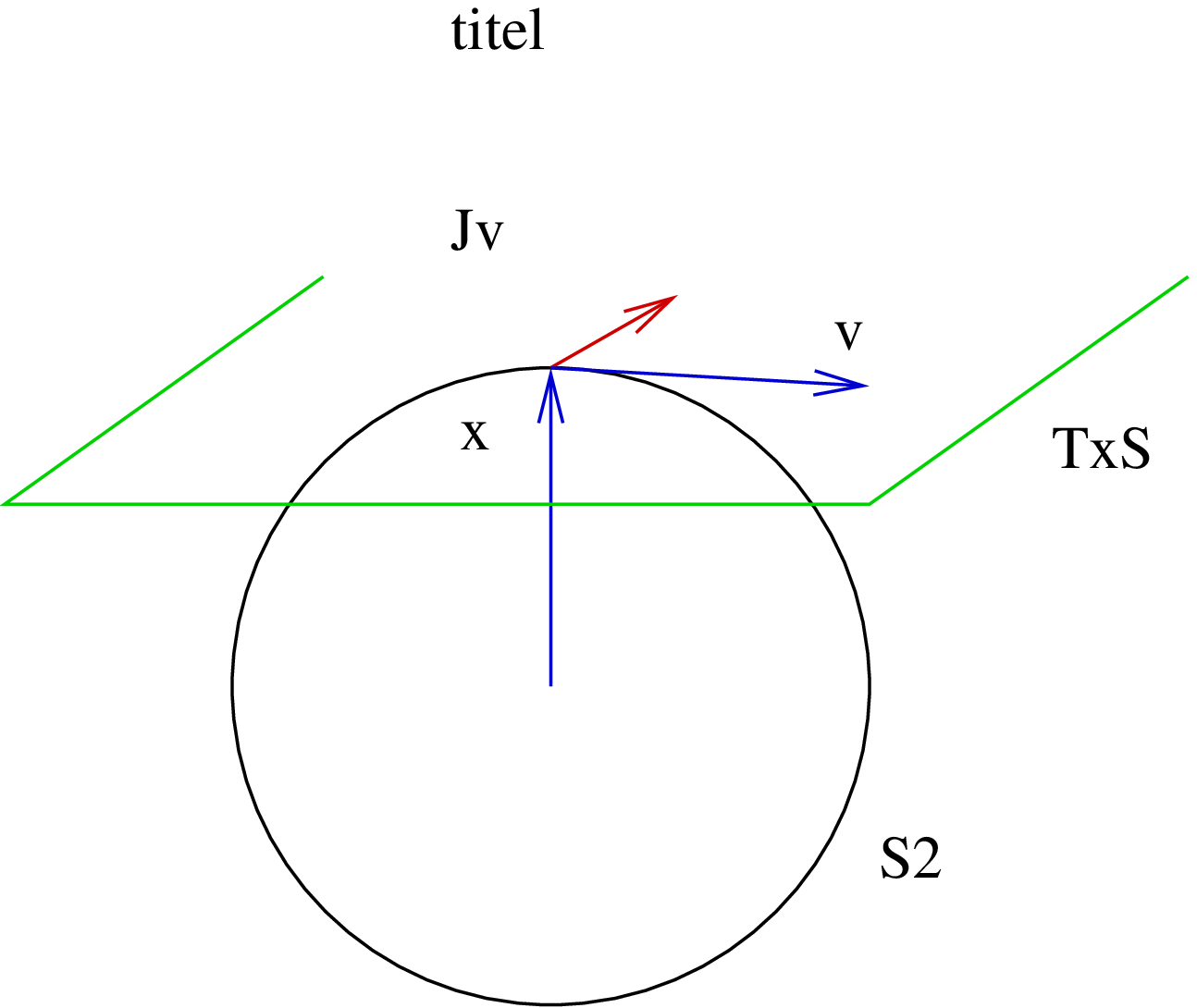}
\caption{Complex structure of  $\mathbb{S}^2$.}
\label{fig_e8}
  \end{figure}


\begin{proposition}\label{le_elsigma}
Consider $X\in\RR^7$ of norm 1 and $\mathbb{S}^6=\{Y\in\RR^7:n(Y)=1\}$. Recall the usual identification $T_X\mathbb{S}^6=X^\perp$. The linear map
\begin{equation}\label{eq_laJ}
J\colon T_X \mathbb{S}^6 \to T_X \mathbb{S}^6,\qquad  J(Y)=X\times Y,
\end{equation}
    satisfies $J^2=-\id$, so that it is a complex structure,  which allows to consider the 6-dimensional real vector space $W=T_X \mathbb{S}^6$ as a 3-dimensional complex space under $\CC\times W\to W$, $\mathbf{i}Y:=J(Y)$. Besides
    $$
\sigma\colon W\times W\to \CC, \qquad\sigma(Y,Z)=n(Y,Z)-\mathbf{i}n(J(Y),Z),
$$
is a nondegenerate Hermitian form.
If $H_X=\{g\in G_{\Omega_1}: g(X)=X\}$ denotes the isotropy group of $X$, the map
\begin{equation}\label{eq_SU3}
H_X\to\mathrm{SU}(W,\sigma),\quad g\mapsto g\vert_{W}
\end{equation}
is well defined and a Lie group isomorphism.   
Therefore,
$$
G_{\Omega_1}/\mathrm{SU}(3)\cong\mathbb{S}^6.
$$
\end{proposition}

The notation $J$ in Eq.~\eqref{eq_laJ} reminds the usual complex structure in Figure~\ref{fig_e8}.
 
\begin{proof}
The map $\sigma$ is $\RR$-linear as $J$ is and $n$ is bilinear. Besides $\sigma$ is $\CC$-linear in the first variable,
$\sigma(\mathbf{i}Y_1,Y_2)=\sigma(J(Y_1),Y_2)=n(J(Y_1),Y_2)-\mathbf{i}n(J^2(Y_1),Y_2)=\mathbf{i}\sigma(Y_1,Y_2)$, since $J$ is an skew-adjoint linear map such that $J^2=-\id$. (Recall that $J$ is our map $f$ in the proof of  Theorem~\ref{teo_casoR}.)
Also, 
$$
\overline{\sigma(Y_1,Y_2)}=n(Y_1,Y_2)+\mathbf{i}n(J(Y_1),Y_2)=n(Y_1,Y_2)-\mathbf{i}n(J(Y_2),Y_1)=\sigma(Y_2,Y_1).
$$
And $\sigma(Y,Y)=n(Y)$ is a real positive number, different from $0$ if $Y\ne0$. Besides, the nondegeneracy of $n$ implies the one of $\sigma$.\smallskip

Now, if $g\in H_X$, then $g\colon V\to V$ is determined by its restriction $g\vert_W$, which proves the injectivity of the map in Eq.~\eqref{eq_SU3}. The facts    $n(g(Y_1),g(Y_2))=n(Y_1,Y_2) $ and $gJ=Jg$ lead to $\sigma(g(Y_1),g(Y_2))=\sigma(Y_1,Y_2) $, so that  $g\vert_W\in \mathrm{U}(W,\sigma)$.

In order to prove that $\det g\vert_{W}=1$, recall that an element in $\GL_3(\FF)$ (for $\FF$   either $\CC$ or $\RR$)
preserving an alternating nonzero $3$-form in $\wedge^3(\FF^3)^*$ has determinant $1$. In our case, we can choose 
$\{X_1,X_2,X_3,Y_1,Y_2,Y_3\}$ an $\RR$-basis of $X^\perp$ such that $X_3=X_1\times X_2$ and $Y_i=X\times X_i=f(X_i)$. Then $\cB_0=\{X_1,X_2,X_3\}$ is a $\CC$-basis of $W$ and $\cB_0'=\{X_1+\mathbf{i}Y_1,X_2+\mathbf{i}Y_2,X_3+\mathbf{i}Y_3\}$ is   a $\CC$-basis of $\ker(F-\id)$ for $F\colon V^\CC\to V^\CC$ given by $F(Y)=\mathbf{i}X\times Y$.
Of course $g^\CC$ preserves this eigenspace since it commutes with $F=\mathbf{i}f^\CC$.
As $\Omega_1^\CC$ is a nonzero alternating form on $\ker(F-\id)$, 
and the matrix of $g\vert_{W}$ relative to $\cB_0$ coincides with the matrix of $g^\CC\vert_{\ker(F-\id)}$ relative to $\cB_0'$, then $\det g\vert_{W}=1$.\footnote{We have dealt with the complexification because $\Omega_1\colon W\times W\times W\to\RR\subset\CC$ is not $\CC$-linear, and if we consider the restriction to $\Omega_1\colon W'\times W'\times W'\to\RR$ for $W'$ the real vector space spanned by $\{ X_1,X_2,X_3\}$, then $\Omega_1$ is an alternating nonzero 3-form, but $g(W')\not\subset W'$.}

\smallskip

So \eqref{eq_SU3} provides a monomorphism of groups and hence $\dim H_X\le\dim \mathrm{SU}(3)=8$. According to Remark~\ref{re_cotasdimensiones}, $\dim H_X\ge8$, what implies  $\dim H_X=8$ and the map \eqref{eq_SU3} is also surjective.
The transitivity of the action gives then $G_{\Omega_1}/\textrm{SU}(W,\sigma)\cong G_{\Omega_1}\cdot X=\mathbb{S}^6$.
  \end{proof}
  \smallskip

Recall that topological properties of $H$, $G$ and $G/H$ are closely related in homogeneous spaces. Thus, 
as a corollary of the above proposition, $G_{\Omega_1}$ is compact and connected, due to \cite[Lemma~11.18]{ONeill}, and simply connected \cite[Proposition~11.17]{ONeill}, because both $\mathrm{SU}(3)$ and $\mathbb{S}^6$ are compact, connected  and  simply connected.

Hence  we will use the notation $G_{2,-14}$ for the connected compact Lie group $G_{\Omega_1}$ and $G_{2,2}$ for the   (real) split group $G_{\Omega_0}$ (also connected as a consequence of Section~\ref{se_casosplit}), notation provided by the signature of the Killing form of the related Lie algebra. (When there is no ambiguity, we will use the notation $G_2$ for $G_{\Omega_0}^+$ the complex connected group of type $G_2$,     and $\frg_2$ for its Lie algebra $L_\CC$.)
Similarly,  $\frg_c$ deserves the name $\frg_{2,-14}$, as we were using for $L_\RR$   the name $\frg_{2,2}$.\smallskip

 \begin{remark}\label{re_ehr}
 In fact, $\mathbb{S}^6$ is a \emph{reductive} homogeneous space, that is, if
  $\frh=\{d\in\frg_c: d(X)=0\}$ is the Lie algebra of the isotropy group $H_X$, which is a compact Lie algebra isomorphic to the skew-Hermitian matrices $\frsu(W,\sigma)\cong\frsu(3)$, then there is a complementary subspace $ \mathfrak{m}$ such that
 $$
  \mathfrak{g}_c=\mathfrak{h}\oplus  \mathfrak{m},\qquad [ \mathfrak{h}, \mathfrak{m}]\subset  \mathfrak{m}.
 $$
Indeed, we can take $ \mathfrak{m}:=\{\varphi_Y:Y\in X^\perp\}\subset\frgl(V)$, for 
 \begin{equation}\label{eq_losfi}
 \varphi_Y(Z):=X\times (Y\times Z)-Y\times (X\times Z)+(X\times  Y)\times Z,
 \end{equation}
 and then check that $\varphi_Y\in\Der(V,\times)=\frg_c$ (see \cite[Eq.~(2.11)]{s6Alb}).
 The fact $ \varphi_Y(X)=2Y$ implies that $\mathfrak{m}$ is isomorphic to the vector space $X^\perp=W$ and that $\mathfrak{m}\cap\frh=0$. By dimension count, $\mathfrak{g}_c=\mathfrak{h}\oplus  \mathfrak{m}$. 
 
Moreover, this complementary invariant subspace $\mathfrak{m}$ is unique (in this example), coinciding with the subspace orthogonal to $\mathfrak{h}$ relative to the Killing form of $\mathfrak{g}_c$, according to \cite[Proposition 2.3]{s6Alb}.

In general, the tangent space to the orbit, $\mathbb S^6$ in this case, is identified with $\frg_c/\frh$, but due to the fact the homogeneous space is reductive, the identification becomes $T_X\mathbb S^6\cong \frg_c/\frh \cong\mathfrak{m}\cong X^\perp$, which besides coincides with the usual identification $T_X\mathbb S^6 \cong X^\perp$
inside $\RR^7$. This   allows to pass algebraical properties of $(\mathfrak m, [\ ,\ ]_{\mathfrak m})$ to geometrical ones, where $[\ ,\ ]_{\mathfrak m}$ denotes the projection on $\mathfrak m$ of the Lie bracket. As
$$
 [\varphi_{Y_1},\varphi_{Y_2}]_{\mathfrak m}=\varphi_{\mathrm{pr}_{W}(Y_1Y_2-Y_2Y_1)},
$$
the (new nonassociative) algebra 
  $(\mathfrak m, [\ ,\ ]_{\mathfrak m})$  is isomorphic to $(W,*)$, for $Y_1*Y_2= \mathrm{pr}_{W}(Y_1Y_2-Y_2Y_1)$, which turns out to be the \emph{compact vector color algebra} whose complexification is the vector color algebra originally introduced by Domokos and K\"ovesi-Domokos in the study of color symmetries arising from the Gell-Mann quark model (references in \cite{s6Alb}).   
  
Next we try to express the commutator, projected down to $\mathfrak{m}$, without passing through the identification with $W$. If the chosen identification   is $W\to\mathfrak{m}\le\frg_c$, $Y\mapsto \varphi_{J(Y)}$, then the following identity holds
\begin{equation}\label{eq_paraIlka2}
[Y_1,Y_2]_\mathfrak{m}=2J(Y_1\times Y_2),
\end{equation}
where recall that we think of $J\colon \OO_0\to\OO_0$, $J(Y)=X\times Y$, so that $J^2$ is the opposite of the projection on $W=X^\perp$. Indeed, some easy computations show that 
$[J(Y_1),J(Y_2)]=-X[Y_1,Y_2]X=-2X(Y_1\times Y_2)X$,
so that 
$$
\begin{array}{ll}
[J(Y_1),J(Y_2)]_W&=-2(X(Y_1\times Y_2)X)_W=-2X(Y_1\times Y_2)_WX\\
&=2X^2(Y_1\times Y_2)_W=-2(Y_1\times Y_2)_W=2J^2(Y_1\times Y_2)
,\end{array}
$$
since the elements in $W$ anticommute with $X$.
Then,
$$
[\varphi_{J(Y_1)},\varphi_{J(Y_2)}]_\mathfrak m=\varphi_{[J(Y_1),J(Y_2)]_W}=\varphi_{J(2J(Y_1\times Y_2))},
$$ which obviously means \eqref{eq_paraIlka2} through the identification.\smallskip
 
  More geometrically,  $\mathbb{S}^6$ is a \emph{naturally reductive} $G_{2,-14}$-homogeneous space, that is, the canonical 
connection is compatible with the metric, 
$$\langle [\varphi_{Y_1} ,\varphi_{Y_2} ]_{\mathfrak m},\varphi_{Y_3}\rangle +\langle \varphi_{Y_2},[\varphi_{Y_1} ,\varphi_{Y_3} ]_{\mathfrak m}\rangle =0$$
 for all $Y_1,Y_2,Y_3\in W$. This is a direct consequence of the fact that  $\langle \ ,\ \rangle\colon\mathfrak m\times\mathfrak m\to\RR$, the restriction of the Killing form to $\mathfrak m$, corresponds through the above identification\footnote{
 With both the above identifications, $Y\mapsto \varphi_{J(Y)}$ and $Y\mapsto \varphi_{Y}$, since $n(J(Y_1),J(Y_2))=n(Y_1,Y_2)$.
 } between  $\mathfrak m$ and $W$ with the restriction of the octonionic norm to $W$, up to scalar multiple. 
 Besides, $n(J(Y),Y_3)=n(-Y\times X, Y_3)=n(X,Y\times Y_3)$ gives, in particular, that
\begin{equation}\label{eq_paraIlka}
n(J(Y_1\times Y_2),Y_3)=n(X,(Y_1\times Y_2)\times Y_3),
\end{equation}
which provides useful information on the torsion 3-form of the canonical connection \cite[Proposition~4.2]{IlkaMAM}.

 \end{remark}

 \subsection{A model of compact $G_2$}\label{se_modelocompacto}

  The reductive decomposition of $\mathfrak g_c$ in the above remark 
  suggests the 
    following linear model of the compact Lie algebra of type $G_2$ and of its natural module. \smallskip
 
 \begin{proposition}\label{pr_modelocompacto}
 Take   $W=\CC^3$ and  $\sigma\colon W\times W\to\CC$   the Hermitian nondegenerate form
 given by $\sigma(u,v)=\sum_{i=1}^3 u_i\bar v_i$.  
 Recall that the unitary algebra $\fru(W)$ is spanned by the maps $\sigma_{u,v}$, where $\sigma_{u,v}(w) :=
\sigma(w, u)v -\sigma(w, v)u$ for any $u,v, w\in W$.
 Then the vector space 
 $$
 \mathcal L=\frsu(W)\oplus W,
 $$
 with anticommutative product given by 
 \begin{itemize}
 \item[$\RIGHTcircle$] $\frsu(W)$ is subalgebra,
 \item[$\RIGHTcircle$] $[\phi, u] = \phi(u)$ for any $\phi\in\frsu(W)$ and $u \in W$,
 \item[$\RIGHTcircle$] $[u, v] = (3\sigma_{u,v}-\tr(\sigma_{u,v} )\id_W)+ 2\overline{u \times v}$ for any $u, v \in W$,
 \end{itemize}
for $\times$ the usual vector product in $W=\CC^3$, is a compact Lie algebra of type  $G_2$  isomorphic to $\frg_c$.
(Also $\tr(\sigma_{u,v} )=\sigma(  v,u)-\sigma(u, v)=-2\mathbf{i}\,\mathrm{Im}\, \sigma(u,v)$.)\smallskip

Besides $\RR^7\equiv \RR\oplus W=:\mathcal V$ is an irreducible $\mathcal L$-module for the action $\cdot$ given by 
 \begin{itemize}
  \item[$\LEFTcircle$] For any $\phi\in\frsu(W)$, and $u \in W$,
  $$\phi \cdot1=0,\qquad \phi \cdot u=\phi(u);$$
 \item[$\LEFTcircle$]
  For any $u, v \in W$,
 $$u \cdot1=- 2\mathbf{i}u,\qquad u \cdot v=-2\mathrm{Im}\,\sigma(u,v)- \,\overline{u \times v}.$$
 \end{itemize} \smallskip
 
 Finally, the 3-form on $\mathcal V$ given by
 $$
 \Omega(s+u,t+v,r+w)=-\mathrm{Im}(s\sigma(v,w)+t\sigma(w,u)+r\sigma(u,v))+\mathrm{Re}(\det(u,v,w))
 $$
 for any $r,s,t\in\mathbb R$, $u,v,w\in W$,
 is generic, and $\mathcal L=\Der(\mathcal V,\Omega)\subset\mathfrak{so}(\mathcal V,n)$, for the positive definite  norm   
 $$
 n(s+u)=s^2+\sigma(u,u).
 $$
 The related cross product on $\mathcal V$ is given by
 $$
 (s+u)\times (t+v)=- \mathrm{Im}\,\sigma(u,v)+\big(\mathbf{i}sv-\mathbf{i}tu+\overline{u \times v}\big).  
 $$
 \end{proposition} 
 
 Although this model is extracted from \cite{modelosg2}, a difficulty when reading such paper is that the considered field $k$ is not necessarily $\mathbb R$, but a ground field of arbitrary characteristic, $K$ is not necessarily $\mathbb C$, but a two dimensional
composition algebra over $k$ (even in the case $k=\mathbb R$, there are two possibilities for $K$: either $\mathbb C$  or $\mathbb R\oplus \mathbb R$) and $W$ in the paper is a free left $K$-module of rank 3  endowed with a Hermitian nondegenerate form $\sigma\colon W\times W\to K$ with trivial Hermitian  discriminant. These choices produce a model valid even in the split case (considering  $\mathbb R\oplus \mathbb R$ instead of $\mathbb C$) but for avoiding the reader to dive in such general situation, we will sketch a proof adapted to our concrete setting.

 \begin{proof}
 Recall that $\OO=\RR\oplus V$ for $V=\OO_0= \RR\langle  e_i:1\le i\le 7\rangle$.
 If we identify $\CC\equiv\RR\oplus\RR e_7$, then $V=\RR e_7\oplus W $ for
 $W=\RR\langle  e_i:1\le i\le 6\rangle =\CC\langle  e_1,e_2,e_3\rangle$ and $W$ is isomorphic to $\CC^3$ in a natural way.
 This identification and the nondegenerate Hermitian form $\sigma\colon W\times W\to \CC$ given by $\sigma(u,v)=n(u,v)-n(e_7\times u,v)e_7=-\mathrm{pr}_{\RR\oplus\RR e_7}(uv)$ can  be used to write
  the product in the octonion division algebra $\OO=\CC\oplus W$ as
 \begin{equation}\label{eq_prmodelo}
 (\alpha+u)(\beta+v)=(\alpha\beta-\sigma(u,v))+\big( \alpha v+\overline{\beta} u+\overline{u\times v}),
 \end{equation}
 for $\alpha,\beta\in \CC$, $u,v\in W$, and the norm as
 $
 n(\alpha+u)=\alpha\overline{\alpha}+\sigma(u,u).
 $
 Now recall the decomposition $\Der(\OO)=\{d\in\Der(\OO): d(e_7)=0\}\oplus\{\varphi_u:u\in W\}$, where the derivations $\varphi_u$ are defined in Eq.~\eqref{eq_losfi}. If $d(e_7)=0$, then $d(W)\subset W$ and moreover $\sigma(d(u),v)+\sigma(u,d(v))=0$, so that we can consider the vector space isomorphism
 $$\begin{array}{rcl}
 \Der(\OO)&\longrightarrow&\frsu(W)\oplus W=:\mathcal{L}\\
  d&\mapsto&d\vert_W\in \frsu(W)\\
  \varphi_u&\mapsto & \mathbf{i}u.
  \end{array}
  $$
  This induces on $\mathcal{L}$ the Lie algebra structure given in the statement of this proposition. Computations can be extracted from  \cite[Theorem~5.3]{modelosg2}, where the correspondence with the notation in \cite{modelosg2} is $\varphi_u=\frac12D_{e_7,u}$ (take into account $D_{e_7,u}(e_7)=4u=2\varphi_u(e_7)$).
  
  The action of $\mathcal L$ on $\mathcal V$ (or on the octonions) is consequence of \cite[Lemma~5.1]{modelosg2}, passing through the above isomorphism, taking care of the fact that we have decided to omit $e_7$ in the final expression.
  
  The  cross product is immediate from Eq.~\eqref{eq_prmodelo}, since $\times =\frac12[\ ,\ ]$ in $\mathcal V$.
   The 3-form is simply obtained by taking $\Omega(s+u,t+v,r+w)=n((s+u)\times(t+v),r+w)$, again with the caution of the $e_7$-omission. 
 \end{proof}
 
 

 \subsection{Parenthesis on real forms of $\frg_2$}\label{se_solodos}
 
 It is a well known fact of Lie theory that there  
 are just two real forms of type $G_2$ up to isomorphism, namely,  $\frg_{2,-14}$ and  $\frg_{2,2}$.
  We would like to arrive at this conclusion without using strong machinery of Lie theory but only our tools of 3-forms, although we will add some comments about more usual techniques too. 
 
 First we  come back to the complex setting to prove that the stabilizer of any 3-form which contains the group $G_2$ cannot  strictly contain $G_2$. Take $V=\CC^7$, $\Omega_0$ the generic 3-form in Theorem~\ref{teo_casoC}, $\frg_2=\Der(V,\Omega_0)$ and $n=n_{\Omega_0,\cB}\colon V\times V\to\CC$ the related nondegenerate form ($\cB$ any fixed basis of $V$). Through this subsection, $\cC=\CC\oplus V$ will denote the complex octonion algebra.
  
 \begin{lemma}\label{le_solounan}
 Up to scalar multiple, $n$   is the unique   $\frg_2$-invariant symmetric bilinear map $V\times V\to\CC$.
  \end{lemma}
 
 \begin{proof}
 Giving a $\frg_2$-invariant symmetric bilinear map $q\colon V\times V\to\CC$ is equivalent to giving a homomorphism of $\frg_2$-modules from $V$ to $V^*$, by means of $\tilde q\colon V\to V^*$, $\tilde q(X)\in V^*$, $\tilde q(X)(Y)=q(X,Y)$.
 Using the notations in Section~\ref{se:modelog2}, the module $V$ is 
   $\frg_2$-irreducible generated by the vector $X=\left(\begin{array}{c}
0\\0\\e_3
\end{array}\right)\in V_{\omega_1=-\varepsilon_3}$ (multiplying  successively  by $L_{-\alpha_i}$, $i=1,2$, we immediately get all the weight spaces). So, a homomorphism of  $\frg_2$-modules from $V$ to $V^*$ is determined by providing the image of $X$. But the image of $X$ must also have weight $-\varepsilon_3$, so that it is determined up to scalar since 
$\dim (V^*)_{\omega_1}=\dim V_{\omega_1}=1$.  
 \end{proof}
 
 (This lemma is equivalent to the fact mentioned in Remark~\ref{re_S2enwedge3} that $S^2V^*=V(2\omega_1)\oplus V(0)$ contains only one copy of a trivial submodule type $V(0)$.)
 
 \begin{lemma}\label{le_contenerg2essergenerica}
 Any $0\ne\Omega\in\wedge^3V^*$ such that $G_{\Omega_0}^+\subset G_\Omega$ is necessarily generic.
 \end{lemma}
 
 \begin{proof}
 Let us see that $n_{\Omega,\cB}\colon V\times V\to\CC$ is nondegenerate, which is a sufficient condition to assure $\Omega$ generic according to  Remark~\ref{re_reciproco}.
 Note that   $n_{\Omega,\cB}\colon V\times V\to\CC$ is   $G_{\Omega}\cap\SL(V)$-invariant
by the same argument as in Lemma~\ref{le_parteconexa}b). Besides Lemma~\ref{le_parteconexa}b) tells us that  $G_{\Omega_0}^+\subset \SL(V)$, so that $n_{\Omega,\cB}\colon V\times V\to\CC$ is $G_{\Omega_0}^+$-invariant.
 By Lemma~\ref{le_solounan}, 
 $n_{\Omega,\cB}$ is a nonzero scalar multiple of $n=n_{\Omega_0,\cB}$. In particular, $n_{\Omega,\cB}$ is nondegenerate (and generic by Remark~\ref{re_reciproco}).
 \end{proof}
 
  \begin{lemma}\label{le_modulos}
If $L_1$ is  a real simple Lie algebra and $W$ is an $L_1$-irreducible (finite-dimensional) module, then one of the following possibilities happens:
\begin{itemize}
\item[a)]   $\End_{L_1}W\cong\RR$. In this case $W^\CC$ is an $L_1^\CC$-irreducible module (i.e., $W$ is absolutely irreducible);
\item[b)]   $\End_{L_1}W\cong\CC$. In this case $W^\CC$ is sum of two nonisomorphic  irreducible $L_1^\CC$-modules;

\item[c)]   $\End_{L_1}W\cong\HH$. In this case $W^\CC$ is sum of two  isomorphic  irreducible $L_1^\CC$-modules. Besides $\dim_\RR W$ is multiple of $4$.
\end{itemize}
\end{lemma} 
 
  \begin{proof}
 On one hand, $\End_{L_1}W $ is a real division (associative) algebra, since for any $0\ne f\in \End_{L_1}W $,
 the submodule $\ker(f)$ has to be 0 and the submodule $\textrm{Im}(f)$ has to be $W$, so that $f$ is invertible. Now, the classical Frobenius theorem (1877) characterizes the finite-dimensional associative division algebras over the real numbers: they are  $\RR$, $\CC$ and $\HH$ and hence $d=\dim_{\CC}\End_{L_1^{\CC}}W^{\CC}=\dim_\RR\End_{L_1}W $ is either $1$, $2$ or $4$, respectively.
 
 On the other hand, Schur's lemma says that, for $U_1$ and $U_2$ irreducible $L_1^{\CC}$-modules, 
 $\dim_\CC\Hom_{L_1^\CC}(U_1,U_2)$ is $1$ if $U_1$ and  $U_2$ are isomorphic and it is $0$ otherwise. Consequently, if
 $U\cong s_1U_1\oplus\dots\oplus s_kU_k$ is the decomposition of an arbitrary $L_1^{\CC}$-module $U$ as a sum of irreducible submodules  with $U_i$ not isomorphic to $U_j$ if $i\ne j$, then $\dim_\CC\End_{L_1^\CC}U=\sum_{i=1}^ks_i^2$. Now, we apply this to the $L_1^{\CC}$-module $U=W^{\CC}$.
 \begin{itemize}
\item[a)]  If $d=1$, then $k=1=s_1$ and  $W^\CC$ is an $L_1^\CC$-irreducible module.
\item[b)] If $d=2$, then $k=2$ and $s_1=s_2=1$. Hence  $W^\CC$ is sum of two nonisomorphic  irreducible $L_1^\CC$-modules.
\item[c)]  If $d=4$, then either $k=1$ and $s_1=2$ or $k=4$ and $s_i=1$ for all $i=1,\dots,4$. We want to rule out the second possibility: it suffices to realize that $W^\CC\cong 2W$ is the decomposition of $W^\CC$ as a sum of irreducible $L_1$-modules,  so that it is not possible $W^\CC\cong\sum_{i=1}^4U_i$ as $L_1^{\CC}$-module because in particular when considering as $L_1$-module there would be at least $4$ irreducible summands instead of just $2$. Hence we have proved that $U\cong 2U_1$  is sum of two  isomorphic  irreducible $L_1^\CC$-modules. 

Finally, if $\psi\colon\End_{L_1}W\to\HH$ is the isomorphism, then an action $\HH\times W\to W$ is given by $(q,w)\mapsto \psi^{-1}(q)(w)$, so   $W$ can be endowed with   a quaternionic vector space structure, what in particular implies that $\dim_\RR W$ is multiple of $4$.
\end{itemize}
 \end{proof}

  \begin{proposition}\label{pr_elmoduloreal}
Let $L_1$ be  a real Lie algebra such that $L_1^\CC=\Der(\cC)$. Then there exists $V_1$ an $L_1$-module such that $V_1^\CC=V$.
 \end{proposition} 
 
 \begin{proof}
  We can consider $V=\CC^7$ as $L_1$-module of (real) dimension 14. Now the complexified module $V^\CC=V\otimes_\RR\CC$ (of complex dimension 14) is not an irreducible $L_1^\CC$-module but sum of two isomorphic modules.
  Indeed, since the arbitrary element $h=\diag\{0,s_1,s_2,s_3,-s_1,-s_2,-s_3\}\in\frh$ acts on both $\left(\begin{array}{c}0\\0\\e_3\end{array}\right)\otimes 1$ and 
  $\left(\begin{array}{c}0\\0\\e_3\end{array}\right)\otimes \mathbf{i}$ with the same weight $-s_3=-\varepsilon_3(h)=\omega_1(h)$, it follows that 
  $V^\CC\cong2V(\omega_1)$ is the decomposition of  $V^\CC $ as a sum of irreducible $L_1^\CC$-submodules. This means that none of the situations a), b) and c) in Lemma~\ref{le_modulos}
is satisfied (the third possibility since $14$ is not multiple of $4$), in other words, $V$ is not $L_1$-irreducible. Take $V_1$ an irreducible $L_1$-submodule of $V$. As $V_1^\CC$ is a proper submodule of $V^\CC\cong 2V(\omega_1)$, this ensures that  $V_1^\CC\cong V(\omega_1)$.
 \end{proof}
 
  
 That is, $V=\RR^7$ is a real representation  of any real form of $\mathfrak g_2$.
 Note that this means that not only $\frg_{2,2}\subset\frgl_7(\RR)$, but also $\frg_{2,-14}$  is a subalgebra of $\frgl_7(\RR)$, as we have already emphasized in Proposition~\ref{pr_modelocompacto}.

\begin{proposition}\label{pr_finalmentedos}
There are  only two real forms of $\Der(\cC)$ up to isomorphism; namely, $\Der(\RR^7,\Omega_0)\cong\frg_{2,2}$ and
$\Der(\RR^7,\Omega_1)\cong\frg_{2,-14}$.
 \end{proposition} 
 
 \begin{proof}
 If a Lie algebra $L_1$ is  a real   form of $ \Der(\cC)\equiv\Der(\CC^7,\Omega_0)$, take $V_1$ the $L_1$-module of dimension $7$ as in Proposition~\ref{pr_elmoduloreal}, i.e. such that  $V_1^\CC=\CC^7$. 
 Consider the restriction of the $\CC$-trilinear alternating map $\Omega_0\colon \CC^7\times  \CC^7\times  \CC^7\to\CC$ to 
 $\Omega_0\vert_{V_1}\colon V_1\times  V_1\times  V_1\to\CC$: thus 
 $\textrm{Re}(\Omega_0\vert_{V_1})$ and $\textrm{Im}(\Omega_0\vert_{V_1})$ are both $\RR$-trilinear alternating maps, and at least one of them is nonzero, say $\Omega_1'$. As trivially they are $L_1$-invariant, then 
 $L_1\subset \Der(V_1,\Omega_1')$.  Now Lemma~\ref{le_contenerg2essergenerica} implies that $\Omega_1'^\CC$ is generic. Hence  $\Omega_1'$ is generic too, so it belongs to the orbit of either $\Omega_0$ or $\Omega_1$, what finishes the proof.
 \end{proof}

  \begin{remark}  
 
 In particular, Proposition~\ref{pr_finalmentedos} allows us to conclude that $\frg_c$ is compact without using properties of homogeneous spaces, because $\frg_c$ is not split since it contains a compact subalgebra isomorphic to $\frsu(3)$. Alternatively, we can argument without a deep detail why $\kappa$ is negative definite. The computation of the Killing form is always difficult, but some clues are: 
\begin{itemize}
\item $\mathfrak{h}$ and $\mathfrak{m}$ in Remark~\ref{re_ehr} are orthogonal for $\kappa$, since when we complexify, $\mathfrak{h}^\CC\cong \cL_{\bar 0}$ and $\mathfrak{m}^\CC\cong \cL_{\bar 1}\oplus \cL_{\bar 2}$ as in Remark~\ref{re_grad}, and the pieces of any group-grading (in this case a $\ZZ_3$-grading) satisfy $\kappa(\cL_{\bar i},\cL_{\bar j})=0$ if $\bar i+\bar j\ne\bar0$, because $(\ad \cL_{\bar i}\ad\cL_{\bar j})(\cL_{\bar k})\cap\cL_{\bar k}\subset \cL_{\bar i+\bar j+\bar k}\cap\cL_{\bar k}=0$;
\item The restriction $\kappa\vert_{\frh}$ is a (positive) multiple of the Killing form of $\frsu(3)$;
\item There is $\alpha\in\RR$ a positive scalar such that $\kappa(\varphi_Y,\varphi_Z)=-\alpha n(Y,Z)$. (The key is that both maps $\kappa$ and $n$ are $\frh$-invariant maps $\mathfrak{m}\times\mathfrak{m}\to\RR$, so representation theory says that one should be scalar multiple of the other. This allows to find $\alpha$ by computing   $\kappa$ in one concrete nonorthogonal couple $(Y,Z)$.)
\end{itemize}
  \end{remark}  
 
  \begin{remark}  
  The usual way of knowing that there are just two real forms of type $G_2$ is sketched here.

As mentioned in Section~\ref{se_reales}, the real forms (up to isomorphism) of a complex Lie algebra are in bijective correspondence with involutive (order one or two) automorphisms (up to conjugation) of  the complex algebra. But in $\frg_2$ there is only one conjugacy class of nontrivial order two automorphisms. Indeed, for $\cC=\CC\oplus\CC^7$ the octonion algebra and  $\frg_2=\Der(\cC)$, then
$$
\Aut(\cC)=G_2\longrightarrow \Aut(\frg_2),\qquad f\mapsto \Ad(f)
$$
is a group isomorphism, where
$$
\Ad(f)\colon \frg_2\to\frg_2,\qquad \Ad(f)(d)=fdf^{-1}.
$$
(Take into consideration that $\Der(\cC)=\frg_2\cong\ad(\frg_2)=\Der(\frg_2)$, since any derivation of a semisimple Lie algebra over $\FF$ is inner. This fact is the linear version of the above.)
But $\Ad(f)^2=\id_{\frg_2}$ means that $f^2$ commutes with $\Der(\cC)$, so that $f^2\in\CC\id_{\cC}$, since the centralizer
$\textrm{Cent}_{\frgl(\cC)}\Der(\cC)=\CC\id_{\cC}$. As
$-\id_{\cC}$  is not an automorphism, then 
  $f^2= \id_{\cC}$. Note finally that in $\cC$ there is a unique order two automorphism  up to conjugation, which acts as the identity in a copy of $\CC\oplus\CC^3$ (quaternion algebra provided by the -essentially unique- cross product in $\CC^3$) and as   minus the identity in the orthogonal subspace. \end{remark}

\subsection[hola]{Case the norm is not positive definite}\label{se_casosplit}\smallskip
 
 Again, we are interested in the action $G_{\Omega_0}\times V\to V$, although this time the norm has signature $(4,3)$, so that   the orbit $G_{\Omega_0}\cdot X$ for a nonisotropic vector $X$ with norm $n(X)=-1$    will be   contained in the hypercuadric $\{Y\in V: n(Y)=-1\}$.  
 (Recall that $G_{\Omega_0}\subset\SO(V,n)$ according to Lemma~\ref{le_parteconexa}, since we are considering here the real case.)
 
 {We follow notation   in O'Neill's book \cite[4]{ONeill}. So
 $\RR_{\nu}^{n+1}$ denotes $\RR^{n+1}$ with the norm $q$ given by $\diag\{-I_\nu,I_{n+1-\nu}\}$. The nullcone
 $\{x\in\RR_{\nu}^{n+1}:q(x)=0\}$ is a hypersurface  diffeomorphic to $(\RR^\nu\setminus0)\times\mathbb{S}^{n-\nu}$; the pseudosphere of radius $r>0$ is $ \mathbb{S}_{\nu}^n(r)=\{x\in\RR_{\nu}^{n+1}:q(x)=r^2\}$ and the pseudohyperbolic space 
of radius $r>0$ is $ H_{\nu}^n(r)=\{x\in\RR_{\nu+1}^{n+1}:q(x)=-r^2\}$. } Thus, our hyperquadric  is   $\{Y\in \RR_{4}^{7}: n(Y)=-1\}\equiv H_3^6(1)$.

 Similarly to Proposition~\ref{le_elsigma} we have,

 \begin{proposition}\label{le_comoelsigma}
Take $X\in V$ such that $n(X)=-1$, and $W=X^\perp$. Recall that $W=W^+\oplus W^-$, for the (totally isotropic) eigenspaces
$W^\pm=\{Y\in V: f(Y)=\pm Y\}$, since the map $f\colon W\to W$ given by $f(Y)=X\times Y$ satisfies $f^2=\id$.
Then, for $H_X=\{g\in G_{\Omega_0}: g(X)=X\}$ the isotropy group of $X$, the map
\begin{equation}\label{eq_SL3}
H_X\to\mathrm{SL}(W^+),\quad g\mapsto g\vert_{W^+}
\end{equation}
is well defined and a Lie group isomorphism.   
Furthermore,
\begin{equation}\label{eq_ehr}
G_{{\Omega_0}}/\mathrm{SL}_3(\RR)\cong H_3^6(1).
\end{equation}
\end{proposition}

\begin{proof}
Of course we can consider the restriction of $g\in H_X$ to $W^+$ since $g$ commutes with $f$.
Now $\Omega_0\colon W^+\times W^+\times W^+\to\RR$ is a trilinear alternating nonzero map preserved by $g$, so that $\det g\vert_{W^+} =1$, and the map in \eqref{eq_SL3} is well defined.

The injectivity is a consequence of the fact that $g\vert_{W^+}$ determines $g\vert_{W^-}$, since $W^+$ is dual to $W^-$ and   $  H_X\subset\SO(W,n)$. In particular, $\dim H_X\le\dim \mathrm{SL}_3(\RR)=8$. According to Remark~\ref{re_cotasdimensiones}, $\dim H_X\ge8$, what implies  $\dim H_X=8$ and the map \eqref{eq_SL3} is also surjective.

The last task is to prove that the orbit $G_{\Omega_0}\cdot X$ fills the whole $\{Y\in V: n(Y)=-1\}\equiv H_3^6(1)$. But for any $Y\in 
H_3^6(1)$, we find $Y_1,Y_2,Y_3\in\ker(f-\id)$ with $\Omega(Y_1,Y_2,Y_3)=-4$ and $Z_i=\frac12 Y_{i+1}\times Y_{i+2}$ such that $Y_1+Z_1=2Y$. (This, like in the proof of Theorem~\ref{teo_casoC}, implies that there is $\varphi\in G_{\Omega_0}$ such that $\varphi\left(\frac{E_1+F_1}2\right)=Y$, so that the orbit $ G_{\Omega_0}\cdot \frac{E_1+F_1}2=H_3^6(1)$.) In order to find the required basis, it suffices to take $Y_1=Y+f(Y)\in\ker(f-\id)$ and $Y_2,Y_3\in\ker(f-\id)\cap \langle Y,f(Y)\rangle^\perp$ (adjusting the scalars to have $\Omega_0(Y_1,Y_2,Y_3)=-4$). Then $Z_1=\frac12 Y_{2}\times Y_{3}$ is equal to $Y-f(Y)$ because $Y-f(Y)\in\ker(f+\id)$ satisfies $n(Y+f(Y),Y-f(Y))=-2$ and $n(Y-f(Y),Y_s)=0$ for $s=2,3$ (properties which characterize the element $Z_1$).
\end{proof}

Again the connectedness of $G_{\Omega_0}$ is an immediate  consequence of this proposition. Also there is a complex version of this proposition, which  will be commented in the next subsection.

The homogeneous space described in  \eqref{eq_ehr} is again a reductive homogeneous space, with related reductive decomposition  provided by Remark~\ref{re_grad}.\footnote{
A reader interested in all  the eight reductive homogeneous spaces of $G_2$ (both split and compact), their invariant affine 
connections, holonomy algebras and so on can consult  \cite{ehrg2}. 
} This time the complementary subspace $T_XH_3^6(1)$ is not $\frsl_3(\RR)$-irreducible, but it breaks as a sum of the natural 3-dimensional module and its dual one. 

Note that  $H_3^6(1)$ is diffeomorphic (and antiisometric) to $\mathbb{S}_3^6$ \cite[Lemmas~24 and 25]{ONeill}, so that in particular $G_{2,2}$ acts transitively too on the pseudosphere $\mathbb{S}_3^6$.

\begin{remark} In fact, the same arguments can be used to obtain a stronger result, namely, our real groups of type $G_2$ can be identified with certain Stiefel manifolds:
\begin{itemize}
\item[i)] $G_{2,-14}\cong\{(X_0,X_1,X_2)\in (\RR^7)^3:n(X_i,X_j)=\delta_{ij}, \ X_0\perp X_1\times X_2\}$, for $n$ a positive definite norm;  
\item[ii)] $G_{2,2}\cong\{(X_0,X_1,X_2)\in (\RR^7)^3:n(X_i,X_j)=-\delta_{ij}, \ X_0\perp X_1\times X_2\}$, for $n$ a  norm with signature $(4,3)$.
\end{itemize}

As the group preserves the product and the norm, it is clear that it also acts on the corresponding set of ordered triples. Let us see in each case why this is a transitive action with isotropy group equal to the identity.
For i), the endomorphism of $V$ as in Eq.~\eqref{eq_transi} leads a triplet to another one. For ii), take $X=X_0$, $Y_1=X_1+X\times X_1$ (that belongs to the eigenspace $V_{X,1}$), $Y_2=X_2+X\times X_2$, $Y_3\in  V_{X,1}$ such that $n(Y_3,Y_1\times Y_2)=-4$, $Z_i=\frac12 Y_{i+1}\times Y_{i+2}$. Equation~\eqref{transisplit} says that such a set can be lead to any other set (basis, in fact) constructed in the same way. But $Y_1+Z_1=2X_1$ and $Y_2+Z_2=2X_2$ as at the end of the proof of Proposition~\ref{le_comoelsigma}.  \smallskip

Observe that this remark provides a way of understanding the real groups of type $G_2$ also related to octonions,  described in \cite[4.1]{Baez}. For instance, for the compact case: 
$(X_0,X_1,X_2)$ is called a \emph{ basic triple} if $X_i$'s are zero trace octonions whose square is $-1$, $X_1$ anticommutes with $X_0$ and $X_2$ anticommutes with $X_0$, $X_1$ and $X_0X_1$
  (shorter, three orthonormal vectors such that $ \Omega_1(X_0,X_1,X_2)=0$). 
 In this case $(X_1,X_2,X_0)$ is a basic triple too. Besides the subalgebras of $\OO$ generated by $\{X_0\}$, $\{X_0,X_1\}$ and $\{X_0,X_1,X_2\}$ are respectively isomorphic to $\CC$, $\HH$ and $\OO$. According to i) the group $G_{2,-14}$ is in bijective correspondence with the set of basic triples, for the  correspondence  $g\in G_{2,-14}\leftrightarrow(g(e_1),g(e_2),g(e_7))$.
\end{remark}

\subsection{Connectedness}\label{se_conexo}

We have postponed some points on the connectedness until now, namely, Corollaries \ref{co_conexossplit} and \ref{co_conexoscompactos}, due to the fact that  the  handling  of groups is  more difficult   than the one of algebras. First we will check the version for groups of Proposition~\ref{pr_preservacross}, enclosing also the positive definite case. (In fact, Proposition~\ref{pr_preservacross} is a trivial consequence of Proposition~\ref{pr_autmalcevyc}, but not conversely.)
   
  \begin{proposition}\label{pr_autmalcevyc}
  Let $\Omega$ be a generic 3-form on $V=\FF^7$ ($\FF$ either $\RR$ or $\CC$), $n$ the norm related to the corresponding cross product $\times $ and $\cC=\FF\oplus V$ the related octonion algebra.
The map
$$
f\in\Aut(V,\times)\mapsto \tilde f\in\Aut(\cC),\quad\left\{\begin{array}{l}\tilde f(1)=1,\\\tilde f\vert_V=f,\end{array}\right.
$$
is a Lie group isomorphism. Furthermore,   $\Aut(V,\times)\subset\mathrm{O}(V,n)\cap G_\Omega$.
\end{proposition}

\begin{proof}
Take $f\in  \Aut(V,\times)$ and linearly independent elements  $X,Y\in V$. By \eqref{eq_mainidentity}, 
$$
n(X,Y)f(X)-n(X)f(Y)=f(X\times(X\times Y))=n(f(X),f(Y))f(X)-n(f(X))f(Y),
$$
and consequently $n(X,Y)=n(f(X),f(Y))$.  
As besides $\Omega(X,Y,Z)=n(X\times Y,Z)$, clearly the above implies that also $\Aut(V,\times)\subset\mathrm{O}(V,n)\cap G_\Omega$.

Now $\tilde f(XY)=\tilde f(-n(X,Y)1+X\times Y)$ coincides with 
$\tilde f(X)\tilde f(Y)=f(X)f(Y)=-n(f(X),f(Y))1+f(X)\times f(Y)$, so that  $\tilde f$ is an automorphism of the octonion algebra.\smallskip

Conversely take $f\in\Aut(\cC)$. As $f(X)=f(1X)=f(1)f(X)$ for all $X$, then $f(1)=1$ ($f$ bijective). Let us check that $f(V)\subset V$. We apply $f$ in Eq.~\eqref{eq_cuad} to get
$$
-n(X)1=f(-n(X)1)=f(X^2)=f(X)^2=2n(f(X),1)f(X)-n(f(X))1 
$$
when $X\in V$, but $f(X)\notin \FF1$, hence $n(f(X),1)=0$ and $f(X)\in V$ (Remark~\ref{re_cuadratica}). We have that 
$$
\begin{array}{ll}
f(X\times Y)&=f(XY+n(X,Y)1)=f(X)f(Y)+n(X,Y)1\\
&=(-n(f(X),f(Y))+n(X,Y))1+f(X)\times f(Y)
\end{array}
$$
if $X,Y\in V$, so the projection on $V$ gives $f\in  \Aut(V,\times)$.
\end{proof}

As we have commented throughout the paper,

\begin{lemma}\label{pr_conreal}
 Let $\Omega$ be a generic 3-form on $V=\RR^7$. Then $G_\Omega$ is connected.
\end{lemma}
\begin{proof}
We can assume that $\Omega$ is either $\Omega_0$ or  $\Omega_1$.
According to \cite[Proposition~3.66]{Warner}, if $H$ is a connected closed subgroup of a Lie group $G$ and $G/H$ is connected, then $G$ is also connected. This immediately gives the result taking into consideration Propositions~\ref{le_elsigma} and \ref{le_comoelsigma}, respectively.
\end{proof}

This proves, in particular, Corollary~\ref{co_conexoscompactos} and the real part of Corollary~\ref{co_conexossplit}.
The complex case requires to be careful:  By Lemma~\ref{le_parteconexa},
   $G_\Omega$ is not connected,
since $\det(\omega\id_V)=\omega^7=\omega\ne1$, so   $\omega\id_V\notin G_\Omega^+$. 

\begin{lemma}\label{pr_concomplejo}
Let $\Omega$ be a generic 3-form on $V=\CC^7$ and $n$ a  
related nondegenerate bilinear form. Then
$$
\Aut(V,\times)= G_\Omega^+=G_\Omega\cap\SL(V)=G_\Omega\cap\mathrm{O}(V,n)\cong\Aut(\cC).
$$
\end{lemma}

In particular, this finishes  the complex part of Corollary~\ref{co_conexossplit}.

\begin{proof}
By Lemma~\ref{le_parteconexa}, $G_\Omega^+\subset G_\Omega\cap\SL(V)=G_\Omega\cap\mathrm{O}(V,n)$. In order to prove that both groups are equal, it suffices to prove that $G_\Omega\cap\SL(V)$ is connected because the Lie algebras of both groups are the same ($\Der(V,\Omega)\subset\frso(V,n)$ because $\Omega$ is in the orbit of $\Omega_0$, Lemma~\ref{le_antisim} and Proposition~\ref{pr_enson}). Consider the action of this group on the connected manifold $\{X\in\CC^7:n(X)=-1\}$ (since the group preserves $n$.) The isotropy subgroup of $X$ is isomorphic to $\SL_3(\CC)$ because of the same arguments as in the proof of Proposition~\ref{le_comoelsigma}. But the action is transitive. Indeed, if $X$ has norm $-1$, there is   $\varphi\colon V\to V$   given by $\varphi(E_0)=X$, $\varphi(E_i)=Y_i$ and $\varphi(F_i)=Z_i$ ($Y_i$'s and $Z_i$'s have been chosen as in \eqref{transisplit}), that belongs obviously to $G_\Omega$, but besides preserves $n$ (again by  \eqref{transisplit}) so that  $\varphi\in G_\Omega\cap\mathrm{O}(V,n)$ and thus the action is transitive.
Again \cite[Proposition~3.66]{Warner} gives that $G_\Omega\cap\SL(V)$ is connected.

The rest of the proof is then a trivial consequence of Proposition~\ref{pr_autmalcevyc} and the fact of having the same Lie algebras: Propositions~\ref{pr_conservacross} and ~\ref{pr_enson}.
\end{proof}

In consequence the automorphism group of an octonion algebra is always a connected  Lie group, both in the two real cases (corresponding to division and split octonion algebras, respectively) and also in the complex one. Instead, not all of them are simply connected.

\begin{proposition}\label{co_simpleconexion}
  $G_{2,2}$ is not simply connected, its fundamental group is $\ZZ_2$.
\end{proposition}

\begin{proof}
For a coset manifold $M=G/H$, there is an exact sequence of groups and homomorphisms \cite[Proposition~11.17]{ONeill},
\begin{equation}\label{eq_sec}
0\to\pi_2(M)\to\pi_1(H)\to\pi_1(G)\to \pi_1(M)\to\pi_0(H)\to\pi_0(G).
\end{equation}
We apply this sequence to  $G_{2,2}/\SL_3(\RR)\cong \mathbb S_3^6$. As $\mathbb S_3^6$ is diffeomorphic (not isometric) to $\mathbb S^3\times\RR^3$, then it has trivial fundamental group  as well as trivial second homotopy group, as the 3-dimensional sphere has. So $\pi_1(H)\cong\pi_1(G) $, and the result follows since   $\pi_1(\SL_3(\RR))=\pi_1(\SO(3))\cong\ZZ_2$.
\end{proof}

Of course Eq.~\eqref{eq_sec} gives that $G_{2,-14}$ is simply connected. In fact, the complex group $G_2$ is also  simply connected, but in order to apply the   exact sequence of groups to the coset manifold $G_{2}/\SL_3(\CC)\cong   \{X\in\CC^7:n(X)=1\}$, one would have   to understand well how is the manifold $\{X\in\CC^7:n(X)=1\}$. Instead,
we can conclude the simply connectedness in the complex case   from the fact that the determinant of the Cartan matrix in \eqref{eq_CartanmatrixG2} is 1, by using nonelementary facts on structure theory of Lie groups. (The determinant equal to 1 implies that the weight lattice $\sum\ZZ\omega_i$ and the root lattice  $\sum\ZZ\alpha_i$ coincide, but the fundamental group is related with the \emph{distance} between them.)  

\begin{remark}\label{re_elsc} 
The previous proposition implies that there are just three connected real Lie groups of type $G_2$, namely, $G_{2,-14}$, $G_{2,2}$ and its double covering $\widetilde G_{2,2}$. One can wish an explicit description of this simply connected Lie group $\widetilde G_{2,2}$, but note that this group is not a matrix Lie group: any of its linear representations    is not faithful, so it cannot be seen as a closed subgroup of $\GL_m(\RR)$ for any $m$.\footnote{If $  \widetilde G_{2,2}\to\GL_m(\RR)$ is a representation, take the associated Lie algebra representation of $\mathfrak g_{2,2}$, and extend it by complexification to a representation of $\mathfrak g_{2}=\mathfrak g_{2,2}^\CC$. As the complex group $G_2$ is simply connected, we may exponentiate the previous representation to one of $G_2$. Then we restrict it to $G_{2,2}\subset G_2$ and compose with the projection $\widetilde G_{2,2}\to G_{2,2}$ to obtain a $\widetilde G_{2,2}$-representation. It is not difficult to check that this representation coincides with the first one, since the corresponding Lie algebra representations are equal, and hence the center of $\widetilde G_{2,2}$ acts trivially. 
}
 (\cite[\S2]{Vogan} exposes  the basic structure theory of  $\widetilde G_{2,2}$.)\smallskip

However, there are many Lie groups of type $G_2$ if we remove the hypothesis of connectedness. For instance \cite{Harvey,Baez}, take the 
group 
$$
\{f\in\GL(V):f(x,y,z)=(f(x),f(y),f(z))\quad\forall x,y,z\in V\}=\{\pm f:f\in   G_{2,2}\}
$$
where the associator  in $V=\RR^7$ is defined by $(x,y,z):=(xy)z-x(yz)$, which is nonzero because the split octonion algebra $\cC=\RR\oplus V$ given by Eq.~\eqref{eq_procto} is not associative. The above group is isomorphic to $G_{2,2}\times \ZZ_2$, of course nonconnected. This example works similarly for the division  octonion algebra.
\end{remark}

\section{$G_2$ and   spinors 
}\label{se:nose}
 
Spinors were invented by Dirac in creating his relativistic quantum theory of the electron, but today they are relevant also 
in quantum theory, relativity, nuclear physics, atomic and molecular physics, and condensed matter physics. 
  Mathematician readers interested in seeing how
  spinors relates to modern theoretic physics 
  can consult   the introduction of the Srni lecture notes  \cite[Sections~1.3,1.4]{Agricola06}. \smallskip

In this section we will prove that, when the group $\spin(7)$ acts on the eight-dimensional spin representation, the subgroup stabilizing any unit \emph{spinor} is again our compact group of type  $G_2$. This happens not only for the  definite norm and the compact group $G_{2,-14}$, but for the norm of signature $(4,3)$ and the split group $G_{2,2}$, thus providing transitive actions on the sphere (pseudosphere, respectively) of dimension 7. 
%
We are going to make all the above completely precise. 

 We will work throughout the section with the real field, although all this works equally well in the complex setting. \smallskip
 

Recall first that if $n\colon V\to \RR$ is a nondegenerate quadratic form,
 the \emph{Clifford algebra} $\Cl(V,n)$ is 
a unital associative algebra that   is generated by   $V$ subject to the relations $v^2=n(v)1$ for all $v\in V$ (that is, the tensor algebra $  \cT(V)={\displaystyle \oplus _{k\geq 0}V\otimes \stackrel{k}{\cdots} \otimes V}$
quotient   by the two-sided ideal generated by elements $v\otimes v-n(v)1$).  The product induced by the tensor product in the quotient algebra is written using juxtaposition (e.g. uv). 
Of course Clifford algebras satisfy the following universal property: if $\rho\colon V\to A$ is a linear map into a unital associative algebra $A$ such that $\rho(v)^2=n(v)1_A$, then there is a unique homomorphism of associative algebras $\tilde\rho\colon \Cl(V,n)\to A$ such that $\tilde\rho\vert_V=\rho$.

If the dimension of $V$  is $m$ and $\{e_1, ..., e_m\}$ is an orthogonal basis of $(V, n)$, then a basis of  $\Cl(V,n)$
is   ${\displaystyle \{e_{i_{1}}e_{i_{2}}\cdots e_{i_{k}}\mid 1\leq i_{1}<i_{2}<\cdots <i_{k}\leq m{\mbox{ and }}0\leq k\leq m\}}$, so that $\dim \Cl(V,n)=\sum_{k=0}^m\left(\begin{array}{l}m\\k\end{array}\right)=2^m$. 
The natural $\ZZ$-grading on $\cT(V)$ given by assigning the degree $\deg(e_{i_{1}}\otimes e_{i_{2}}\cdots \otimes e_{i_{k}})=k$, induces a $\ZZ_2$-grading on $\cT(V)$ which is inherited by the Clifford algebra $\Cl(V,n)$, since the quotient ideal is homogeneous. Thus,  
$$
 \Cl(V,n)=\Cl_{\bar0}(V,n)\oplus\Cl_{\bar1}(V,n),
$$
being the even part of the Clifford algebra (respectively odd), denoted by $\Cl_{\bar0}(V,n)$ (resp. $\Cl_{\bar1}(V,n)$), spanned by elements which are products of an even (resp. odd) number of vectors in $V$. In particular,  $\dim\Cl_{\bar0}(V,n)=2^{m-1}$. If $m$ is even, then  $\Cl(V,n)$ is simple, and otherwise, $\Cl_{\bar0}(V,n)$ is simple (facts about real and complex Clifford algebras can be consulted in \cite[Chapter~9]{Harvey}). 
The spin group lives in the group of invertible elements of  $\Cl_{\bar0}(V,n)$, namely,
$$
\spin(V,n):=\{\pm a_1\dots a_{2r}: a_i\in V,\prod_{i=1}^{2r}n(a_i)=1\}.
$$
For each $a\in V$ with $n(a)\ne0$,  consider $\tau_a\colon V\to V$, $\tau_a(v)= v-\frac{2n(a,v)}{n(a)}a$ the reflection relative to the hyperplane orthogonal to $a$, which is an isometry with $\det\tau_a=-1$. It is easy to express this reflection in terms of the multiplication in the Clifford algebra: $\tau_a(v)=-ava^{-1}$, since $\tau_a(a)=-a$ and $\tau_a(v)=v$ for any $v$ orthogonal to $a$. (The element $a$ is invertible in $\Cl(V,n)$ since $a^2=n(a)1\in\RR$.) This permits us to define the group homomorphism
\begin{equation}\label{eq_doblerecubridor}
\spin(V,n)\to\SO(V,n),\quad \pm a_1\dots a_{2r}\mapsto   \tau_{a_1}\dots \tau_{a_{2r}},
\end{equation}
which is an epimorphism, since every orthogonal transformation is composition of reflections, by Dieudonné's theorem, and the determinant equal to 1 forces the number of reflections to be even.
 Besides the kernel is $\{\pm1\}\cong\ZZ_2$. 
 Indeed, if $ \tau_{a_1}\dots \tau_{a_{2r}}=\id_V$, then $a_1\dots a_{2r}$ commutes with $V$ and hence it belongs to the center of $\Cl(V,n)$. The center of the Clifford algebra depends on $m$ the dimension of $V$: If $m$ is even, the center is $\RR1$, but if $m$ is odd, the center is $\RR1\oplus\RR e_1\dots e_m$. In any case, as  $a_1\dots a_{2r}$ is in the center, then $r=0$.\smallskip

We will apply all the above to $V=\RR^7$ and $-n$ the opposite   of the norm related to any generic 3-form in $V$ (hence either $n$ is positive definite or $n$ has signature $(4,3)$).\footnote{The classical notation for $\Cl(V,n)$ is   $\Cl_{p,q}(\RR)$ if $(p,q)$ is the signature of $n$, and also $\Cl_{p}(\RR)$ if $q=0$.} Besides, let $\times$ be the related cross product and $\cC=\RR\oplus\RR^7$ the octonion algebra (either split or division) as in Eq.~\eqref{eq_procto}. (This change of notation makes easier to deal with both cases simultaneously.) We will  soon need some important properties of this (nonassociative) algebra, as for instance:
\begin{itemize}
\item[P1)] Alternativity: $x(xy)=(x^2)y$ and $(xy)x=x(yx)$ for all $x,y\in\cC$, in other words, the associator $(x,y,z)=(xy)z-x(yz)$ is alternating (obviously trilinear). 

It suffices to prove it if $x,y\in V$ (for $1$ it is evident). If $x$ and $y$ are orthogonal, $x(xy)=x\times(x\times y)=-n(x)y=(x^2)y$ by \eqref{eq_mainidentity}. And if $x=y$ (or scalar multiple), of course $(x,x,x)=-n(x)x+xn(x)1=0$.\smallskip

\item[P2)]  If $x=s+u$, $s\in\RR$, $u\in V$, we denote by $\bar x:=s-u$. Then $V=\{x\in\cC:\bar x=-x\}$ and $\tr(x)=2n(x,1)=x+\bar x$. Also, $x\bar x=s^2+n(u)=n(x)1$, so that every nonisotropic element is invertible. Besides, for all $x,y\in\cC$,
$$
\bar x\bar y=\overline{yx}.
$$

This is a straightforward computation taking into account the skew-symmetry of the cross product, since, for $x=s+u$, $y=t+v$ ($s,t\in\RR$, $u,v\in V$), $\bar x\bar y= st-n(u,v)1-tu-sv+u\times v  $ and $\overline{yx}  =st-n(v,u)1-\big(tu+sv+v\times u  \big)   $.\smallskip

\item[P3)]  First Moufang Identity: $(xax)y = x(a(xy))$ for all $x,y,a\in\cC$. 

This is true in any alternative algebra, since
$$
\begin{array}{rl}
(xax)y &-\ x(a(xy)) = (xa,x,y) + (x,a,xy) \stackrel{\textrm{P1}}{=} -(x, xa,y)-(x,xy,a) =\\
&= -(x^2a)y  + x((xa)y )- (x^2y)a+ x((xy)a) =\\
&= -(x^2, a, y) -x^2(ay) - (x^2,y,a) - x^2(ya) + x\big((xa)y + (xy)a\big) =\\
&= -x^2(ay+ya)+  x\big((xa)y + (xy)a\big)=\\
&= x\big( -x (ay+ya)+(xa)y + (xy)a \big)\stackrel{\textrm{P1}}{=}x\big((x,a,y)+(x,y,a)   \big)=0.
\end{array}
$$\smallskip
\item[P4)]  Second Moufang Identity: $(xy)(ax) = x(ya)x$ for all $x,y,a\in\cC$, since
$$
(xy)(ax) -x(ya)x=(x,y,ax)-x(y,a,x)\stackrel{\textrm{P1}}{=}-(x,ax,y)-x(y,a,x)
$$
but 
$$
-(x,ax,y)=x((ax)y)-(xax)y\stackrel{\textrm{P3}}{=}x\big( (ax)y-a(xy) \big)=x(a,x,y)\stackrel{\textrm{P1}}{=}x(y,a,x).
$$
\end{itemize}

Let $\rho\colon V\to \textrm{End}(\cC)$ be the map $\rho(u)=L_u$, for $L_u$ the left multiplication in $\cC$ given by $L_u(x)=ux$. The alternativity proved in (P1) provides $\rho(u)^2=L_{u^2}=L_{-n(u)1}=-n(u)\id$. As before, the universal property gives 
$\tilde\rho\colon\Cl(V,-n)\to \textrm{End}(\cC)$ a homomorphism of associative algebras and hence $\cC$ is an eight-dimensional representation of $\Cl(V,-n)$ (not faithful). The restriction to the even Clifford algebra 
\begin{equation}\label{def_spinors}
\tilde\rho\colon\Cl_{\bar0}(V,-n)\stackrel{\cong}{\longrightarrow} \textrm{End}(\cC)
\end{equation}
 is an isomorphism of associative algebras, since it is an injective linear map between vector spaces of the same dimension ($\dim\textrm{End}(\cC)=8^2=2^{7-1}$). The injectivity is a consequence of that $\ker(\tilde\rho)$ is an ideal of the simple algebra $\Cl_{\bar0}(V,-n)$ and the map is nonzero. Equation~\eqref{def_spinors} implies that $\cC$ is an irreducible $\Cl_{\bar0}(V,-n)$-representation. The spin group $\spin(V,-n)$ acts on $\cC$ via the Clifford algebra, 
 and this representation $\tilde\rho\vert_{\spin(V,-n)}$
 is called the \emph{spin representation}. Its elements (belonging to $\cC$, i.e. to $\RR^8$) are called \emph{spinors}. That is, if $x\in\cC$ and $g=\pm a_1\dots a_{2r}\in\spin(V,-n)$, for $a_i\in V,\prod_{i=1}^{2r}n(a_i)=1$, the spin action is
$$
g\cdot x=\tilde\rho(\pm a_1\dots a_{2r})(x)=\pm L_{a_1}\dots L_{a_{2r}}(x)=\pm a_1(a_2(\dots (a_{2r}x)\dots)).
$$
Let us check first that this is an action by isometries:\footnote{This is not consequence of \eqref{eq_doblerecubridor}.
 If $x\in V$, any element $g\in\spin(V,-n)$ acts also on $x$ by means of its image by  \eqref{eq_doblerecubridor} on the orthogonal group, but these are completely different actions: $  \tau_{a_1}\dots \tau_{a_{2r}}(x)\ne \pm a_1(a_2(\dots (a_{2r}x)\dots)).$  From a different viewpoint, $\spin(V,-n)$ acts on $V$ by  \eqref{eq_doblerecubridor}, but it cannot be considered a subgroup of $\GL(V)$. The least $m$ such that $\spin(V,-n)\subset \GL_m(\RR)$ is $m=8$.
 }  
 $\tilde\rho(\spin(V,-n))\subset \textrm{O}(\cC,n)$.
%
 If $x,y\in V$, recall that $L_x(y)=xy=-n(x,y)1+x\times y$, so that $n(L_x(y),z)=n(x\times y,z)=-n(y,x\times z)=-n(L_x(z),y)$ if $x,y,z\in V$. But $n(L_x(y),z)=-n(L_x(z),y)$ is also true for any $y,z\in \cC$: if $y=z=1$, then $n(L_x(y),z)=n(x,1)=0=-n(L_x(z),y)$; and if $y=1$, $z\in V$, then
$-n(L_x(z),y)=-n(xz,1)=-n(-n(x,z)1,1)=n(x,z)= n(L_x(y),z)$.

Consequently $\spin(V,-n)$ acts on the manifold $\{x\in\cC\mid n(x)=1\}$, which is the sphere $\mathbb{S}^7$ when $n$ is positive definite, and, following again O'Neill's notation, it is the pseudosphere $\mathbb{S}_4^7$ when $n$ is isotropic ($\mathbb{S}_4^7\equiv \mathbb{S}_4^7(1)$). 
Our goal is to check that this is a transitive action such that the isotropy subgroup of any element has type $G_2$. For the transitivity, we use the next auxiliary result.

\begin{lemma}
If $x\in\cC$, $n(x)=1$, then there is $a,b\in V$ such that $ab=x$.
\end{lemma}

\begin{proof}
Take $x=s1+u$ with $u\in V$ and $n(x)=1$. If $u=0$, then $x=\pm1$ and the result is clear: take any $v\in V$ with $n(v)\ne0$ and then $v(\pm\frac{\bar v}{n(v)})=x$. If $u\ne0$ but $n(u)=0$, take $v\in V$ such that $n(u,v)\ne0$. Then $\cQ=\langle 1,u,v,u\times v\rangle$ is a 4-dimensional subalgebra of $\cC$ such that $n\vert_{\cQ}$ is nondegenerate (usually called a \emph{quaternion} subalgebra). Indeed, taking in mind \eqref{eq_mainidentity}, $u(u\times v)=n(u,v)u$, $uv=-n(u,v)1+u\times v$ and $v(u\times v)=-n(v,u)v+n(v)u$, all of them belonging to $\cQ$. If, on the contrary,   $n(u)\ne0$, take $v\in V$ orthogonal to $u$ but nonisotropic and again $\cQ=\langle 1,u,v,u\times v\rangle$ is a quaternion subalgebra of $\cC$
(now $uv=u\times v$, $u(u\times v)=-n(u)v$ and $v(u\times v)=n(v)u$).

In both cases, nondegeneracy gives $\cC=\cQ\oplus\cQ^\perp$ and there is $w\in\cQ^\perp$ with $n(w)\ne0$. Note that $x\cQ^\perp\subset \cQ^\perp$, since for any $y\in \cQ^\perp$ and $z\in\cQ$, 
$$
n(xy,z)
=sn(y,z)-n(y,uz)=n(y,\bar xz)\in n(\cQ^\perp,\cQ\cQ)\subset n(\cQ^\perp,\cQ)=0.
$$
As $x$ is invertible (P2), then $x\cQ^\perp=\cQ^\perp$ and there is $b\in \cQ^\perp$ with $w=xb$.
Hence $n(b)=n(w)\ne0$ and $x=(xb)\left(\frac{\bar b}{n(b)}\right)$   (P1) is a product of two elements in $  \cQ^\perp\subset  1^\perp=V$.
\end{proof}

With the notations above,

\begin{proposition}\label{pr_esferas7}
The group $\spin(V,-n)$ acts transitively on $\{x\in\cC\mid n(x)=1\}$.
The isotropy subgroup  $\{ g\in\spin(V,-n):g\cdot 1=1 \}$  is isomorphic to $\Aut(\cC)$.
\end{proposition}



\begin{proof}
Let $x\in\cC$ be an element with $n(x)=1$ and take $a,b\in V$ with $ab=x$ as in the previous lemma. Now  $ab\in\spin(V,-n)$ (take care with the abuse of notation, here $ab$ denotes the product of elements of $V$ in the Clifford algebra, while in the above line the product $ab$ was the octonionic product)  since $n(a)n(b)=n(x)=1$, and  besides $ab\cdot 1=L_aL_b(1)=ab=x$. Hence the orbit of the element $1\in\cC$ is the whole $\{x\in\cC:n(x)=1\}$. This finishes the transitivity of the action.\smallskip
 
Now take $H=\{g\in\spin(V,-n):g\cdot1=1\}$ the isotropy subgroup of $1\in\cC$.\footnote{Although $1$ is a remarkable element in the octonion algebra $\cC$, it is not a remarkable unit spinor, all of them are undistinguishable as a consequence of the transitivity.} We are going to prove that $\tilde\rho(H)$ (isomorphic to $H$) coincides with $\Aut(\cC)$. 
Consider a triple product in $\cC$ given by
\begin{equation}\label{triple}
\langle\ ,\ ,\ \rangle\colon \cC\times\cC\times\cC\to\cC,\qquad \langle x,y,z\rangle:=(x\bar y)z.
\end{equation}
(And observe that the octonionic product is recovered in some way from this triple product since $xy=\langle x,1,y\rangle$.)
Then, for any $x,y,z\in\cC$, $a\in V$,
$$
\begin{array}{ll}
\langle L_a x,L_a y,L_a z\rangle&=((ax)(\overline{ay}))(az)\stackrel{\textrm{P2}}=((ax)(\bar y\bar a))(az)\\
&=-((ax)(\bar y  a))(az)\stackrel{\textrm{P4}}=-(a(x\bar y)  a)(az)\stackrel{\textrm{P3}}=-a((x\bar y)  (a(az)))\\
& \stackrel{\textrm{P1}}=n(a)a((x\bar y)z)=n(a)L_a\langle x,y,z\rangle.
\end{array}
$$
Hence any $g\in\spin(V,-n)$ preserves the triple product: there are 
$a_1,\dots, a_{2r}\in V$ with $\prod_{i=1}^{2r}n(a_i)=1$ such that $g\cdot x=\pm L_{a_1}\dots  L_{a_{2r}}(x)$, and then
$$
\begin{array}{ll}
\langle g\cdot x,g\cdot y,g\cdot z\rangle&=(\pm1)^3n(a_1)L_{a_1}\dots n(a_{2r})L_{a_{2r}}\langle x,y,z\rangle\\
&=\pm L_{a_1}\dots  L_{a_{2r}}\langle x,y,z\rangle=g\cdot \langle x,y,z\rangle.
\end{array}
$$
Now, if $g\in\spin(V,-n)$ fixes the element $1\in\cC$ (consider $g$ as an endomorphism of $\cC$ by means of $\tilde\rho$, that is, $g(x):=\tilde\rho(g)(x)= g\cdot x$), then
$$
g(xy)=g(\langle x,1,y\rangle)=\langle g\cdot x,g\cdot 1,g\cdot y\rangle=\langle g\cdot x,1,g\cdot y\rangle=g(x) g(y),
$$
and $g\in\Aut(\cC)$. This proves $\tilde\rho(H)\subset \Aut(\cC)$. Besides 
$$
\begin{array}{ll}\dim H&=\dim\spin(V,-n)-\dim (\spin(V,-n)\cdot 1)\\&=
\dim\SO(V,-n)-\dim\{x\in\cC:n(x)=1\}=\left(\begin{array}{l}7\\2\end{array}\right)-7=14,  
\end{array}
$$ 
what implies that the connected components of $\tilde\rho(H)$ and of $\Aut(\cC)$ coincide. As we proved  that $\Aut(\cC)$ is connected (see Lemma~\ref{pr_conreal}), we conclude that  $\tilde\rho(H)= \Aut(\cC)$.
\end{proof}

This gives immediately,

\begin{corollary}\label{co_esferas7}
Consider $\mathcal C=\RR^8$ endowed with a norm $n$ either positive definite or of signature $(4,4)$, take 
$x\in\cC$ with $n(x)=1$, and $V=\langle x\rangle^\perp\cong\RR^7$.
Consider the spin representation of $\spin(V,-n)$ on $\cC$ given by the restriction of \eqref{def_spinors}.
Then
the stabilizer  of the unit spinor $x$,
$$
\{ g\in\spin(V,-n):g\cdot x=x \}\cong\left\{ \begin{array}{ll}
G_{2,-14}&\text{if $n$ is definite,}\\
G_{2,2}&\text{otherwise,}
\end{array}  \right.
$$
and hence
\begin{equation}\label{eq_spincocienteG2}
\spin(7)/G_{2,-14}\cong \mathbb{S}^7,\quad\qquad\spin(3,4)/G_{2,2}\cong \mathbb{S}_4^7.
\end{equation}
\end{corollary}

Observe that the fundamental groups of the real Lie groups of type $G_2$
computed in Proposition~\ref{co_simpleconexion} 
  agree with the ones we would have obtained by looking at the coset manifolds above, 
 since 
$$\pi_1(\spin(7))=\pi_1(G_{2,-14})=0,\qquad \pi_1(\spin(3,4))=\pi_1(G_{2,2})=\ZZ_2  .$$
(Recall that the spin groups are nonnecessarily simply connected if the signature is not definite.)\smallskip


Both homogeneous spaces in Eq.~\eqref{eq_spincocienteG2} are naturally reductive homogeneous spaces. The algebraical structure produced in $(\mathfrak m,\frac12[\ ,\ ]_{\mathfrak m})$ is isomorphic to $(V,\times)$ for $V=\RR^7$ and $\times$ the related cross product as in the above sections. We can observe that this is a simple Malcev (non-Lie) algebra. Many geometrical properties about torsion and curvature on $\mathbb{S}^7$ and $\mathbb{S}_4^7$ can be concluded from this algebraical fact. A more detailed study of these reductive homogeneous spaces, jointly with 
the determination of their invariant affine connections, can be found in 
\cite{s7Alb}.\smallskip

We can also conclude, taking into account Propositions~\ref{le_elsigma} and \ref{le_comoelsigma}, besides Corollary~\ref{co_esferas7}, that
\begin{corollary}
\
\begin{itemize} 
\item[i)] If $n$ is positive definite and $\varphi_1,\varphi_2 \in\cC$ are two orthogonal unit spinors,
$$ 
\{g\in\spin(V,-n):g\cdot \varphi_1=\varphi_1,g\cdot \varphi_2=\varphi_2\}\cong\mathrm{SU}(3) .
$$
Thus, $\spin(7)/\mathrm{SU}(3)\cong V_{2}(\RR^8)$ is the Stiefel manifold of orthogonal pairs of unital elements in $\RR^8$.
\item[ii)]  If $n$ has signature $(4,3)$
and $\varphi_1,\varphi_2\in\cC$ are two orthogonal   spinors with $n(\varphi_1)=1=-n(\varphi_2)$, then 
$$\{g\in\spin(V,-n):g\cdot \varphi_1=\varphi_1,g\cdot \varphi_2=\varphi_2\}\cong\SL_3(\RR).$$ 
Thus,
$$\frac{\spin(3,4)}{\mathrm{SL}_3(\RR)}\cong V_{1,1}(\RR^{4,4}):=\{a\in\Mat_{8,2}(\RR):a^t  \tiny\left(\begin{array}{cc}I_4&0\\0&-I_4\end{array}\right)a=\tiny\left(\begin{array}{cc}1&0\\0&-1\end{array}\right)\}.$$ 
\end{itemize}    
\smallskip 
\noindent Furthermore, if we consider a third spinor, 
the stabilizer of the three spinors  will be either $\mathrm{SU}(2)$ or $\mathrm{SL}_2(\RR)$.\smallskip
\end{corollary}

%
%
%

\begin{remark}\label{re_modeloIlka}
When $n$ is positive definite, we are going to see all the above in a more concrete way adapted from \cite[\S2.1]{agrilecturenotes} to try to relate the different approaches and models. If $\{e_i\}_{i=1}^7$ is the canonical basis of the subspace of the zero trace elements in $\OO$ as in Figure~\ref{fig_fano}, and we denote by $\kappa_i:=\rho(e_i)=L_{e_i}$, then we have that $\{\kappa_i\}_{i=1}^7\subset\frso(\OO,n)$ is a family of skew-symmetric operators such that 
$$ 
\kappa_i^2=-\id,\qquad  \kappa_i\kappa_j=-\kappa_j\kappa_i \ \textrm{ if   $i\ne j$.}
$$
If $e_8$ denotes the identity element in $\OO$ and $\varphi_{u,v}:=n(u,-)v-n(v,-)u\in\frso(\OO,n)$, then  $\varphi_{ij}:=\varphi_{e_i,e_j}$ sends $e_i$ to $e_j$ and $e_j$ to $-e_i$,   
 and
$$
\begin{array}{ll}
\kappa_1=-\varphi_{18}+\varphi_{23}+\varphi_{47}-\varphi_{56},&\kappa_{5}=\varphi_{ 16}-\varphi_{27 }-\varphi_{34 }-\varphi_{58 } ,\\
\kappa_2=-\varphi_{13}-\varphi_{28}+\varphi_{46}+\varphi_{57}, &\kappa_{6}=-\varphi_{ 15}+\varphi_{24 }-\varphi_{37 }-\varphi_{68 },\\
\kappa_{3}=\varphi_{12 }-\varphi_{38 }-\varphi_{45 }+\varphi_{67 },&\kappa_{7}=\varphi_{ 14}+\varphi_{25 }+\varphi_{36 }-\varphi_{78 }.\\
\kappa_{4}=-\varphi_{17 }-\varphi_{26 }+\varphi_{35 }-\varphi_{48 },&
\end{array}
$$
Note that the map $\frso(V,-n)\to\Cl(V,-n)^-$ given by  $\varphi_{u,v}\mapsto -\frac12[u,v]$ is a Lie algebra monomorphism.\footnote{If 
$\mathcal A$ is an associative algebra, then $\mathcal A^-$ denotes the Lie algebra defined over the same vector space by the Lie bracket given by the commutator $[a,b]=ab-ba$.
}
Thus we can consider $\frso(V,-n)$ living in $\Cl(V,-n)^-$, and $\mathfrak{spin}(V,-n)$ the related copy of $\frso(V,-n)$ contained in $\tilde\rho(\Cl_{\bar0}(V,-n))=\frgl(\OO)$. That is, 
  $\mathfrak{spin}(V,-n) =[\rho(V),\rho(V)]$, which    is then spanned by $\{\kappa_i\kappa_j=\frac12[\kappa_i,\kappa_j]:1\le i< j\le7\}$. 
  
  As we have shown that   $\Aut(\OO)=\tilde\rho\left(\{g\in\spin(V,-n):g\cdot e_8=e_8\}\right)
 $, then 
  $$   \{d\in\mathfrak{spin}(V,-n):d(e_8)=0\}=\Der(\OO)\cong \Der(V,\Omega_1)=\frg_c$$
provides an isomorphism by means of the restriction map $d\mapsto d\vert_V$. Observe that an arbitrary element in
$\mathfrak{spin}(V,-n)$ does not act on $V$ (but on $\OO$), but it does if it annihilates  $\langle e_8\rangle=V^\perp$.
 
  Thus $\frg_c$ can be identified with the set of   elements $d=\sum a_{ij}\kappa_i\kappa_j\in\mathfrak{spin}(V,-n)$ with $d(u)=0$ ($u\in\OO$ any nonzero fixed element), what provides some linear equations in the coefficients $a_{ij}$. To be precise, if we choose $u=e_8$, 
then $0=d(u)=-\sum_{i<j} a_{ij}\kappa_i(e_j)\in V$, whose coordinates relative to the basis $\{e_i\}_{i=1}^7$ are zero, that is,
\begin{equation}\label{eq_ecddeg2} 
\frg_c=\left\{\begin{array}{lcr} &&a_{14 }+ a_{ 25}+ a_{36 }=0\\\sum a_{ij}\kappa_i\kappa_j\in\mathfrak{spin}(V,-n):
 &a_{23 }+ a_{47 }- a_{56 }=0
 &a_{13 }- a_{ 46}- a_{57 }=0\\ 
 &a_{ 12}- a_{45 }+ a_{67 }=0
 & a_{17 }+ a_{26 }- a_{35 }=0\\
 &a_{ 16}- a_{27 }- a_{34 }= 0
 & a_{15 }- a_{24 }+ a_{37 }=0\end{array}\right\}.
\end{equation}
Observe that  the matrix of $\varphi_{i,j}$ relative to the above basis is $E_{ji}-E_{ij}$
so that  $\sum_{i<j} a_{ij}\kappa_i\kappa_j\mapsto -(a_{ij})$ with $a_{ji}=-a_{ij}$ is an isomorphism of $\mathfrak{spin}(V,-n)$ onto the Lie algebra $\frso(7)$ of the skewsymmetric matrices of size 7.

According to the previous sections, the subalgebra 
  $   \mathfrak{h}= \{d\in\mathfrak{spin}(V,-n):d(e_8)=0=d(e_7)\}$ should be isomorphic to $\frsu(W,\sigma).$
  In order to understand it   directly, we compute   
  $$    \mathfrak{h}= \left\{\begin{array}{lll}
  \sum a_{ij}\kappa_i\kappa_j: &a_{i7}=0\,\forall i &  a_{14 }+ a_{ 25}+ a_{36 }=0\\
  &a_{26}=a_{35} &a_{16}=a_{34}\quad
  a_{15}=a_{24} \\
  &a_{23}=a_{56} &a_{12}=a_{45}\quad a_{13}=a_{46}
 \end{array}\right\},
  $$
  so that the precise isomorphism $\psi\colon\mathfrak{h}\to \frsu(3)=\{c\in \Mat_3(\CC):c+\bar{c}^t=0\}$ is given by
  \begin{equation}\label{eq_iso1}
  \psi \left( \begin{array}{ccc} a&b&0\\-b^t&a&0\\0&0&0   \end{array} \right)= a+\mathbf{i} b
   \end{equation}
   for $a,b\in \Mat_3(\RR), a=-a^t, b=b^t,\tr(b)=0$.\smallskip
   
  The choice of an $\mathfrak{h}$-invariant complementary subspace should allow us to recover the compact model in  Proposition~\ref{pr_modelocompacto}.
  We denote by
  $$
   \mu_{x,y}:=\left( \begin{array}{ccc}  l_y&l_x&2x\\l_x&-l_y&2y\\-2x^t&-2y^t&0   \end{array} \right)
  $$
  where $l_x$ is the coordinate matrix of the cross product in $\RR^3$ as in Eq.~\eqref{loslx}. If $\mathfrak{m}:=\left\{ \mu_{x,y} :x,y\in\RR^3 \right\}$, clearly it is the required subspace, since  $\mathfrak{g}_c= \mathfrak{h}\oplus\mathfrak{m}$ 
  and  
  $$
  \left[\left( \begin{array}{ccc} a&b&0\\-b^t&a&0\\0&0&0   \end{array} \right),\mu_{x,y}\right]=\mu_{ax+by,ay-bx}.
  $$
 Thus we can extend the Lie algebra isomorphism \eqref{eq_iso1} to an isomorphism from 
 $$\psi\colon\mathfrak{g}_c\to \mathcal L=\frsu(W)\oplus W,\qquad\psi(\mu_{x,y})= -y-\mathbf{i}x\in\CC^3=W,
 $$
  because
  $$
  \begin{array}{l}
  {[}\mu_{x,y},\mu_{u,v}]_\mathfrak{m}=2\mu_{\,y\times u+x\times v,\,x\times u-y\times v},\vspace{8pt}\\
   {[}\mu_{x,y},\mu_{u,v}]_\mathfrak{h}=3\psi^{-1}\big( l_{y\times v+x\times u}+\mathbf{i}\textrm{pr}_0(uy^t+yu^t-vx^t-xv^t) \big),
   \end{array}
  $$
  where $\textrm{pr}_0(b)=b-\frac13\tr(b)I_3\in\frsl_3(\RR)$ denotes the projection on the special linear algebra. (Recall $\tr(uy^t)=u^ty$.)
  Furthermore, we can translate the natural action of $\frg_c\subset\frso(7)$ on the irreducible real representation $\RR^7$ given by the usual matrix product. For $\psi'\colon\RR^7\to\mathcal{V}=\RR\oplus \CC^3$ given by 
  $$\psi'\left( \begin{array}{ccc} x\\y\\s   \end{array} \right)=s-x+\mathbf{i}y,$$
 one easily checks   that  $\psi(d)\cdot\psi'(X)=\psi'(dX)$, which provides an alternative proof of Proposition~\ref{pr_modelocompacto}.
  \smallskip

  Note that if we repeat all the above for isotropic $n$, we would get a model analogous to the split one in Proposition~\ref{pr_defdeL} (with blocks $3+3+1$ now). 
 \end{remark}

\begin{remark}
All these points of view are, of course, interconnected. If $n$ is positive definite, any choice of spinor provides a 3-form in such a way that the group fixing the spinor coincides with the group fixing the related 3-form by the usual pullback. 

Indeed, take $0\ne u\in\cC$. We have $\cC=\RR u\oplus \langle u\rangle^\perp$, but, by means of the action of $\Cl(V,-n)$, $1\cdot u=u$ and $\rho(V)(u)\subset \langle u\rangle^\perp$ (recall $n(L_vu,u)=0$ since $L_v$ is skew-adjoint). As any (nonzero) element in $V$ is invertible in the Clifford algebra, $V\cdot u=\rho(V)(u)= \langle u\rangle^\perp$  by dimension count. Take $A_u\colon V\times V\to V$ the multiplication defined by
$$
A_u(y,x)\cdot u=y\cdot (x\cdot u)+n(x,y)u.
$$
(Since $n(y\cdot (x\cdot u)+n(x,y)u,u)=0$, it follows that there is exactly one element $v\in V$ such that $v\cdot u=y\cdot (x\cdot u)+n(x,y)u$.) This provides a skew-symmetric $(1,2)$-tensor as $y\cdot (x\cdot u)+x\cdot (y\cdot u)=\tilde\rho(y  x+x  y)(u)=\tilde\rho(-2n(x,y)1)(u)=-2n(x,y)u$ and so $A_u(y,x)=-A_u(x,y)$. It is not difficult to check also that
$A_u(x,y)$ is orthogonal to $x$ and $y$ in $V$, and that
$n(A_u(x,y))=n(x)n(y)-n(x,y)^2$, by using that $n(x  u,y  u)=n(x,y)n(u)$ in $\cC$. Hence $A_u$ is a cross product for $n$ and 
the map $\omega_u\colon V\times V\times V\to \RR$   given by 
$$
\omega_u(x,y,z)=n(x,A_u(y,z))
$$ 
is the required 3-form.

If $n$ is split we need $u\in\cC$ nonisotropic to assure first that $\cC=\RR u\oplus \langle u\rangle^\perp$. Then all works similarly to the definite case since $V\cdot u=\langle u\rangle^\perp$ (simply take into account that we can find a basis of $V$ formed by nonisotropic vectors), and then $\omega_u$ defined as before is a 3-form such that $g^*\omega_u=\omega_u$ if $g\in\spin(V,-n)$ with $g\cdot u=u$.

The arguments become closed since we can define in the set of spinors $\RR^8=\RR u\oplus \langle u\rangle^\perp$ a product making it an octonion algebra such that $u$ is its identity element, simply by using the cross product related to the 3-form $\omega_u$. \smallskip

Observe that the spinorial representation $\cC$ is $\spin(V,-n)$-irreducible, but not $G_2$-irreducible: $\cC=\RR u\oplus \langle u\rangle^\perp$ is the decomposition as a sum of two irreducible $G_2$-modules. In fact, $\langle u\rangle^\perp$ is the nontrivial representation for $G_2$ of the least possible dimension, 7.

\end{remark}

\begin{remark}
The triple product described in Eq.~\eqref{triple} is closely related to   generic 3-forms. 
(Again, it is not important if $\FF$ is either $\RR$ or $\CC$, $V=\FF^7$ and $\cC=\FF\oplus V$ is an octonion algebra.)
First, note that
$$
\Lambda\colon V\times V\times V\times V\to\FF, \quad \Lambda(x,y,z,u)=\frac12n(x,\langle y,z,u\rangle -\langle u,z,y\rangle )
$$
is multilinear and alternating, that is, a 4-form\footnote{ Note that $\langle y,z,u\rangle -\langle u,z,y\rangle=\Omega(y,z,u)-(y,z,u)$ in $\cC$, so that $\Lambda(x,y,z,u)=-\frac12n(x,(y,z,u))$ since $x\in V$.  In the compact case,  $\Lambda$ is called the \emph{coassociative 4-form} for $\OO$ in \cite[6]{Harvey}.   }: 
If one of the last  three  entries is repeated, then $\langle y,z,u\rangle -\langle u,z,y\rangle =0$ is an immediate consequence of the alternativity (P1).
Also, $2\Lambda(x,x,y,z)=-n(x,(xy)z)+n(x,(zy)x)=n(xz,xy)-n(x^2,zy)=n(x)n(z,y)-x^2n(1,zy)=0$ since $x^2=-n(x)1$ and $L_y$ is skew-adjoint for the norm. 

Of course the group $G_2=\Aut(\cC)$ preserves this 4-form because $G_2$ preserves both the norm and the triple product.  Furthermore, $G_2$ can  also be characterized as the group preserving the 4-form.\footnote{
To be more precise, this happens in both the real cases. In the complex case, the group preserving the 4-form is not $G_{\Omega_0}$ but $G_{\Omega_0}^+\times\{\mathbf{i}^s\id_V:s=0,\dots,3\}$ (their identity components coincide), since the maps of determinant 1 preserve the volume form and hence $\Omega_0\wedge \Lambda$.  
}
 To this purpose, simply observe that for our generic $G_2$-invariant 3-form
$$
\Omega\colon   V\times V\times V\to\FF, \quad \Omega(x,y,z)=n(xy,z),
$$
we have $0\ne \Omega\wedge\Lambda\in\wedge^7V^*\cong\FF$. We will check this fact in the definite case (hence, it will be also true in the complex case, and so in the real-split):   It is a straightforward computation, using the products in Figure~\ref{fig_fano}, that if $0\ne(e_i,e_j,e_k)$ (if these three elements do not associate, they are not in the same line in the Fano plane, hence), then there is $l$ a different index such that $\langle e_i,e_j,e_k\rangle -\langle e_k,e_j,e_i\rangle=\pm2e_l$, and we compute
$$
\Lambda=-e^{1245}+e^{1267 }-e^{ 1346}-e^{1357 }+e^{ 2347}-e^{2356 }-e^{ 4567},
$$
with the notation used in Eq.~\eqref{eq_omegandefinida}. Thus,
$$
\Omega_1\wedge \Lambda=-7 e^{1234567}\ne0.
$$
(With a different notation, this means that the Hodge star operator $\star\Omega_1=\Lambda$, up to a scalar depending on the choice of a fixed nonzero 7-form.)
\end{remark}

\begin{remark}
I.~Agricola pointed me out that it is not difficult to characterize the three real and complex algebras of type $G_2$ in the language of $p$-forms. We will do it in the compact case. We identified $\frso(V,-n)$ with $\wedge^2V^*$ as vector spaces by means of
$
\varphi_{i,j}\mapsto e^{ij},
$
which allows to see $\frg_c\subset \wedge^2V^*$ as in Remark~\ref{re_modeloIlka} and also to have a Lie bracket in $\wedge^2V^*$.\footnote{The Lie bracket in $\wedge^2V^*$ becomes  $[\alpha_1\wedge\beta_1, \alpha_2\wedge\beta_2]= n(\alpha_1,\alpha_2) \beta_1\wedge\beta_2
-n( \beta_1,\alpha_2)\alpha_1\wedge\beta_2
-n(\alpha_1,\beta_2) \beta_1\wedge\alpha_2
+n( \beta_1,\beta_2)\alpha_1\wedge\alpha_2 $, for $\alpha_i,\beta_i$ 1-forms.} Then,
$$
\frg_c=\{\alpha\in \wedge^2V^*:\star(\Omega_1\wedge\alpha)=\alpha\},
$$
and $\wedge^2V^*=\frg_c\oplus \mathfrak{m}$ for the $\frg_c$-invariant complement
\begin{equation}\label{eq_comp_m}
\mathfrak{m}=\{u\intprodl\Omega_1:u\in V\}\le \wedge^2V^*,
\end{equation}  
which can be also characterized as
\begin{equation}\label{eq_caracV}
\mathfrak{m}=\{\alpha\in \wedge^2V^*:\star(\Omega_1\wedge\alpha)=-2\alpha\}.
\end{equation}
We are now trying to provide justifications of the above facts.
First note that $F$ is diagonalizable with real eigenvalues, since it is a symmetric endomorphism relative to the scalar product $\langle\ ,\ \rangle $  in $\wedge^2V^*$ induced by $n$. Indeed, for $\alpha,\beta\in \wedge^2V^*$,
$$\langle F(\alpha),\beta\rangle e^{1234567}=\Omega_1\wedge \alpha\wedge \beta=
\Omega_1\wedge \beta\wedge \alpha=\langle F(\beta),\alpha\rangle e^{1234567}.
$$
Besides the sum of the eigenvalues is zero, since its matrix relative to the basis $\{e^{ij}:1\le i<j\le7\}$ has all the entries in the diagonal equal to zero and in particular zero trace. 
For easy computations, one can use the coassociative 4-form, as $F(\alpha)=-\alpha\intprodl\Lambda$.\smallskip

All this is better understood with the help of a convenient grading of $\wedge^2V^*$ compatible with $F$.
Take the $\ZZ_2^3$-grading on the octonion algebra $\OO$ given by
$$
\OO_{(\bar 1,\bar 0,\bar 0)}=\langle e_1\rangle, \qquad
\OO_{(\bar 0,\bar 1,\bar 0)}=\langle e_2\rangle, \qquad
\OO_{(\bar 0,\bar 0,\bar 1)}=\langle e_7\rangle.
$$
That is, the basic elements $e_i$ are all of them homogeneous, with degrees
$\deg(e_3)=(\bar 1,\bar 1,\bar 0)$, $\deg(e_4)=(\bar 1,\bar 0,\bar 1)$, $\deg(e_5)=(\bar 0,\bar 1,\bar 1)$,
$\deg(e_6)=(\bar 1,\bar 1,\bar 1)$ and $\deg(1)=(\bar 0,\bar 0,\bar 0)$.\footnote{
This remarkable grading exists   in the split case too, where every nonzero homogeneous element is invertible. That is why, although $\OO_s$ is not a division algebra, it is a graded division algebra. In particular, this allows to construct both real octonion algebras as twisted group algebras for the group $\ZZ_2^3$.
} Call $g_i=\deg e_i$.
Any partition of a subspace $\cC$ in pieces indexed by a group $\cC=\oplus_{g\in G} \cC_g$ produces a group-grading on $L=\frgl(\cC)$ by means of 
$L=\oplus_{g\in G} L_g$ with $L_g=\{f\in\frgl(\cC):f(\cC_h)\subset \cC_{gh}\}$, i.e., $[L_g,L_h]\subset L_{gh}$. This is our case for $G=\ZZ_2^3$. The subalgebra $\mathfrak{spin}(V,-n)$ inherits this grading ($\mathfrak{spin}(V,-n)_g=\mathfrak{spin}(V,-n)\cap L_g$),
and the homogeneous components are
$$\begin{array}{ll}
W_i:=(\mathfrak{spin}(V,-n))_{g_i}&=\{x\in\mathfrak{spin}(V,-n): x(\cC_g)\subset \cC_{gg_i}\,\forall g\}\\
&=\langle \{ e_j^*\wedge(e_ie_j)^*:f\ne i \}\rangle\\
&\subset\{\alpha\in \wedge^2V^*:\alpha\cdot e_8\subset\langle e_i\rangle\}.
\end{array}
$$
For instance $W_4=\langle e^{17},e^{26},e^{35}  \rangle$, and in all the cases $W_i$ is spanned by the $2$-forms $e^{jk}$ such that $e_i$, $e_j$ and $e_k$ associate ($e_ie_j=\pm e_k$).
Besides, as the linear map $W_i\to\langle e_i\rangle$, $\alpha\mapsto \alpha\cdot e_8$ is not   zero, its kernel is a plane 
$$W_i'=\{\alpha\in W_i:\alpha\cdot e_8=0\}\le\frg_c.$$

The $\ZZ_2^3$-grading $\wedge^2V^*=\oplus_{i=1}^7 W_i$ splits our Lie algebra of type $B_3$ into seven Cartan subalgebras (the point is that, when we multiply two of them, we get another Cartan subalgebra!), as well as  happens with the  $\ZZ_2^3$-grading $\frg_c=\oplus_{i=1}^7 W_i'$.\footnote{When one works with the analogous $\ZZ_2^3$-grading on the split algebra $\frg_{2,2}$, any basis formed by homogeneous elements  satisfies that all its elements are semisimple. The bases with this property are important for Physics, but usually they are    not easy to find  in the split algebras, since all the elements in the root spaces are nilpotent. There are models of both the real forms of $\frg_2$ based on this $\ZZ_2^3$-grading. Group gradings on the complex algebra $\frg_2$ are studied in \cite{g2}, and on its real forms in \cite{reales}.} 
Now, the endomorphism $F$ commutes with the three order 2 automorphisms of $\wedge^2V^*$ producing the grading, so that $W_i$ are $F$-invariant subspaces. 
In fact, all of them are undistinguishable. The matrix of $F\vert_{W_4}$ relative to $B_{W_4}=\{ e^{17},e^{26},e^{35}  \}$
is
$$
  \left(\begin{array}{ccc}0&-1&1\\-1&0&1\\1&1&0\end{array}\right),
$$
so that the kernel $\ker(F\vert_{W_4}-\id)$ coincides with
$$
W_4'=\{a_{17}e^{17}+a_{26}e^{26}+a_{35}e^{35}:a_{17}+a_{26}-a_{35}=0\}.
$$
Similarly the equations defined by $\ker(F\vert_{W_i}-\id)$ for $i=1,\dots,7$  are equal to the seven equations in \eqref{eq_ecddeg2}. 

Now, the third eigenvalue of $F\vert_{W_4}$ has to be $-1-1=-2$, whose  eigenspace is the  {orthogonal} one to $W_4'$, i.e.
$\ker(F\vert_{W_4}+2\id)=\langle e^{17}+e^{26}-e^{35}\rangle=\langle e_4\intprodl\Omega_1 \rangle $.
A more conceptual argument for proving the characterization of the spin representation is to use that 
$$
\Omega_1\wedge(u\intprodl\Omega_1)\wedge(v\intprodl\Omega_1)=6n(u,v)e^{1234567},
$$
for any $u,v\in V$.
Thus, $\langle F(u\intprodl\Omega_1),v\intprodl\Omega_1\rangle=6n(u,v)=2\langle  u\intprodl\Omega_1 ,v\intprodl\Omega_1\rangle$, which proves \eqref{eq_caracV}.

In particular $\frg_c\cap \mathfrak{m}=0$. In order to prove \eqref{eq_comp_m}, we
check $[W_i',e_j\intprodl\Omega_1]\subset \langle (e_ie_j)\intprodl\Omega_1\rangle$. Again, for $\alpha=a_{17}e^{17}+a_{26}e^{26}+a_{35}e^{35}$ an arbitrary 2-form in $W_4$,
$$
[\alpha, e_1\intprodl\Omega_1]=(-a_{26}+a_{35}) e_7\intprodl\Omega_1 +(a_{17}+a_{26}-a_{35})e^{14},
$$
which belongs to $\mathfrak{m}$ if (and only if) $ \alpha\in\frg_c$. This computation also gives   the algebraic structure provided by the reductive decomposition\footnote{For a more algebraical approach, this reductive decomposition is $\wedge^2V^*=\frso(\OO_0,n),\ \frg_c=\Der(\OO),\ \mathfrak{m}=\ad(\OO_0)$.}, since
$$
[e_4\intprodl\Omega_1,e_1\intprodl\Omega_1]_{\mathfrak{m}}=-e_7\intprodl\Omega_1,\qquad
[e_4\intprodl\Omega_1,e_1\intprodl\Omega_1]_{\mathfrak{g}_c}=2e^{14}-e^{25}-e^{36},
$$
so that $[\ ,\ ]_\mathfrak{m}=\frac12\times$ under the natural correspondence between  $\mathfrak{m}$ and $V$.


%
%

\end{remark}

\bigskip
 
Finally,   some textbooks to achieve a good background for our study are:
\begin{itemize}
\item[{$\davidsstar$}] More about octonions in Okubo's book \cite{libroOkubo};
\item[{$\davidsstar$}] About exceptional algebras   \cite{Jacobsondeexcepcionales,Schafer};
\item[{$\davidsstar$}] About exceptional groups   \cite{Adams,Rosenfeld}.
\end{itemize}
 \medskip

\textbf{Acknowledgments}

To my colleague and always master Alberto Elduque, for helping me with some technical points in the proof of Proposition~\ref{pr_uncrossproduct} and for   his availability; to Ilka Agricola, for reading and improving several  versions of this manuscript and for encouraging me to write it and to make it accessible; and to Skip Garibaldi, for his nice suggestions and references.

\end{document}